\documentclass[reqno]{amsart}

\usepackage{hyperref}
\usepackage{amscd, amssymb , amsmath,epsfig,subfig, fullpage}
\usepackage{ braket, dsfont}
\usepackage{cancel}
\usepackage{mathrsfs}
\usepackage{epsfig}
\usepackage{color}
 \usepackage{graphicx}
\usepackage{leftidx}
 \usepackage{tikz-cd}
 \usepackage{psfrag}
 
\definecolor{brightlavender}{rgb}{0.75, 0.58, 0.89}
\definecolor{carnelian}{rgb}{0.7, 0.11, 0.11}

\usepackage[all]{xy}
\usepackage{verbatim}
\usepackage{stmaryrd}
\usepackage[mathscr]{euscript}
\usepackage{psfrag}

\newcommand{\trait}{\nobreakdash-\hspace{0pt}}

\newcommand{\Proj}{\operatorname{Proj}}

\newcommand{\tto}{\xrightarrow}
\newcommand{\chr}{{\mathsf{c}}}
\newcommand{\ptr}{\operatorname{ptr}}
\newcommand{\tr}{\operatorname{tr}}

\newtheorem{Df}{Definition}[section]

\newtheorem{theorem}[Df]{Theorem}

\newtheorem{lemma}[Df]{Lemma}

\newtheorem{corollary}[Df]{Corollary}

\theoremstyle{definition}
\newtheorem{remark}[Df]{Remark}
\newtheorem{example}[Df]{Example}

\newcommand{\rcoev}{\stackrel{\longrightarrow}{\operatorname{coev}}}
\newcommand{\rev}{\stackrel{\longrightarrow}{\operatorname{ev}}\!\!}
\newcommand{\lev}{\stackrel{\longleftarrow}{\operatorname{ev}}\!\!}
\newcommand{\lcoev}{\stackrel{\longleftarrow}{\operatorname{coev}}}

\newcommand{\cat}{\mathcal{C}}

\newcommand{\C}{\ensuremath{\mathbb{C}}}

\newcommand{\R}{\ensuremath{\mathbb{R}}}

\newcommand{\End}{\operatorname{End}}
\newcommand{\Hom}{\operatorname{Hom}}

\newcommand{\bbone}{\text{\usefont{U}{bbold}{m}{n}1}}
\MakeRobust{\bbone}
\newcommand{\unit}{\ensuremath{\mathds{1}}}
\newcommand{\un}{\unit}
\newcommand{\Id}{\mathrm{id}}
\newcommand{\id}{\mathrm{id}}

\newcommand{\FK}{{\Bbbk}}

\newcommand{\ve}{\varepsilon}
\newcommand{\wt}{\widetilde}

\newcommand{\ms}[1]{\mbox{\tiny\ensuremath{#1}}}

\newcommand{\Int}{\operatorname{Int}}

\newcommand{\kk}{\Bbbk}
\newcommand{\ideal}{\mathcal{I}}

\newcommand{\Kup}{\mathcal{K}}

\newcommand{\mt}{{\operatorname{\mathsf{t}}}}
\newcommand{\qt}{\mt}

\newcommand{\Rib}{\operatorname{Rib}}

\newcommand{\LL}{\mathcal{L}}

\newcommand{\TSkein}{{\mathscr{S}}}
\newcommand{\Skein}{\TSkein}

\newcommand{\Emb}{\mathrm{Emb}}
\newcommand{\Vect}{\mathrm{Vect}}

\newcommand{\cob}{\operatorname{\textbf{Cob}}}
\newcommand{\nc}{\mathsf{nc}}
\newcommand{\man}{\operatorname{\textbf{Man}}}
\newcommand{\Sp}{\mathbb{S}}

\newcommand{\ldual}[1]{\leftidx{^\vee}{\!#1}{}}
\newcommand{\rdual}[1]{{#1}^\vee}
\newcommand{\lldual}[1]{\leftidx{^{\vee\vee}}{\!#1}{}}
\newcommand{\rrdual}[1]{{#1}^{\vee\vee}}
\newcommand{\bdual}[1]{{#1}^{**}}
\newcommand{\cc}{\cat}
\newcommand{\bb}{\mathcal{B}}
\newcommand{\aaa}{\mathcal{A}}
\newcommand{\zz}{\mathcal{Z}}
\newcommand{\opp}{\mathrm{op}}

\newcommand{\labela}{\renewcommand{\labelenumi}{{\rm (\alph{enumi})}}}

\newcommand{\pathtoFig}{}

\usepackage[calc]{picture}\usepackage{settobox}
\newlength{\posx}
\newlength{\posy}
\newlength{\Imagew}
\newlength{\Imageh}
\newlength{\Textw}
\newlength{\Texth}
\newsavebox{\Image}
\newsavebox{\Text}
\newcommand{\epsh}[2]{
  \savebox{\Image}{\epsfig{figure=\pathtoFig#1,height=#2}}
  \settoboxwidth{\Imagew}{\Image}
  \settoboxheight{\Imageh}{\Image}
  \begin{array}{c} \hspace{-1.3mm}
    \raisebox{-4pt}{\usebox{\Image}}
    \hspace{-1.9mm}\end{array}
}
\newcommand{\epsw}[2]{
  \savebox{\Image}{\epsfig{figure=\pathtoFig#1,width=#2}}
  \settoboxwidth{\Imagew}{\Image}
  \settoboxheight{\Imageh}{\Image}
  \begin{array}{c} \hspace{-1.3mm}
    \raisebox{-4pt}{\usebox{\Image}}
    \hspace{-1.9mm}\end{array}
}

\newcommand{\putlc}[3]{
  \savebox{\Text}{#3}
  \settoboxwidth{\Textw}{\Text}
  \settoboxheight{\Texth}{\Text}
  \setlength\posx{-\Imagew+0.4pt+0.01\Imagew*\real{#1}}
  \setlength\posy{-0.5\Imageh+2.5pt+0.01\Imageh*\real{#2}-0.5\Texth}
  \put(\posx,\posy){#3}
}

\newcommand{\putrc}[3]{
  \savebox{\Text}{#3}
  \settoboxwidth{\Textw}{\Text}
  \settoboxheight{\Texth}{\Text}
  \setlength\posx{-\Imagew+0.4pt+0.01\Imagew*\real{#1}-\Textw}
  \setlength\posy{-0.5\Imageh+2.5pt+0.01\Imageh*\real{#2}-0.5\Texth}
  \put(\posx,\posy){#3}
}
\newcommand{\putc}[3]{
  \savebox{\Text}{#3}
  \settoboxwidth{\Textw}{\Text}
  \settoboxheight{\Texth}{\Text}
  \setlength\posx{-\Imagew+0.4pt+0.01\Imagew*\real{#1}-0.5\Textw}
  \setlength\posy{-0.5\Imageh+2.5pt+0.01\Imageh*\real{#2}-0.5\Texth}
  \put(\posx,\posy){#3}
}
\newcommand{\putcb}[3]{
  \savebox{\Text}{#3}
  \settoboxwidth{\Textw}{\Text}
  \settoboxheight{\Texth}{\Text}
  \setlength\posx{-\Imagew+0.4pt+0.01\Imagew*\real{#1}-0.5\Textw}
  \setlength\posy{-0.5\Imageh+2.5pt+0.01\Imageh*\real{#2}-\Texth}
  \put(\posx,\posy){#3}
}
\newcommand{\putct}[3]{
  \savebox{\Text}{#3}
  \settoboxwidth{\Textw}{\Text}
  \settoboxheight{\Texth}{\Text}
  \setlength\posx{-\Imagew+0.4pt+0.01\Imagew*\real{#1}-0.5\Textw}
  \setlength\posy{-0.5\Imageh+2.5pt+0.01\Imageh*\real{#2}}
  \put(\posx,\posy){#3}
}

\newcommand{\co}{\colon}

\begin{document}

\title{Non-compact (2+1)-TQFTs from non-semisimple spherical categories}

\begin{abstract}
This paper contains three related groupings of results.  First, we consider a new notion of an admissible skein module of a surface associated to an ideal in a (non-semisimple) pivotal category.  Second, we introduce the notion of a chromatic category and associate to such a category a finite dimensional non-compact (2+1)-TQFT by assigning admissible skein modules to closed oriented surfaces and using Juh\'asz's presentation of cobordisms. The resulting TQFT extends to a genuine one if and only if the chromatic category is semisimple with nonzero dimension (recovering then the Turaev-Viro TQFT). The third grouping of results concerns sided chromatic maps in finite tensor categories.  In particular, we prove that every spherical tensor category (in the sense of Etingof, Douglas et al.) is a chromatic category (and so can be used to define a non-compact (2+1)-TQFT).
\end{abstract}

\author[F. Costantino]{Francesco Costantino}
\address{Institut de Math\'ematiques de Toulouse\\
118 route de Narbonne\\
 Toulouse F-31062}
\email{francesco.costantino@math.univ-toulouse.fr}

\author[N. Geer]{Nathan Geer}
\address{Mathematics \& Statistics\\
  Utah State University \\
  Logan, Utah 84322, USA}
\email{nathan.geer@gmail.com}

\author[B. Patureau-Mirand]{Bertrand Patureau-Mirand}
\address{Université Bretagne Sud, CNRS UMR 6205, LMBA, F-56000 Vannes,
  France}
\email{bertrand.patureau@univ-ubs.fr}

\author[A. Virelizier]{Alexis Virelizier} \address{Univ. Lille, CNRS, UMR 8524 - Laboratoire Paul Painlev\'e, F-59000 Lille, France}
\email{alexis.virelizier@univ-lille.fr}

\maketitle

\setcounter{tocdepth}{1}
\tableofcontents

\section*{Introduction}
In the seminal paper \cite{At}, Atiyah introduced the notion of a $(n+1)$-TQFT which is equivalent to a symmetric monoidal functor from the category  of $(n+1)$-dimensional cobordisms $\cob$ to the category of vector spaces $\Vect_\FK$.
Two milestones in this area are
the Reshetikhin-Turaev and Turaev-Viro
$(2+1)$\trait TQFTs associated to certain semisimple categories. The first is based on modular categories, see~\cite{RT2,Tu,BHMV95}, the second is based on spherical categories, see~\cite{TV,BW99}, and these constructions are related in~\cite{TVi5}.
Later the first approach has been extended to constructions coming from non-semisimple modular categories, see for example \cite{KerLy2001, BCGP14,D17,DRGGPMR2022}.
The focus of this paper is to extend the second approach to non-compact $(2+1)$-TQFTs coming from non-semisimple spherical categories.  Here ``non-compact'' means  $\cob$ is replaced with its subcategory $\cob^\nc$ of cobordisms where each component has nonempty source
(this terminology is used by Lurie \cite[Definition 4.2.10]{Lu2010}).

Another active area of research is the study of skein modules.  In general, a skein module of a manifold $M$ is an algebraic object defined as a formal linear combination of embedded graphs in $M$, modulo local relations.  An important example of such modules is the Kauffman skein algebra of a surface introduced independently by Przytycki \cite{Pr1991,Pr1999}
and  Turaev \cite{Tu1991b}.  It has a simple and combinatorial definition where the local relations are determined by the Jones polynomial or equivalently the  Kauffman bracket.  
In particular, the Kauffman  skein algebra $\Skein(S^2)$ associated to the 2-sphere is one dimensional, its dual $\Skein(S^2)^*$ is cononically isomorphic to the linear span of the quantum trace on the category of finite dimensional modules over $U_q(\mathfrak{sl}_2)$, and the natural pairing of these spaces recovers the Jones polynomial.  
 The simple definition of the Kauffman  skein algebra hides deep connections to many interesting objects like character varieties, TQFTs, quantum Teichm\"uller spaces, and many others, see for example \cite{Bu,BW2011,
Mull, Si2000, dTh}.  
 
The results of this paper fall into three main groupings. First, we give a new notion of admissible skein modules associated to an ideal of a pivotal $\FK$-category.  Second, we show these modules are the TQFT spaces of a finite dimensional non-compact $(2+1)$-TQFT
associated to a new type of category called a chromatic category (which is a pivotal $\FK$-category endowed with
a non-degenerate m-trace and a chromatic map).  Finally, we show that any (non-semisimple) spherical tensor category (as defined in \cite{DSPS20,EGNO}) is a chromatic category.

Let us describe each of these groupings of results in more detail.   Let $\cat$ be pivotal $\FK$-category, that is, a 
$\FK$\trait linear pivotal category such that hom-spaces are finite dimensional vector spaces and $\End_\cat(\unit) = \FK \Id_\unit$.
Given an ideal $\ideal$ (a full subcategory of $\cat$ closed under tensor product and retracts) and an 
oriented surface~$\Sigma$,  we define the admissible skein module $\Skein_{\ideal}(\Sigma)$ as the $\FK$-span of $\ideal$-admissible ribbon graphs in $\Sigma$ modulo the span of $\ideal$-skein relations (see Definition~\ref{sect-adm-skein-modules}).
Loosely speaking, an $\ideal$-skein relation is similar to a usual skein relation except that we require there is an edge colored with an object in $\ideal$ which is not completely contained in the local defining relation. We prove that the mapping class group of $\Sigma$ naturally acts on $\Skein_{\ideal}(\Sigma)$ (see Theorem~\ref{T:MappingClassActs}). We also establish (see  Theorem~\ref{T:DiskRmt}) that admissible skein modules are related to the notion of a modified trace (m-trace) on $\ideal$ defined in \cite{GKP1,GPV13}: the dual $\Skein_\ideal(S^2)^*$ of the admissible skein module of the 2-sphere is canonically isomorphic to the $\FK$-span of m-traces on $\ideal$ (a related result was stated in a talk by Walker \cite{Walker}).   The pairing of this space with $\Skein_\ideal(S^2)$ gives back the renormalized quantum invariants of links coming from these m-traces (see \cite[Section 1.5]{GP2018}), generalizing the above mentioned relationship between the Kauffman skein algebra, the Jones polynomial, and the quantum trace.   

The second main grouping of results of this paper answers the natural question: ``For which categories does the mapping class group action induced by admissible skein modules extend to a TQFT?''. The relevant categories are the \emph{chromatic categories}. These are pivotal $\FK$-categories endowed with a non-degenerate m-trace on the ideal of projective objects and a chromatic map (which plays the role of the so called ``Kirby color'' in the surgery semisimple approach). Note that we do not need chromatic categories to be abelian but instead we assume that any non zero morphism to the unit object $\unit$ is an epimorphism.  We show (Theorem~\ref{T:S}) that any chromatic category $\cc$ gives rise to non-compact TQFT
$$\TSkein\co \cob^\nc\to \Vect_\FK,$$
which extends to a genuine TQFT $\cob\to \Vect_\FK$ if and only if $\cc$ is semisimple with nonzero dimension (as a chromatic category, see Section~\ref{sect-chr-cat}). We prove (Theorem \ref{T:Kcat}) that the TQFT $\TSkein$ is an extension of the 3\trait manifolds invariant ${\mathcal K}_\cat$ defined in \cite{CGPT20}.
While the definition of  $\Kup_\cat$ is based on Heegaard decompositions, our construction only requires understanding local attachment of framed 0 and 1-spheres and then appeals to the
substantial work of Juh\'asz \cite{Juh2018} where the categories of cobordisms are described in terms of generators and relations.
We use the m-trace and chromatic map to build operators on admissible skein modules that satisfy the relations of $\cob^\nc$  and so induce
the functor $\TSkein\co \cob^\nc\to \Vect_\FK$.   Since the 3-manifold invariant $\Kup_\cat$  is  both a generalization of  the Turaev-Viro invariant and a version of the (unimodular) Kuperberg invariant,  the construction of this paper provides  non-compact  TQFTs for these two invariants.

The TQFT $\TSkein$ has several useful properties.  First, the
vector spaces associated to surfaces are easy to understand as they
are skein modules (for example, it is not hard to show they are finite
dimensional as soon as there is a projective generator).  Second, the action
of the mapping class groups on them is very natural. Third, the
algebraic data needed for the construction is easy to formulate with
low-level technology using monoidal categories.  In particular, the
quite technical notions used in many constructions of non-semisimple
TQFTs are replaced with a simple relation (see
Equation~\eqref{E:ChrMapDef}) relating an m-trace and a chromatic map.
The chromatic map has an explicit expression in many examples
(including categories of modules over small (super) quantum groups)
and it is a graphical tool which is easy to manipulate (see the proofs
below).  Finally, the straightforward language of this paper opens the
door for new applications and even broader generalizations.  In
particular, using a similar approach, $(3+1)$-TQFT
are defined in~\cite{CGHP23} from certain ribbon categories.  In another direction, it
would be interesting to extend the present results to non-unimodular
graded categories including the category of modules over the Borel
algebra of the unrestricted quantum group (which should have
applications related to $\mathsf{SL}(2,\C)$ Chern-Simons theory, see
\cite{CDGG2021}).  Furthermore, a graded extension of the techniques
of this paper would also include new examples with perturbative
modules over super Lie algebras, giving TQFTs which should be related
to a conjectural perturbative versions of super CS-theory (see
\cite{aghaei2018, GY2022, Mik2015, RS92} and references within).

The final main grouping of results of this paper concerns the existence of chromatic maps and chromatic categories.    
Given a finite tensor category $\cc$ (in the sense of~\cite{EGNO}), we introduce left and right chromatic maps for $\cc$ (see Section~\ref{sect-chromatic}). Their definition involves the distinguished invertible object of $\cc$. We show that left and right chromatic maps for $\cc$ always exist (Theorem~\ref{thm-chromatic}) and we explicitly describe them for categories of representations of finite dimensional Hopf algebras (Theorem~\ref{them-left-right-chromatic-Hopf}). The proof of the existence uses the central Hopf monad (which describes the center of $\cc$, see~\cite{BV3}) and the existence and uniqueness of (co)integrals based at the distinguished invertible object of $\cc$ (see~\cite{Sh2}). 
As a consequence, we get  (Theorem~\ref{T:SphericalChrom}) that spherical finite tensor categories (over an algebraically closed field) are chromatic categories (and so can be used to define non-compact (2+1)-TQFTs).

The research area around non-semisimple TQFTs is very active and many
recent results are related to this paper, see for example \cite{
  BCGP14, BGPR, BJSS2021, CGP14, CGuP2021, D17, DRGPM20, DRGGPMR2022,
  KTV23,KV2019, Ker2003, KerLy2001, Vir2006}.  In particular, in
\cite{Bart2022}, Bartlett used the same approach to recover
Turaev-Viro TQFT in the semisimple setting.  We expect that our
construction is related to the general universal non-semisimple TQFT
announced by Kevin Walker and David Reutter in \cite{Walker2021}.

\subsubsection*{Acknowledgments} F.C. is supported by CIMI Labex ANR
11-LABX-0040 at IMT Toulouse within the program ANR-11-IDEX-0002-02
and from the french ANR Project CATORE ANR-18-CE40-0024.  N.G.\ is
supported by the Labex CEMPI (ANR-11-LABX-0007-01), IMT Toulouse and
by the NSF grant DMS-2104497.  B.P. thanks the France 2030 framework
program Centre Henri Lebesgue ANR-11-LABX-0020-01 for creating an
attractive mathematical environment.  A.V. is supported by the FNS-ANR
grant OCHoTop (ANR-18-CE93-0002) and the Labex CEMPI
(ANR-11-LABX-0007-01).
\\

Throughout the paper,  $\FK$ is a field and all categories are supposed to be essentially small. \\

\section{Chromatic categories}\label{S:Alg}

In this section, after reviewing some categorical notions, we introduce chromatic maps and chromatic categories (which are the algebraic ingredients to construct non-compact (2+1)-TQFTs).

\subsection{Rigid and pivotal categories}\label{sect-pivotal-cat}
For the basics on monoidal categories, we refer for example to \cite{ML1,EGNO,TVi5}. We  will suppress in our formulas  the associativity  and unitality  constraints   of monoidal categories.   This does not lead  to   ambiguity because by   Mac Lane's   coherence theorem,   all legitimate ways of inserting these constraints   give the same result.
For  any objects $X_1,...,X_n$ with $n\geq 2$, we set
$$
X_1 \otimes X_2 \otimes \cdots \otimes X_n=(... ((X_1\otimes X_2) \otimes X_3) \otimes
\cdots \otimes X_{n-1}) \otimes X_n
$$
and similarly for morphisms.

Recall that a monoidal category is \emph{rigid} if every object admits a left dual and a right dual. Subsequently, when dealing with rigid categories, we shall always assume tacitly that for each object $X$, a left dual $\ldual{X}$ and a right dual $\rdual{X}$  have been chosen, together with their (co)evaluation morphisms
$$
\lev_X \co \ldual{X}\otimes X \to\unit,  \quad \lcoev_X\co \unit  \to X \otimes  \ldual{X}, \quad
\rev_X \co X\otimes \rdual{X}  \to\unit, \quad   \rcoev_X\co \unit  \to \rdual{X} \otimes X,
$$
where $\unit$ is the monoidal unit of $\cat$.   

A \emph{pivotal category} is a rigid monoidal category with a choice of left and right duals for objects so that the induced  left and right dual functors coincide as monoidal functors. Then we write $X^*=\ldual{X}=\rdual{X}$ for any $X \in \cc$, the dual $f^* \co Y^* \to X^* $ of a morphism $f \co X \to Y$ is computed by
$$
f^*=(\lev_Y \otimes \id_{X^*})(\id_{Y^*}\otimes f \otimes \id_{X^*})(\id_{Y^*}\otimes \rcoev_X)=(\id_{X^*} \otimes \rev_Y)(\id_{X^*}\otimes f \otimes \id_{Y^*})(\lcoev_X \otimes \id_{Y^*}),
$$
and
$$
\phi=\{\phi_X=(\Id_{X^{**}} \otimes \lev_X)(\rcoev_{X^*} \otimes \Id_X)\co X \to X^{**}\}_{X \in \cc}
$$
is a monoidal natural isomorphism relating the (co)evaluation morphisms, called the \emph{pivotal structure} of $\cc$.  

The categorical \emph{left trace} and \emph{right trace} of any endomorphism $f \co X \to X$ of a pivotal category $\cc$ are defined by
$$
\tr_l(f)=\lev_X(\Id_{X^*} \otimes f) \rcoev_X  \quad \text{and} \quad \tr_r(f)=\rev_X(f \otimes \Id_{X^*}) \lcoev_X.
$$
Both take values in the commutative monoid $\End_\cc(\unit)$ of endomorphisms of the monoidal unit $\unit$ and share a number of properties of the standard trace of matrices such as cyclicity (i.e., symmetry). More generally, the \emph{left partial trace} 
of a morphism $g \co X \otimes Y \to X \otimes Z$ is the morphism
$$
\ptr_l^X(g)= (\lev_X\otimes \Id_Z )( \Id_{X^*} \otimes g)( \rcoev_X\otimes \Id_Y) \co Y \to Z,
$$
and the  \emph{right partial trace} of a morphism $h\co X \otimes Y \to Z \otimes Y$ is the morphism
$$
\ptr_r^Y(h)=(\Id_Z \otimes \rev_Y)(h\otimes \Id_{Y^*})(\Id_X \otimes \lcoev_Y)  \co X \to Z. 
$$

\subsection{Penrose graphical calculus}\label{sect-Penrose}
We represent morphisms in a pivotal category $\cc$ by plane   diagrams to be read from the bottom to the top.
Diagrams are made of   oriented arcs colored by objects of~$\cc$  and of boxes colored by morphisms of~$\cc$.  The arcs connect
the boxes and   have no mutual intersections or self-intersections.
The identity $\id_X$ of an object $X$, a morphism $f\co X \to Y$, the composition of two morphisms $f\co X \to Y$ and $g\co Y \to
Z$, and the monoidal product of two morphisms $\alpha \co X \to Y$
and $\beta \co U \to V$  are represented as follows:
\[
  \id_X=
  \epsh{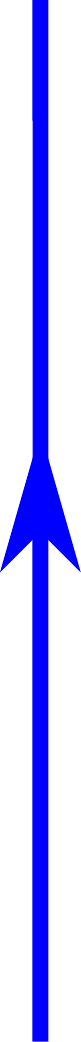}{10ex}
  \putlc{95}{50}{$\ms{X}$}
  \quad,\quad
  f=
  \epsh{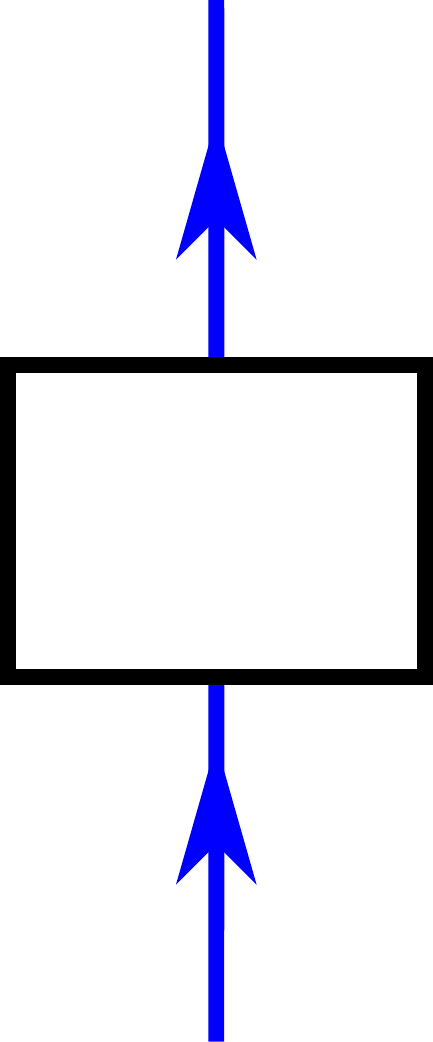}{10ex}
  \putc{49}{50}{$\ms{f}$}
  \putlc{60}{21}{$\ms{X}$}
  \putlc{60}{80}{$\ms{Y}$}
  \quad,\quad
  g\circ f=
  \epsh{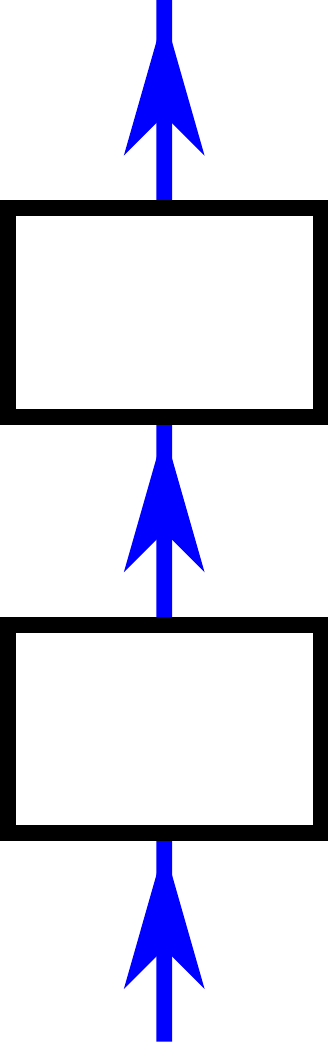}{11ex}
  \putc{49}{70}{$\ms{g}$}
  \putc{48}{30}{$\ms{f}$}
  \putlc{69}{91}{$\ms{Z}$}
  \putlc{68}{50}{$\ms{Y}$}
  \putlc{68}{10}{$\ms{X}$}
  \quad,\quad
  \alpha\otimes \beta=
  \epsh{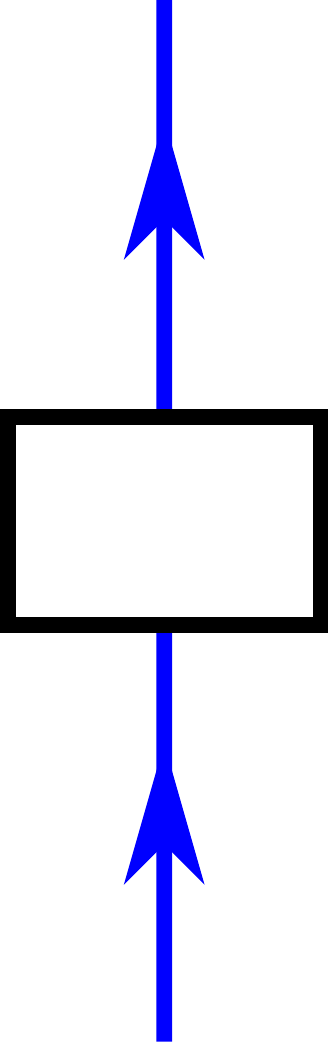}{10ex}
  \putrc{34}{80}{$\ms{Y}$}
  \putc{49}{50}{$\ms{\alpha}$}
  \putrc{35}{21}{$\ms{X}$} \;
  \epsh{fig2f.pdf}{10ex}
  \putlc{66}{81}{$\ms{V}$}
  \putc{50}{50}{$\ms{\beta}$}
  \putlc{66}{19}{$\ms{U}$}\;\;.
\]
A box whose lower/upper side has no attached strands represents a morphism with source/target $\un$.
If an arc colored by $X$ is oriented downward,
then the corresponding object   in the source/target of  morphisms
is~$X^*$. For example, $\id_{X^*}$  and a morphism $f\co X^* \otimes
Y \to U \otimes V^* \otimes W$  may be depicted as:
\[
  \id_{X^*}=
  \epsh{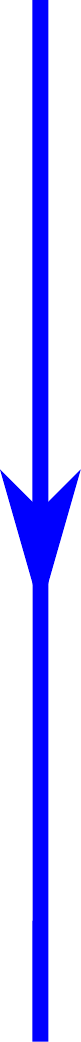}{10ex}
  \putlc{95}{50}{$\ms{X}$}\ 
  \quad\text{and}\quad f=
  \epsh{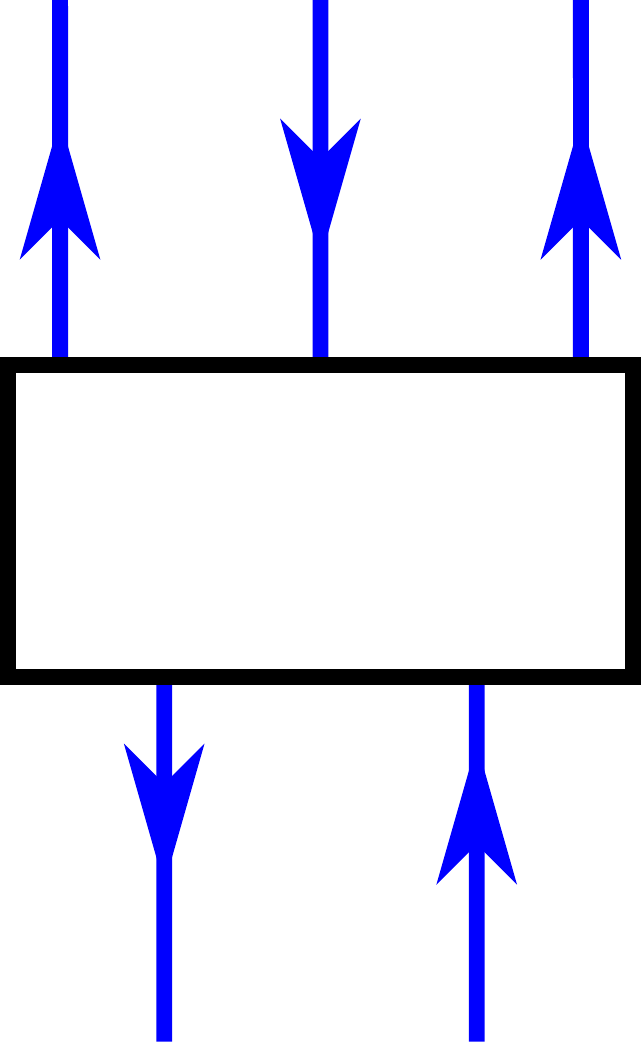}{10ex}
  \putlc{17}{82}{$\ms{U}$}
  \putlc{57}{82}{$\ms{V}$}
  \putlc{98}{82}{$\ms{W}$}
  \putlc{31}{21}{$\ms{X}$}
  \putlc{82}{21}{$\ms{Y}$}
  \putc{50}{50}{$\ms{f}$}\;\;.
\]
The duality morphisms are depicted as
\[
  \lev_V=\ \epsh{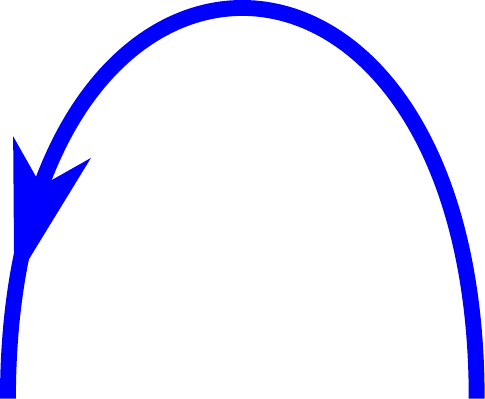}{3ex}
\putrc{0}{49}{$\ms{V}$}
\quad,\quad
\lcoev_V=\ \epsh{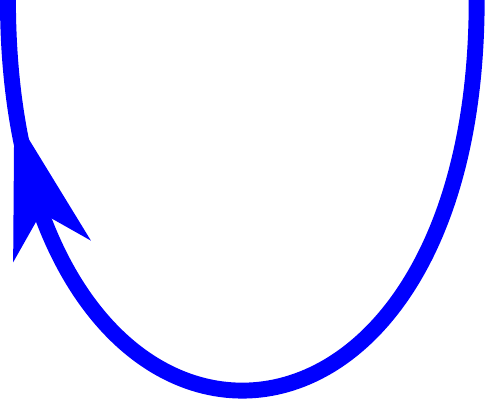}{3ex}
\putrc{0}{53}{$\ms{V}$}
\quad,\quad
\rev_V=\epsh{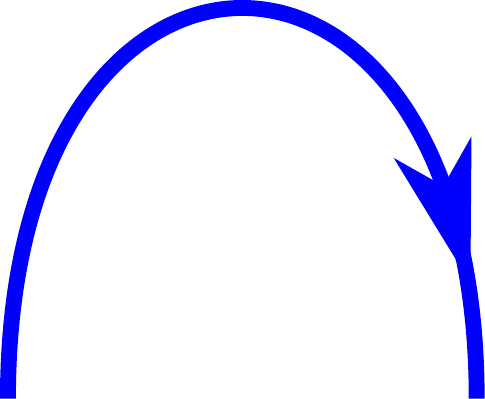}{3ex}
\putlc{100}{49}{$\ms{V}$}
\quad,\quad
\rcoev_V=\epsh{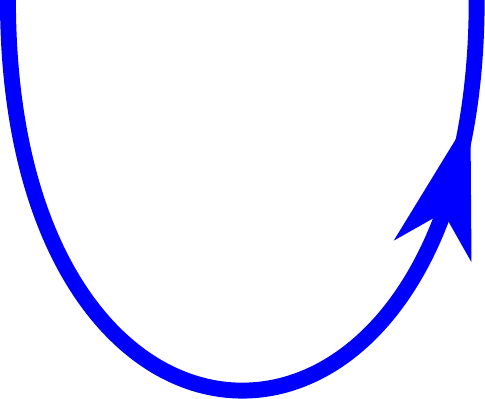}{3ex}
\putlc{99}{52}{$\ms{V}$}\quad.
\]
The partial traces of morphisms
$g \co X \otimes Y \to X \otimes Z$ and $h\co X \otimes Y \to Z \otimes Y$ are depicted as
\[\ptr_l^X(g)=\ 
\epsh{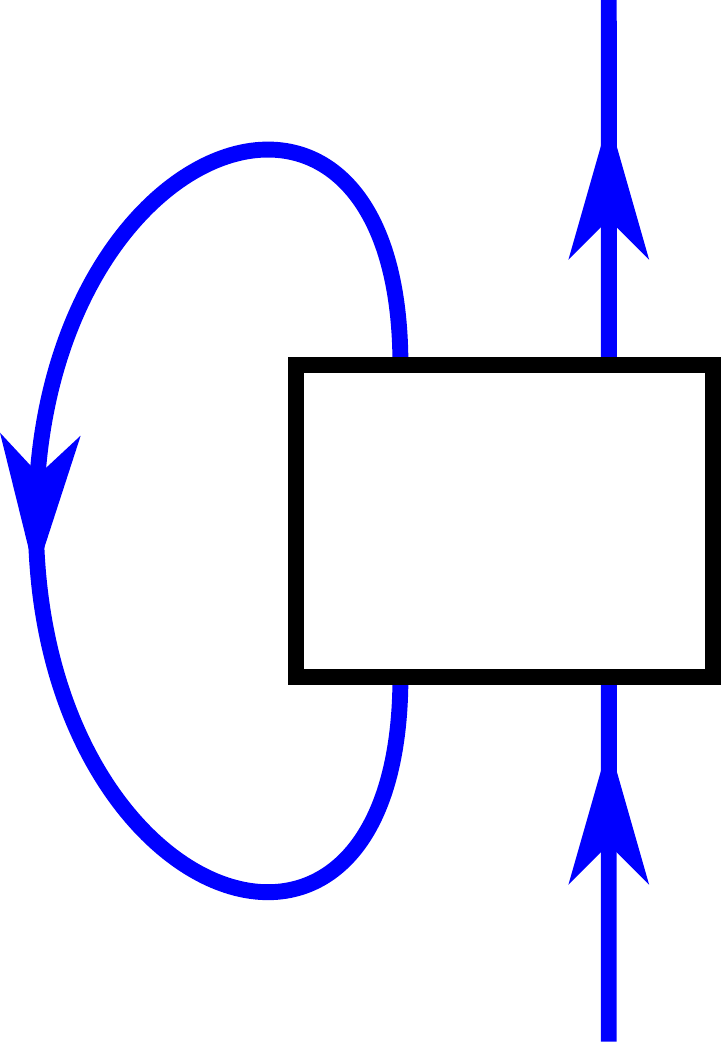}{10ex}
\putc{69}{50}{$\ms{g}$}
\putlc{93}{81}{$\ms{Z}$}
\putlc{94}{20}{$\ms{Y}$}
\putrc{0}{51}{$\ms{X}$}
\quad,\quad
\ptr_r^Y(h)=\,
\epsh{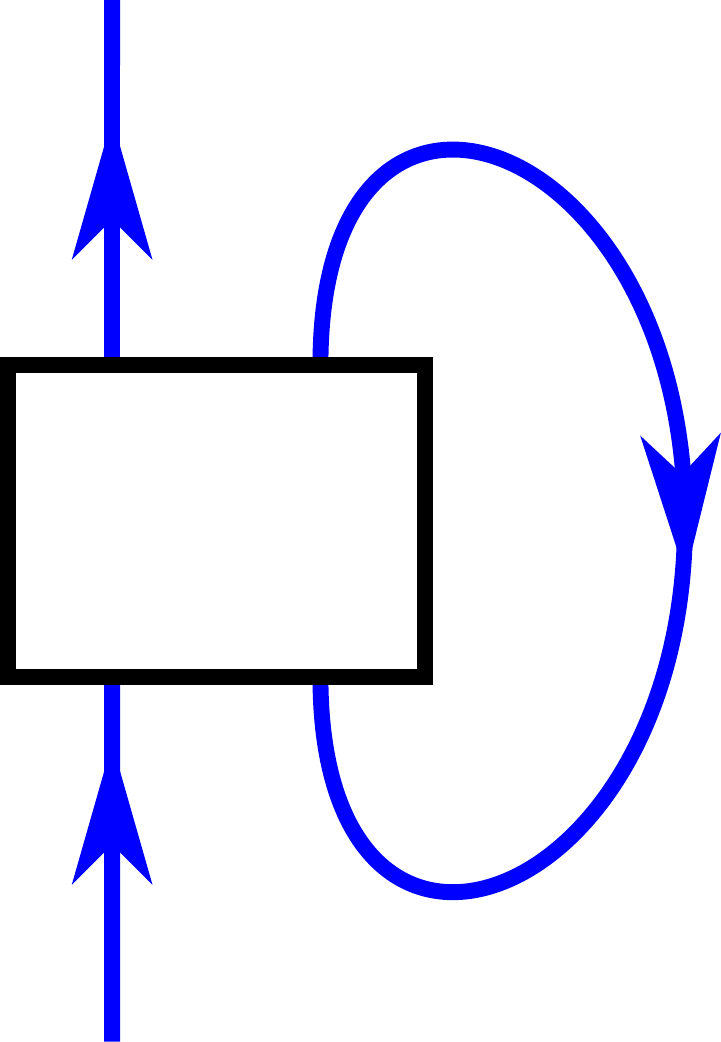}{10ex}
\putc{30}{50}{$\ms{h}$}
\putrc{8}{80}{$\ms{Z}$}
\putrc{8}{20}{$\ms{X}$}
\putlc{100}{52}{$\ms{Y}$}
\]
Note that the morphisms represented by the diagrams
are invariant under isotopies of the diagrams in the plane keeping
fixed the bottom and top endpoints (see \cite{JS,TVi5}).

\subsection{Projective objects, covers, and generators}\label{sect-projective}
An object $P$ of a category $\cc$ is \emph{projective} if the functor $\Hom_\cc(P,-)\colon \cc \to \mathrm{Set}$
preserves epimorphisms. A category has \emph{enough projectives} if every object has an epimorphism from a projective object onto it.

A \emph{projective cover} of an object $X$ of a category $\cc$ is a projective object~$P(X)$ of $\cc$ together with an epimorphism $p \co P(X) \to X$ such that if $g \co P \to X$ is an epimorphism from a projective object $P$ to $X$, then there exists an epimorphism $h \co P \to P(X)$ such that $ph = g$. In an abelian category, a projective cover (if it exists) is unique up to a non-unique isomorphism, and a projective cover of a simple object is indecomposable.

By a \emph{generator} of a preadditive category (that is, a category that is enriched over the category of abelian groups), we mean an object $G$ of the category such that any other object $X$ is retract of $G^{\oplus n}$ for some non-negative integer $n$. 
A \emph{projective generator} of a preadditive category $\cc$ is a generator of the full subcategory of projective objects of $\cc$.

\subsection{Linear monoidal categories}\label{sect-linear-categories}
A monoidal category is \emph{$\kk$-linear} if each hom-set carries a structure of a $\kk$-vector space so that  the composition and monoidal product of morphisms are $\kk$-bilinear. 

By a \emph{$\FK$-category}, we mean a $\kk$-linear monoidal category $\cat$ such that the hom-sets in $\cat$ are finite dimensional and the $\FK$-algebra map $\FK \to \End_\cat(\unit)$, $k \mapsto k \, \Id_\unit$ is an
isomorphism, used then to identify $\End_\cat(\unit)=\FK$. 

We say a $\FK$-category that $\cc$ is \emph{semisimple} if every object of $\cc$ is projective. Note that if $\cc$ is abelian, then $\cc$ is semisimple (in the above sense) if and only if it is abelian semisimple (in the sense every object is a direct sum of simple objects).

\subsection{Finite tensor categories}\label{sect-finite-tensor}
Assume in this subsection that $\kk$ is algebraically closed. Following \cite{EGNO}, a \emph{finite tensor category} (over $\kk$) is a rigid abelian $\kk$\trait category $\cc$ such that:
\begin{itemize}
\item every object of~$\cc$ has finite length,
\item the category $\cc$ has enough projectives,
\item there are finitely many isomorphism classes of simple objects.
\end{itemize}

Let $\cc$ be a finite tensor category. Then the unit object $\un$ of $\cc$ is simple (see \cite[Theorem 4.3.8]{EGNO}). Also, every simple object of $\cc$ has a projective cover, and any indecomposable projective object $P$ of $\cc$ has a unique simple subobject, called the \emph{socle} of $P$ (see \cite[Remark 6.1.5]{EGNO}). In particular, the socle of the projective cover of the unit object $\un$ is an invertible object called the \emph{distinguished invertible} object of $\cc$.
Finally $\cc$ has a projective generator (for example the direct sum of the projective covers of the elements of a representative set of the isomorphism classes of simple objects).

\subsection{Modified traces}\label{SS:trace}
Let $\cat$ be a pivotal $\FK$-category. We first recall from the definition
of a modified trace on an ideal of $\cat$ (see \cite{GPV13,GKP2} for details). 

An object $Y$ of $\cc$ is a \emph{retract} of an object $X$ of $\cc$ if there are morphisms $r \co X \to Y$ and $i \co Y \to X$ such that $ri=\id_Y$. An \emph{ideal} of $\cat$ is a full
subcategory $\ideal$ of~$\cat$ which is
\begin{itemize}
\item \emph{closed under monoidal products}: 
 for all $X\in \ideal$ and
  $Y\in \cat$, we have: $X\otimes Y\in \ideal$ and  $Y \otimes X\in \ideal$,
\item \emph{closed under retracts}: any retract of an object of $\ideal$ belongs to $\ideal$.
\end{itemize}
Recall from \cite{GPV13} that the pivotality of $\cc$ implies that any ideal of $\cc$ is stable under duality.

Let $\ideal$ be an ideal of $\cat$. A family
$\qt=\{\qt_X\co \End_\cat(X)\rightarrow \FK \}_{X\in \ideal}$  of $\kk$-linear forms satisfies the
\begin{itemize}
\item \emph{cyclicity property} if $ \qt_X(g f)=\qt_Y(f g)$ for all morphisms   $f \co  X \to Y$ and $g \co Y \to X$ with $X,Y\in \ideal$;
\item \emph{right partial trace property} if $\qt_{X\otimes Y}(f)=\qt_X\bigl(\ptr_r^Y(f)\bigr)$  for all $f\in \End_\cat(X\otimes Y)$ with  $X\in \ideal$;
\item \emph{left partial trace property} if $\qt_{Y\otimes X}(f)=\qt_X\bigl(\ptr_l^Y(f)\bigr)$  for all $f\in \End_\cat(Y\otimes X)$ with  $X\in \ideal$.
\end{itemize}
A \emph{right m-trace} (respectively \emph{left m-trace}, respectively \emph{m-trace}) on $\ideal$ is a family
$\qt=\{\qt_X\co \End_\cat(X)\rightarrow \FK \}_{X\in \ideal}$ of $\kk$-linear forms
satisfying the cyclicity and right (respectively left, respectively right and left) partial trace properties. 

For example, identifying $\End_\cc(\un)=\kk$, the family $\tr_r=\{f \in \End_\cc(X) \mapsto \tr_r(f) \in \kk\}_{X \in \cc}$ is a right m-trace on $\cc$ and the family  $\tr_l=\{f \in \End_\cc(X) \mapsto \tr_l(f) \in \kk\}_{X \in \cc}$ is a left m-trace on $\cc$ called the \emph{categorical left and right traces} of $\cc$. If these traces coincide, then $\tr=\tr_r=\tr_l$ is a m-trace on~$\cc$ called the \emph{categorical trace} of $\cc$.

A m-trace $\mt$ on an ideal $\ideal$ of $\cc$ is \emph{non-degenerate} if for any $X\in\ideal$, the pairing 
$$
\Hom_\cat(\unit,X)\otimes_\FK\Hom_\cat(X,\unit)\to\FK, \quad u\otimes v\mapsto \mt_X(uv)
$$
is non-degenerate. Given such a non-degenerate trace $\mt$, we set for any $X \in \ideal$, 
\begin{equation}\label{E:DefOmegaLambda}
\Omega_X=\sum_i x^i\otimes x_i \in\Hom_\cat(X,\unit)\otimes_\FK\Hom_\cat(\unit,X)
\quad \text{and} \quad
\Lambda^\mt_X =\sum_ix_i\circ x^i \in\End_\cat(X),
\end{equation}
where $\{x^i\}_i$ and $\{x_i\}_i$ are basis of $\Hom_\cat(X,\unit)$ and $\Hom_\cat(\unit,X)$ which are dual with respect to the m-trace~$\mt$, that is, such that $\mt_X(x_i\circ x^j)=\delta_{i,j}$. Clearly, $\Omega_X$ and $\Lambda^\mt_X$ are independent of the choice of such dual basis. The properties of the m-trace $\mt$ translate to the copairings $\Omega_X$ as follows:

\begin{lemma} \label{P:Omega-nat}
Let $X,Y\in \ideal$ and $Z \in \cc$, and let $f \co X\to Y$ be a morphism in $\cc$. 
  \begin{enumerate}
\labela
  \item \emph{Duality}: If $\Omega_X=\sum_ix^i\otimes x_i$, then
    $\Omega_{X^*}=\sum_i  (x_i)^*  \otimes
    (x^i)^*\in\Hom_\cat(X^*,\unit)\otimes_\FK\Hom_\cat(\unit,X^*).$
  \item \emph{Naturality}: If $\Omega_X=\sum_ix^i\otimes x_i$ and    $\Omega_Y=  \sum_j y^j\otimes y_j $,  then
    $$\sum_ix^i\otimes (f\circ x_i)= \sum_j(y^j\circ f)\otimes
    y_j  \in\Hom_\cat(X,\unit)\otimes_\FK\Hom_\cat(\unit,Y).$$
  \item \emph{Rotation}: If $\Omega_{X\otimes Z}=\sum_iz^i\otimes z_i$
    then $\Omega_{Z\otimes X}=\sum_i\wt z^i\otimes \wt z_i$ where
    $$ \wt z^i=\rev_Z(\Id_Z\otimes
    z^i\otimes\Id_{Z^*})(\Id_{Z \otimes X}\otimes\lcoev_Z) \quad \text{and} \quad \wt z_i=(\Id_{Z \otimes X}\otimes\rev_Z)(\Id_Z \otimes
    z_i\otimes\Id_{Z^*})\lcoev_Z.$$
  \end{enumerate}
\end{lemma}
\begin{proof}
  The duality and rotation properties follow  from the fact that  we apply
  transformations  sending  dual basis to dual basis. The naturality
  can be checked by applying
  $\mt_X(x_k\circ\_)\otimes\mt_Y(\_\circ y^\ell)$ to both sides:  it
 reduces then  to the cyclic property
  $\mt_Y(f\circ x_k\circ y^\ell)=\mt_X(x_k\circ y^\ell\circ f)$ of the m-trace $\mt$.
\end{proof}

\subsection{Chromatic maps}\label{SS:Def-chrmap}
Let $\cat$ be a pivotal $\FK$-category. The full subcategory $\Proj_\cat$ of projective objects of~$\cat$ is an ideal of $\cat$ (see \cite{GPV13}). Assume that $\cc$ is endowed with a non-degenerate m-trace $\qt$ on $\Proj_\cc$.

A \emph{chromatic map} for a projective generator $G$ of $\cc$ is a map $\chr\in\End_\cat(G\otimes G)$
satisfying
\begin{equation}\label{E:ChrMapDef}
(\id_G \otimes \lev_G \otimes \id_G)(\Lambda^\mt_{V\otimes G^*} \otimes \chr)(\id_G \otimes \rcoev_G \otimes \id_G)=\id_{G \otimes G},
\end{equation}
that is,
$$
\epsh{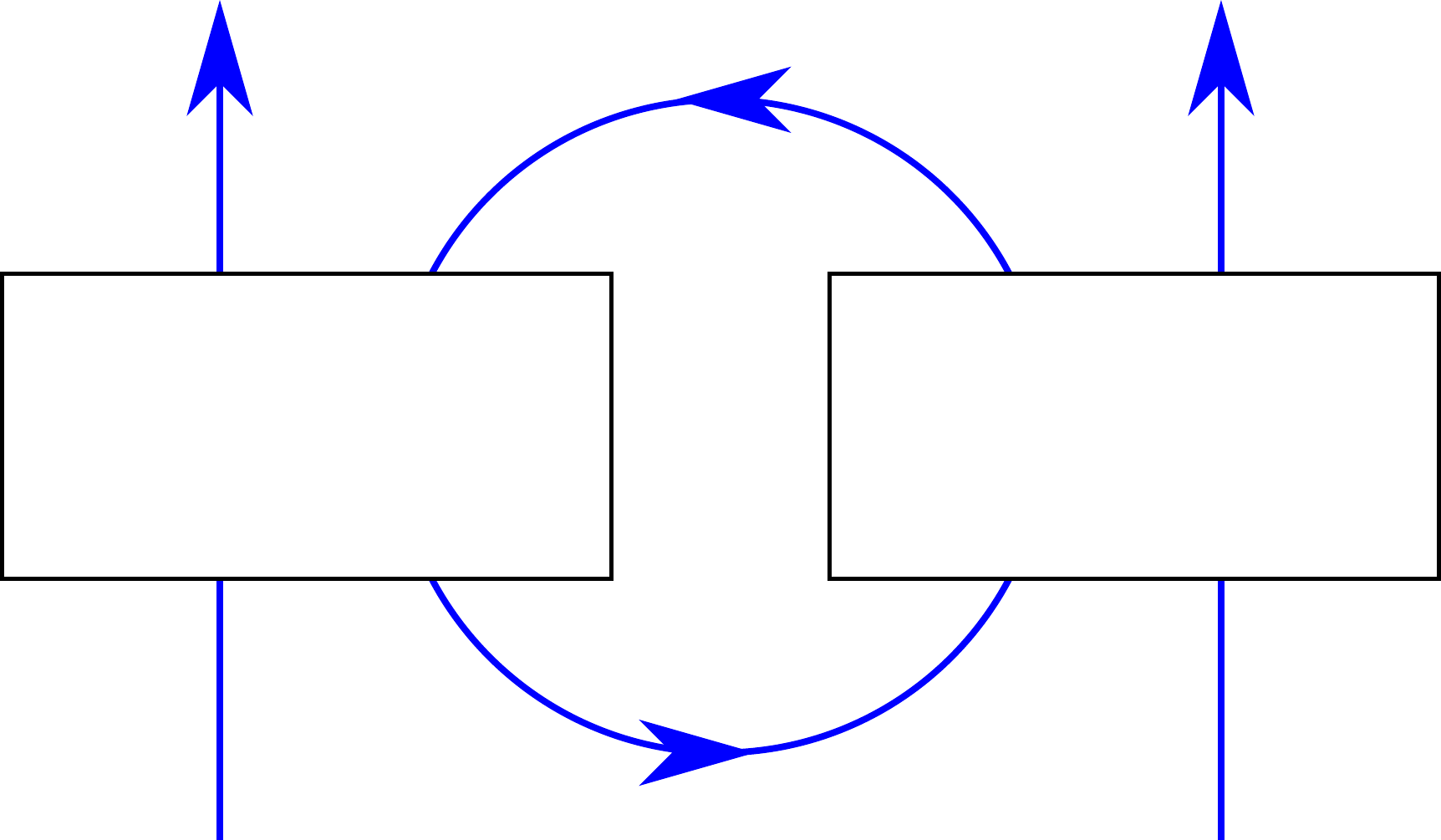}{10ex}
\putc{20}{50}{$\ms{\Lambda^\mt_{G\otimes G^*}}$}
\putc{79}{50}{${\chr}$}
\putrc{12}{91}{$\ms{G}$}
\putlc{89}{92}{$\ms{G}$}
\putcb{47}{4}{$\ms{G}$}
\putct{51}{95}{$\ms{G}$}
\ =\ 
\epsh{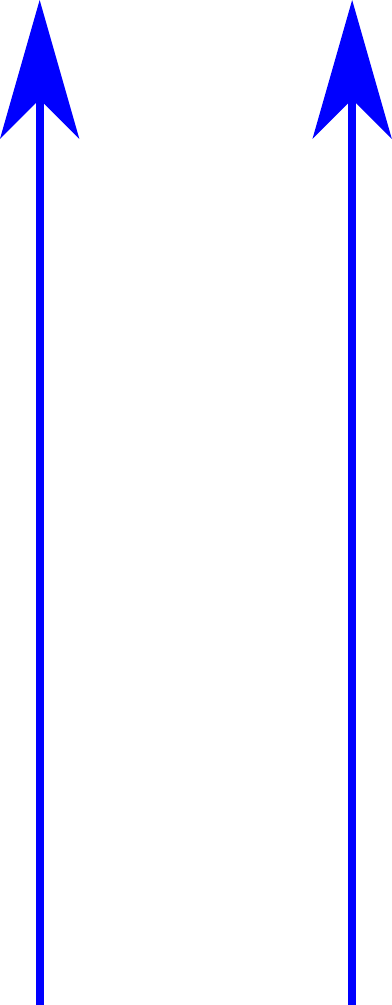}{10ex}
\putrc{-2}{92}{$\ms{G}$}
\putlc{100}{93}{$\ms{G}$}\;.
$$

More generally, a \emph{chromatic map based on a projective object $P$} for a projective generator $G$ is a map $\chr_P\in\End_\cat(G\otimes P)$ such that for all $X\in\cat$, 
\begin{equation}\label{eq:chrP}
(\id_X \otimes \lev_G \otimes \id_P)(\Lambda^\mt_{X\otimes G^*} \otimes \chr)(\id_X \otimes \rcoev_G \otimes \id_P)=\id_{X \otimes P},
\end{equation}
that is,
\[
  \epsh{fig19.pdf}{10ex}
  \putc{20}{50}{$\ms{\Lambda^\mt_{X\otimes G^*}}$}
  \putc{79}{50}{${\chr_P}$}
  \putrc{12}{91}{$\ms{X}$}
  \putlc{89}{92}{$\ms{P}$}
  \putcb{47}{4}{$\ms{G}$}
  \putct{51}{95}{$\ms{G}$}
  \ =\ 
  \epsh{fig20.pdf}{10ex}
  \putrc{0}{92}{$\ms{X}$}
  \putlc{100}{93}{$\ms{P}$}
  \;,\quad\text{or more explicitly}\quad
  \sum_i\,\epsh{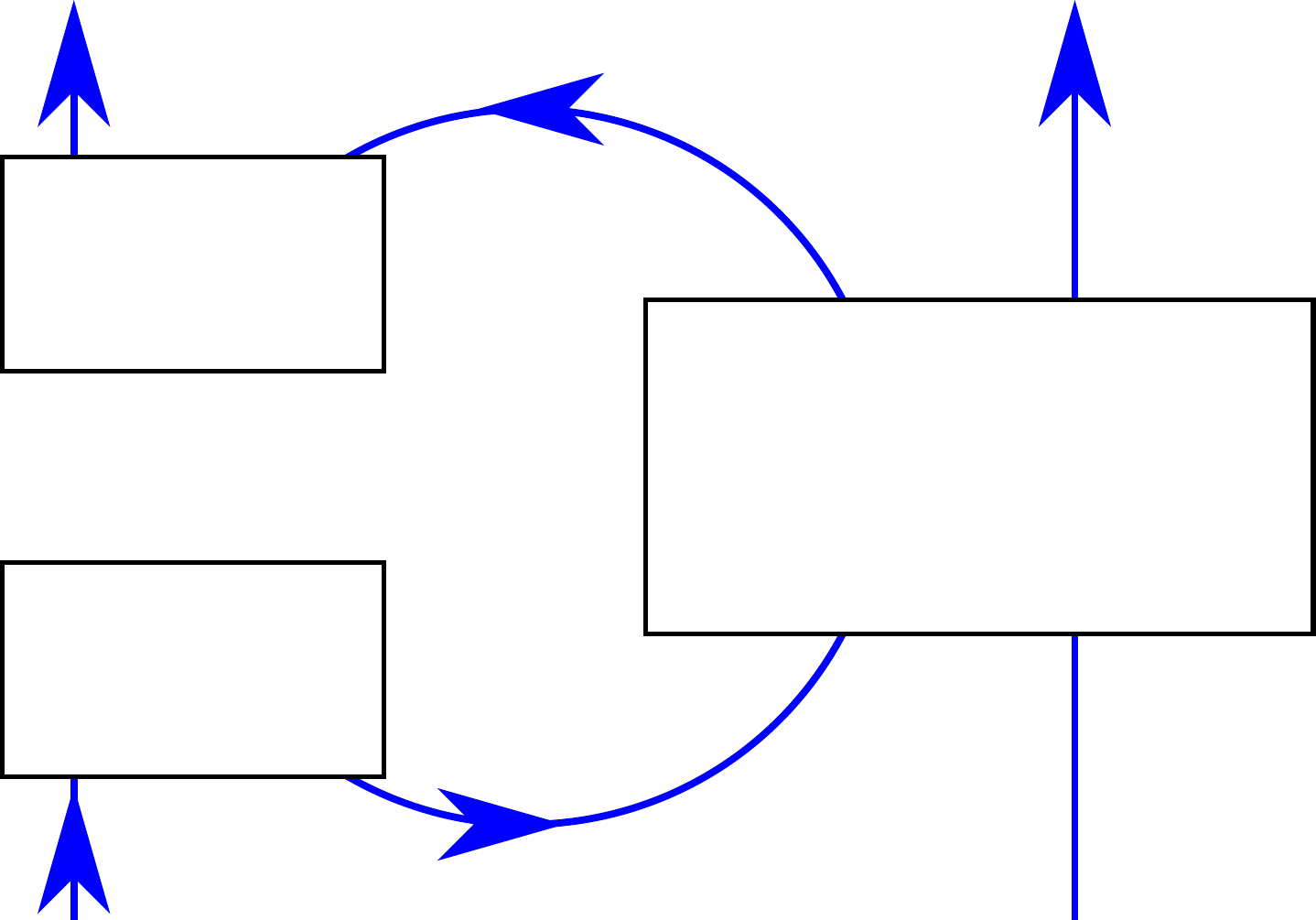}{10ex}
  \putc{14}{72}{$\ms{x_i}$}
  \putc{14}{27}{$\ms{x^i}$}
  \putc{74}{51}{${{\chr}}_P$}
  \putrc{2}{96}{$\ms{X}$}
  \putrc{2}{4}{$\ms{X}$}
  \putct{41}{94}{$\ms{G}$}
  \putcb{37}{5}{$\ms{G}$}
  \putlc{85}{92}{$\ms{P}$}
  \ =\ 
  \epsh{fig20.pdf}{10ex}
  \putrc{0}{92}{$\ms{X}$}
  \putlc{100}{93}{$\ms{P}$}
\]
where $\{x^i\}_i$ and $\{x_i\}_i$ are basis of $\Hom_\cat(X\otimes G^*,\unit)$ and $\Hom_\cat(\unit,X\otimes G^*)$ which are dual with respect to the m-trace~$\mt$.

Clearly, a chromatic map based on $G$ for a projective generator $G$ is a chromatic map for~$G$. Conversely, any chromatic map gives rise to chromatic maps based on projective objects:
\begin{lemma}\label{lem-chromatic-based} 
Let  $\chr\in\End_\cat(G\otimes G)$ be a chromatic map for a projective generator $G$ of $\cc$ and let $P\in\Proj_\cc$.
Pick any non zero morphism $\ve\co G\to\unit$ and a morphism $e_{P,G}\co P\to G\otimes P$  such that $\Id_P=(\ve\otimes\Id_P)e_{P,G}$ (such morphisms always exist). Then the map
$$
\chr_P=    (\Id_{G}\otimes\ve\otimes \Id_P) (\chr \otimes \Id_P) (\Id_{G}\otimes{{e}}_{P,G}) \in \End_\cat(G\otimes P)
$$
is a chromatic map based on $P$ for $G$. 
\end{lemma}
\begin{proof}
We first verify the existence of the maps $\ve$ and ${{e}}_{P,G}$.  Since $G^*\otimes G$ is projective and is a retract of finite direct sum of copies of $G$,  the (nonzero) evaluation epimorphism $\lev_G$ factors through $G$,  and so  there exists a nonzero map
$G\tto\ve\unit$ which then is an epimorphism.  Since $P\otimes P^*$ is projective,  the evaluation morphism $\rev_P \co P\otimes P^*\to\unit$ factors as $\rev_P= \varepsilon \circ {\wt e_{P,G}}$ for some ${\wt e_{P,G}} \co P\otimes P^* \to G$. Then  the map   $e_{P,G}=(\wt e_{P,G}\otimes\Id_{P})(\Id_P\otimes\rcoev_P)$ does satisfy $\Id_P=(\ve\otimes\Id_P){{e}}_{P,G}$. 

Next, denote by  $g\in \End_\cc(G \otimes G)$ be the left hand side of \eqref{E:ChrMapDef} and by $f_{X,P}\in \End_\cc(X \otimes P)$ the left hand side of \eqref{eq:chrP}.  Assume first that $X=Q\in\Proj_\cc$. Since $Q$ is a retract of a direct sum of copies of $G$, there is a finite family $\{\alpha_i \co Q \to G,\beta_i \co G \to Q\}_i$
of morphisms such that $\Id_Q=\sum_i\beta_i\alpha_i$.
Then,  using the naturality of $\Omega$ (see Lemma \ref{P:Omega-nat}) and the fact that $g=\Id_{G \otimes G}$ (since $c$ is a chromatic map for $G$), we obtain:
\begin{gather*}
f_{Q,P}
=\sum_i f_{Q,P}(\beta_i\alpha_i\otimes\Id_{P} )
=\sum_i (\beta_i \otimes \varepsilon \otimes \Id_P)(g\otimes \Id_P)(\alpha_i \otimes e_{P,G}) \\ =\sum_i \beta_i\alpha_i \otimes \bigl((\varepsilon \otimes \Id_P) e_{P,G} \bigr)=\Id_{Q \otimes P}.
\end{gather*}
Finally, let $X\in\cat$. The naturality of
$\Omega$
implies that $f_{X,P}(\Id_X \otimes\ve \otimes \Id_P)=(\Id_X
\otimes\ve \otimes \Id_P)f_{X \otimes
  G,P}$. The previous case applied to the projective object $Q=X
\otimes G$ gives that $f_{X\otimes G,P}=\Id_{X \otimes G \otimes
  P}$. Therefore $f_{X,P}(\Id_X \otimes\ve\otimes \Id_P)=(\Id_V
\otimes\ve\otimes \Id_P)$. Now $\Id_X \otimes\ve \otimes
\Id_P$ is an epimorphism because $\ve$ is an epimorphism and
$\cat$ is pivotal. Hence $f_{X,P}=\Id_{X \otimes
  P}$, that is, $\chr_P$ is a chromatic map based on $P$ for
$G$.  
\end{proof}
The existence of chromatic maps does not depend of the choice of the projective generator:
\begin{lemma}\label{lem-chromatic-for}
  Let $G,G'$ be projective generator and
  $\chr_P$ be a chromatic map based on a projective object $P$ for
  $G$. Then there is a finite family $\{\gamma_i \co G \to G',\delta_i \co G' \to G\}_i$ 
of morphisms such that $\sum_i\delta_i\gamma_i=\id_G$ and   $\chr'_P=\sum_i(\gamma_i\otimes\id_P)\chr_P(\delta_i\otimes\id_P)$ is a
  chromatic map based on $P$ for $G'$. 
\end{lemma}
\begin{proof}
  The existence of  $\{\gamma_i,\delta_i\}_i$  comes from the facts that
  $G$ is a retract of $(G')^{\oplus n}$.  
To prove that   $\chr'_P$ is a chromatic map, one can precompose
  $\lev_G$ with $\id_{G^* \otimes G}=\sum_i \id_{G^*} \otimes \delta_i\gamma_i$
  in Equation \eqref{eq:chrP},  and  then slide $\delta_i$ 
  using the naturality of $\Lambda^\mt_\bullet$.
\end{proof}
\subsection{Chromatic categories}\label{sect-chr-cat}
A \emph{chromatic category} (over $\kk$)  is a pivotal $\FK$-category $\cat$  endowed with  a non-degenerate m-trace on $\Proj_\cc$ such that:
\begin{itemize}
\item any non zero morphism to the unit object $\unit$ is an epimorphism,
\item there exists a chromatic map for a nonzero projective generator.
\end{itemize}
Note that Lemmas~\ref{lem-chromatic-based} and~\ref{lem-chromatic-for} imply that in a chromatic category, there are chromatic maps based at any projective object for any projective generator.

First examples of chromatic categories are given by spherical fusion categories and categories of representations of unimodular and unibalanced finite dimensional Hopf algebras, see the Examples~\ref{ex-chromatic-spherical-fusion} and~\ref{ex-Hopf-chromatic} below. A large family of chromatic categories is given by the spherical tensor categories over an algebraically closed field, see Theorem~\ref{T:SphericalChrom}.

A chromatic category is \emph{semisimple} if it is semsimple as a $\FK$-category (see Section~\ref{sect-linear-categories}) or, equivalently, if the  unit object $\unit$ is projective.  Note that the m-trace $\mt$ of a semisimple chromatic category is a nonzero multiple of the categorical trace $\tr$. Indeed the partial trace property implies that $\mt=\mt_\un(\id_\un) \tr$, and $\mt_\un(\id_\un)\neq 0$ because $\mt$ is nonzero.

The \emph{dimension} of a semisimple chromatic category $\cc$ is
$\dim(\cc)=\tr(\chr_\unit)=\frac{\mt_G(\chr_\un)}{\mt_\un(\id_\un)} \in \FK$ 
for any chromatic map~$\chr_\unit$ based on $\un$ for some projective
generator $G$ of $\cc$. (This terminology is justified by the last
assertion of Example~\ref{ex-chromatic-spherical-fusion}.)
 Note that $\dim(\cc)$ does not depend on the choice of $\chr_\unit$ (see Remark \ref{rk:dim-cat}) but does depend on the m-trace.

\begin{example}\label{ex-chromatic-spherical-fusion}
  Let $\cat$ be a spherical fusion $\FK$-category. Here, \emph{fusion}
  means that there is a finite family $I$ of objects of $\cc$ such
  that $\un \in I$, $\Hom_\cc(i,j)=\delta_{i,j} \FK \, \id_i$ for all
  $i,j \in I$, and each object of $\cc$ is a direct sum of objects in
  $I$. (Such fusion categories are in particular semisimple
  $\FK$-categories in the sense of
  Section~\ref{sect-linear-categories}).  Also, \emph{spherical} means
  that the categorical left and right traces of $\cc$ coincide (see
  Section \ref{SS:trace}). Then any object of $\cc$ is projective, the
  categorical trace $\tr$ is non-degenerate, $G=\bigoplus_{i\in I} i$
  is a (projective) generator of $\cc$, and for any object
  $P \in \cc$,
  $$
  \chr_P=\bigoplus_{i\in I} \mathrm{dim}(i) \, \Id_i\otimes \Id_P
  $$
  is a chromatic map based on $P$ for $G$, where
  $\mathrm{dim}(i)=\tr(\id_i) \in \FK$. Formally,
  $\chr_P=\Id_\Omega \otimes \Id_P$, where
  $\Omega=\bigoplus_{i\in I} \mathrm{dim}(i) \,i$ is the so-called
  ``Kirby color'' of $\cc$. Consequently, $\cc$ (endowed with its categorical trace)  is a semisimple
  chromatic category.  Note that the dimension of $\cc$ (as a
  semisimple chromatic category) coincides with its usual definition
  $\dim(\cc)=\sum_{i \in I} \dim(i)^2$ as a spherical fusion category.
(This follows from the computation of $\tr(\chr_\unit)$
for the above chromatic map based on $\unit$.)
\end{example}

\begin{example}\label{ex-Hopf-chromatic} 
Let $H$ be a finite dimensional Hopf algebra over $\kk$.  The category $H$\trait$\mathrm{mod}$ of finite dimensional (left) $H$-modules and $H$-linear homomorphisms is a $\kk$-category.  Assume that $H$ is unimodular and unibalanced in the sense of \cite{BBG}, meaning that the square of the antipode $S$ of $H$ is the conjugation by a square root $g$ of the distinguished grouplike element of $H$.  Pick a nonzero right integral $\lambda \co H \to \kk$ for~$H$. Then~$H$ is a projective generator of $H$\trait$\mathrm{mod}$, the integral $\lambda$ determines a non-degenerate m-trace $\mt$ on $\Proj_{\text{$H$\trait$\mathrm{mod}$}}$ characterized by $\mt_H(f)=\lambda(gf(1))$ for all $f \in \End_H(H)$, and a chromatic map for $H$ is 
$$
 \chr_H \co \left \{\begin{array}{ccl} H\otimes H & \to & H\otimes H \\ x \otimes y & \mapsto & \lambda(S(y_{(1)})gx)\, y_{(2)}\otimes y_{(3)}
  \end{array}\right.
$$
where $y_{(1)} \otimes y_{(2)} \otimes y_{(3)}$ is the double coproduct of $y$. (This follows from \cite[Lemma 6.3]{CGPT20} or the more general computations performed in Section~\ref{exa-H-mod}.) More generally, for any finite dimensional projective $H$-module~$P$, 
$$
 \chr_P= \sum_i (\Id_H \otimes g_i) \chr_H (\Id_H \otimes f_i) \co H \otimes P \to H \otimes P
$$
is a chromatic map based on $P$ for $H$, where $\{f_i\co P\to H,g_i \co H\to P\}_i$ is any finite family of $H$-linear homomorphisms such that $\Id_P=\sum_ig_if_i$.    Consequently, $H$\trait$\mathrm{mod}$  is a   chromatic category. In particular, finite dimensional modules over many small versions of (super) quantum groups fit into this setting. Note that $H$\trait$\mathrm{mod}$ is semisimple (as a chromatic category) if and only if $H$ is semisimple (as an algebra), and if such is the case, then  the dimension of $H$\trait$\mathrm{mod}$ (as a semisimple chromatic category) is equal to $\lambda(1)$ and so is nonzero if and only if $H$ is cosemisimple (by Maschke's theorem for Hopf algebras). Consequently, the chromatic category $H$\trait$\mathrm{mod}$ is semisimple with nonzero dimension if and only if $H$ is semisimple and cosemsisimple, or equivalently (by \cite[Corollary 3.2]{EG}) if and only if $H$ is involutory with $\dim_\kk(H)1_\kk \neq 0$. 
\end{example}

\subsection{Spherical tensor categories}\label{sect-spherical-cat-def}
Assume in this subsection that $\kk$ is algebraically closed.
A finite tensor category is \emph{unimodular} if its distinguished invertible  object (see Section~\ref{sect-finite-tensor}) is the unit object.

A \emph{spherical tensor category}  (over $\kk$) is a pivotal unimodular finite tensor category $\cc$  (over $\kk$) such that the right m-trace on $\Proj_\cc$ (which exists and is unique up to scalar multiple by \cite[Corollary~5.6]{GKP22}) is also a left m-trace. Note that  by \cite[Theorem 1.3]{SS2021}, this definition agrees with \cite[Definition 3.5.2]{DSPS20} where the above condition on the right m-trace is replaced by the equality of the square of the pivotal structure with the Radford equivalence. The first main result of the paper is the following: 
\begin{theorem}\label{T:SphericalChrom}
Any spherical tensor category over an algebraically closed field is a chromatic category.  
\end{theorem}
Theorem~\ref{T:SphericalChrom} is a reformulation of Corollary~\ref{cor-exists-chromatic} below which is a consequence of a related more general result (Theorem \ref{thm-chromatic}) stating the existence of left and right chromatic maps in any finite tensor category.  

Note that the categories of Examples~\ref{ex-chromatic-spherical-fusion} and~\ref{ex-Hopf-chromatic} are examples of spherical tensor categories when the ground field $\kk$ is algebraic closed. Moreover, a spherical tensor category over an algebraically closed field which is semisimple (as a chromatic category or, equivalently, as an abelian category) is a spherical fusion category (in the sense of Example~\ref{ex-chromatic-spherical-fusion}).

\section{Admissible skein modules}\label{sect-adm-skein}

Throughout this section, $\cc$ is a pivotal $\FK$-category and $\ideal$ is an ideal of $\cc$. We introduce $\ideal$-admissible graphs in surfaces and use them to construct the skein module functor.

\subsection{Ribbon graphs}
Loosely speaking, a ribbon graph is an oriented compact surface embedded in manifold which is decomposed into elementary pieces: bands, annuli, and coupons, see \cite{Tu}. A $\cc$-coloring of such a graph is a labeling of the core of each band and annuli with an object of $\cc$ and a compatible morphism to each coupon.
We proceed to precise definitions as in \cite{TVi5} in the case of surfaces. 
A \emph{circle} is a 1-manifold homeomorphic to $S^1$. An \emph{arc} is a 1-manifold homeomorphic to the closed interval $[0,1]$.   The   boundary points of an arc  are called its \emph{endpoints}. A \emph{rectangle} is a   2-manifold  with corners    homeomorphic to $[0,1] \times [0,1]$. The four corner points of a rectangle split its boundary  into four arcs called the \emph{sides}. A \emph{coupon} is an oriented  rectangle   with a distinguished  side called the \emph{bottom base}, the opposite side being    the \emph{top base}.  

A \emph{plexus} is a topological space    obtained from a   disjoint
union of a finite number  of   oriented  circles,  oriented  arcs, and coupons
by gluing    some   endpoints of the arcs to the bases of the coupons. We require  that different endpoints of  the arcs are never glued to the same  point of a (base of a) coupon. The endpoints of the arcs  that are not glued to coupons are called \emph{free ends}. The set of free ends of a plexus $\Gamma$ is denoted by $\partial \Gamma$. The arcs and the circles    of a plexus   are  collectively   called  \emph{strands}.

A \emph{ribbon graph} in an oriented surface  $\Sigma$  is   a  plexus   embedded in~$\Sigma$  such that all coupons  of~$\Gamma$ are embedded  in $\Int (\Sigma)=\Sigma\setminus\partial \Sigma$ preserving   orientation, $\Gamma\cap \partial \Sigma=\partial \Gamma$, the arcs and coupon of $\Gamma$ are smoothly embedded and  the arcs of~$\Gamma$ meet $\partial \Sigma$ transversely.

\subsection{Colored ribbon graphs}
A \emph{$\cc$-coloring} of a  plexus~$\Gamma$ is a function   assigning to every strand  of~$\Gamma$  an object of $\cc$, called its \emph{color},   and assigning to every coupon~$Q$  of~$\Gamma$  a morphism $Q_\bullet\to Q^\bullet$ in $\cc$.   Here  $ Q_\bullet$ and $Q^\bullet$ are objects     of $\cc$ defined as follows.   Let us   call the   endpoints of the arcs of~$\Gamma$  lying on the bottom (respectively, top) base   of~$Q$ the \emph{inputs} (respectively, \emph{outputs}) of~$Q$.   The   orientation of the bottom base of~$Q$ induced by the orientation of  $Q$
determines an order in the set of the inputs. Let $X_i \in \cc$ be the  color of the  arc of~$\Gamma$ adjacent to   the $i$-th input. Set $\varepsilon_i =+$ if   this arc is directed toward  $Q$ at the $i$-th input  and $\varepsilon_i = -$ otherwise.
The    orientation of the top base of~$Q$ induced by the  orientation of  $Q$
determines an order in the set of the outputs, and we take the opposite order. Let $Y_j \in \cc$ be the  color of   the  arc of~$\Gamma$ adjacent to   the $j$-th output. Set $\nu_j =-$ if  this arc is directed toward~$Q$ at the $j$-th output and $\nu_j = +$ otherwise.   Then
$$
Q_\bullet=X_1^{\varepsilon_1} \otimes \cdots  \otimes X_m^{\varepsilon_m} \quad {\rm {and}} \quad  Q^\bullet=Y_1^{\nu_1} \otimes \cdots  \otimes Y_n^{\nu_n},
$$
where $m$ and $n$ are respectively the numbers of inputs and outputs of~$Q$  and, as usual,   $X^+=X$ and $X^-=X^*$ for  $X \in \cc$.  For example, the following coupon whose bottom base is the horizontal bottom one
\[
  \epsh{fig2k.pdf}{12ex}
  \putlc{15}{82}{$\ms{Y_1}$}
  \putlc{56}{82}{$\ms{Y_2}$}
  \putlc{97}{82}{$\ms{Y_3}$}
  \putlc{29}{21}{$\ms{X_1}$}
  \putlc{81}{21}{$\ms{X_2}$}
\]
must be colored with a morphism
$X_1^* \otimes X_2 \to Y_1 \otimes Y_2^* \otimes Y_3$

A ribbon graph is \emph{$\cc$-colored} if its underlying   plexus is endowed with a $\cc$-coloring.

\subsection{Invariants of colored ribbon graphs}
To each free end of a $\cc$-colored ribbon graph $\Gamma$ in $\R \times [0,1]$ is associated a signed object consisting of the color of the arc incident to the free end and of a sign $\pm 1$ depending if that arc is directed up or down. Then one can view $\Gamma$ as a morphism from the sequence of signed objects associated with its bottom free ends (i.e., its free ends in $\R \times \{0\}$) to the sequence of signed objects associated with its top free ends (i.e., its free ends in $\R \times \{1\}$). This defines a monoidal category $\Rib_\cat$ whose objects are finite sequences of signed objects, whose morphisms are isotopy classes of $\cc$-colored ribbon graph in $\R \times [0,1]$, whose composition is given by putting one  $\cc$-colored ribbon graph on top of the other, and whose monoidal product is  given by concatenation. The graphical calculus of Section~\ref{sect-Penrose} gives rise to a monoidal functor
\begin{equation}\label{eq-defF}
F\co \Rib_\cat \to \cat.
\end{equation}
If the left and right traces $\tr_l$ and $\tr_l$ on $\cc$ coincide (see Section~\ref{sect-pivotal-cat}), then $F$ induces an isotopy invariant $F \co \LL \to \End_\cc(\un)=\kk$, where $\LL$ is the class of $\cat$-colored ribbon graphs in $S^2=(\R\times ]0,1[)\cup\{\infty\}$. This invariant can be renormalized using a modified trace as follows.

Denote by $\LL_\ideal$ the class of $\cat$-colored ribbon graphs in $S^2$ having at least one strand colored with an object in $\ideal$. In particular, each $\Gamma \in \LL_\ideal$ is the braid closure of some $\cc$-colored ribbon graph $T_X$ in $\R \times [0,1]$   with exactly one bottom free end and one top free end both supported by arcs oriented upward and colored by some object $X \in \ideal$, so that $F(T_X) \in \End_\cc(X)$.
Then, by \cite[Theorem 5]{GPV13}, each m-trace $\mt$ on $\ideal$ induces an isotopy invariant
\begin{equation}\label{E:DefF'}
F'\co \LL_{\ideal}\to \kk, \quad \Gamma \mapsto F'(\Gamma)=\mt_X\bigl(F(T_X)\bigr).
\end{equation}

\subsection{Admissible graphs}\label{sect-adm-graphs-results}
Let $\Sigma$ be an oriented surface. An \emph{$\ideal$-admissible graph} in $\Sigma$ is a $\cat$-colored ribbon graph $\Gamma$ in $\Sigma$ with no free ends such that each connected component of $\Sigma$ contains at least one strand of $\Gamma$ colored with an object in $\ideal$.   

Given $\ideal$-admissible graphs $\Gamma_1, \dots,\Gamma_k$ in $\Sigma$ and $a_1, \dots,a_k\in \FK$, the linear combination $a_1\Gamma_1 + \cdots + a_n\Gamma_n$ is a \emph{$\ideal$-skein relation (in $\Sigma$)} if there is a coupon $Q$ embedded in $\Sigma$ and $\ideal$-admissible graphs $\Gamma'_1, \dots,\Gamma'_k$ in $M$ such that:
\begin{itemize}
\item $\Gamma'_i$ is isotopic to $\Gamma_i$ (as a $\cat$-colored graph in $\Sigma$) for all $1 \leq i \leq k$;
\item the $\Gamma'_i$s coincide outside $Q$: $\Gamma'_i\cap (\Sigma \setminus Q)=\Gamma'_j\cap (\Sigma \setminus Q)$ for all $1 \leq i,j \leq k$;
\item $\Gamma'_i$ intersects $\partial Q$ only in its bottom and tops bases and transversally along the stands of $\Gamma'_i$ (so that $\Gamma'_i \cap Q$ can be seen as a $\cc$-colored ribbon graph in $\R \times [0,1]$)  for all $1 \leq i \leq k$;
\item $a_1 F(\Gamma'_1 \cap Q) + \cdots + a_k F(\Gamma'_k \cap Q)=0$ (as a morphism in $\cat$);
\item each $\Gamma'_i$ has an edge colored by a projective object which is not entirely contained in the coupon $Q$.
\end{itemize}

Two linear combinations of $\ideal$-admissible graphs are \emph{$\ideal$-skein equivalent} if their difference is an $\ideal$-skein relation.

The next lemma will be useful in the sequel:
\begin{lemma}\label{L:2boxes}
Let $\{\Gamma_i\}_i$ be a finite family of  $\ideal$-admissible graphs in $\Sigma$ which are identical
  outside two disjoint coupons $Q_1,Q_2$ and which intersect these coupons transversally and only their bottom an top bases.  Let $s=\sum_i c_i\Gamma_i$
  be a formal sum (with $c_i \in \FK$) and suppose that $F(s\cap Q_1)\otimes_\FK F(s\cap Q_2)=0$ . Then $s$
  is a sum of skein relations.
\end{lemma}
\begin{proof}
Let $X,Y \in \cc$ such that $F(s\cap Q_1)\in \Hom_\cat(X,Y)$.  
Choose a basis $\{f_j\}_j$ of $\Hom_\cat(X,Y)$.
First apply skein relations in $Q_1$ to replace every graph $\Gamma_i$ with a linear combination of graphs where $Q_1\cap\Gamma_i$ is replaced with a unique coupon colored by one of the morphisms $f_j$.  Then $s$ is skein equivalent to $s'=\sum_j s_j$ where~$s_j$ collects all diagrams whose box $Q_1$ has a coupon colored by $f_j$. Now $F(s'\cap Q_1)\otimes_\FK F(s'\cap Q_2)=0=\sum_j c_j f_j\otimes_\FK F(s_j\cap Q_2)$ for some constants $c_j\in \FK$.
Since the $f_j$ are linearly independent, we conclude that $F(s_j\cap Q_2)=0$, and so $s_j$ is a skein relation.
\end{proof}

\subsection{Admissible skein modules}\label{sect-adm-skein-modules}
The \emph{$\ideal$-admissible skein module} $\Skein_{\ideal}(\Sigma)$
of an oriented surface $\Sigma$ is the quotient of the $\FK$-vector
space generated by the $\ideal$\trait admissible graphs in $\Sigma$ by
its vector subspace generated by the $\ideal$-skein
relations.  The empty graph in $\Sigma$ is not admissible unless $\Sigma$
is empty.  Then $\Skein_\ideal(\emptyset)$ is the 1-dimensional vector
space generated by the empty graph.

\begin{lemma}\label{prop:ddd}
$\Skein_\ideal(\Sigma)$ is generated by $\ideal$-admissible graphs where each strand is colored by an object of $\ideal$.
\end{lemma}
\begin{proof}
By inserting coupons colored by identities and using that an $\ideal$-admissible graph has no free ends, it is easy to see $\Skein_\ideal(\Sigma)$ is
generated by  $\ideal$-admissible graphs with no circles and where each arc is joining two different coupons.  Then we can induct on the number of arcs whose color does not belong to $\ideal$.
Pushing an arc colored by $X \in \ideal$ next to an arc colored by $Y \in \cc$ (through some isotopy) and using a skein relation, we can replace the $Y$-colored arc with an arc colored by $Y\otimes X\in\ideal$ and changing the incident coupons as in the following figure:
$$
\epsh{fig4aSkein}{16ex}\put(1,2){\ms{X}}\put(-25,2){\ms{Y}}
   \put(-19,22){\ms{f}}\put(-19,-19){\ms{g}}
   \qquad\longrightarrow\quad
   \epsh{fig4bSkein}{16ex}\put(-11,2){\ms{Y\!\otimes\! X}}
   \put(-25,22){\ms{f\!\otimes\!\Id_X}}\put(-25,-20){\ms{g\!\otimes\!\Id_U}}\quad.
$$
This reduces the number of arcs whose color does not belong to $\ideal$.
\end{proof}

If $f\co \Sigma\to \Sigma'$ is an orientation preserving  embedding and $\Gamma$ is a ribbon graph in $\Sigma$,
then $f(\Gamma)$ is a ribbon graph in $\Sigma'$ in an obvious
way. Further, if $\Gamma$ is $\cc$-colored, then so if $f(\Gamma)$
(with colors inherited from $\Gamma$).  An embedding $f\co \Sigma\to \Sigma'$ is
\emph{admissible} if $f(\Sigma)$ meets every component of $\Sigma'$ or,
equivalently, if $H_0(f)$ is surjective. The image under an admissible  orientation preserving 
embedding $f$ of an $\ideal$-admissible graph is an $\ideal$-admissible graph.
Clearly, the image under $f$ of a skein relation in $\Sigma$ is a
skein relation in $\Sigma'$.  Consequently the map
$\Gamma \mapsto f(\Gamma)$ induces a $\kk$-linear homomorphism
$$
\Skein_{\ideal}(f) \co \Skein_\ideal(\Sigma) \to \Skein_\ideal(\Sigma').
$$

Let  $\Emb_2^a$ be the category whose objects are oriented surfaces and morphisms are isotopy classes of  admissible  orientation
preserving embeddings. This is a monoidal category with
disjoint union as monoidal product.  Denote by
$\Vect_\FK$ the monoidal category of $\FK$-vector spaces and
$\kk$-linear homomorphisms.

\begin{theorem}\label{T:MappingClassActs}
Recall, $\cc$ is a pivotal $\FK$-category.  The assignments $\Sigma \mapsto \Skein_\ideal(\Sigma)$ and $f \mapsto \Skein_{\ideal}(f)$ define a monoidal functor
$$
\Skein_{\ideal}\co \Emb^a_2 \to \Vect_\FK.
$$
In particular, this functor provides representations of the mapping class group of surfaces. Moreover, if the ideal $\ideal$ has a generator (in the sense of Section~\ref{sect-projective}), then for any closed oriented surface $\Sigma$, the $\FK$-vector space $\Skein_\ideal(\Sigma)$ is finite dimensional.
\end{theorem}
\begin{proof}
The functoriality and monoidality of $\Skein_{\ideal}$ are direct consequences of the definitions. Assume that $\ideal$ has a generator $G$ and let $\Sigma$ be a closed oriented   surface. 
It is sufficient to prove the last statement of the theorem for $\Sigma$ a compact connected surface.  Consider a cellularization of $\Sigma$ consisting in a single vertex $v$, $2g$ closed curves $c_1,\ldots, c_{2g}$ and one disk $D$. Let $\Gamma$ be an $\ideal$-admissible graph in $\Sigma$.  We can assume that~$\Gamma$ intersects each~$c_i$ transversally and that all its strands are $\ideal$-colored (by Lemma \ref{prop:ddd}). By fusing all the strands intersecting each $c_i$, we obtain that $\Gamma$ is skein equivalent to an $\ideal$-colored ribbon graph intersecting each $c_i$
once. Moreover, since $G$ is a generator of $\ideal$ up to applying some skein relation for each $c_i$, we can replace $\Gamma$ with a linear combination of $\ideal$-colored ribbon graphs intersecting $c_i$ via a single edge colored by the generator $G_i=G$ (here we denote the generator with a subscript $i$ so we can discern which one is associated to $c_i$).  Thus, $\Gamma$ is skein equivalent to a linear combination of graphs of the form of a bouquet of circles where each arc intersects a single $c_i$ once and is colored by $G_{i}$, and these arcs end up in a single coupon contained in the disk~$D$ and colored by some $f\in \Hom_{\cat}(\unit,G_{1}\otimes G_{2}\otimes G_{1}^*\otimes G_{2}^*\otimes \cdots\otimes G_{{2g-1}}\otimes G_{{2g}}\otimes G_{{2g-1}}^*\otimes G_{{2g}}^*)$.  Since this space of homomorphisms is finite dimensional (because $\cc$ is a $\kk$-category), we conclude that so is~$\Skein_\ideal(\Sigma)$.
\end{proof} 

In the next theorem, we interpret skein modules of the 2-disk $D^2$ and the sphere 2-sphere in terms of m-traces. 
Note that Walker and Reutter announced in \cite{Walker} a related result. 
\begin{theorem}\label{T:DiskRmt}
Recall, $\cc$ is a pivotal $\FK$-category.  There are canonical $\kk$-linear isomorphisms:
$$
\Skein_\ideal(D^2)^*\cong \{\text{right m-traces on } \ideal\}\cong \{\text{left m-traces on } \ideal\} \;\; \text{ and } \;\;
\Skein_\ideal(S^2)^*\cong \{\text{m-traces on } \ideal\}.
$$
\end{theorem}

We prove Theorem \ref{T:DiskRmt} in Section \ref{sect-proof-DiskRmtn}.

\begin{remark}\label{T:dim3SphereSkein}
Theorems~\ref{T:MappingClassActs} and~\ref{T:DiskRmt} have analogue in dimension 3 by assuming that $\cat$ is moreover ribbon (meaning that $\cc$ has a braiding so that the induced left and right twist coincide), by considering the Reshetikhin-Turaev functor $F$ from the category of $\cc$-colored ribbon graphs in $\R^2 \times [0,1]$ to $\cc$ (see  \cite{Tu}), and by using this functor to define (as above) the skein module $\Skein_\ideal(M)$ associated to an oriented compact 3-manifold $M$.
In particular, for the 3-ball $B^3$ and 3-sphere $S^3$, there are canonical $\kk$-linear isomorphisms
$$
\Skein_\ideal(B^3)^*\cong \Skein_\ideal(S^3)^*
  \cong \{\text{m-traces on } \ideal\}.
$$
These skein modules of 3-manifolds are used in  \cite{CGHP23} to construct (3+1)-TQFTs.
\end{remark}

\subsection{Skein modules elements from bichrome graphs}\label{sect-bichrome}
In this
subsection, we assume that $\cc$ is a chromatic category. Following \cite{CGPT20}, a \emph{bichrome graph} in a closed oriented surface $\Sigma$ is the disjoint union of an admissible graph in $\Sigma$ (called the \emph{blue part}) and finitely many pairwise disjoint unoriented embedded circles in $\Sigma$ (called the \emph{red part}). A \emph{red to blue modification} of a bichrome graph
is the modification in an annulus given by
\begin{equation}
  \label{eq:redtoblue}
  \epsh{fig3}{15ex}\put(1,15){\ms{P}}\longrightarrow
  \epsh{fig7}{15ex}\put(-16,-2){{${{\chr}}_P$}}
  \put(-27,16){\ms G} \put(-2,25){\ms P} \;,
\end{equation}
where $\chr_P$ is any chromatic map based on a projective object $P$ at a projective generator $G$ of $\cc$.
Here we allow the $P$-colored strand to be replaced by several parallel
strands with at least one colored by a projective object. Note that if the category $\cat$ is spherical fusion, then the red to blue modification amounts to arbitrarily orient the red curve and color it with the Kirby color of $\cat$ (see Example~\ref{ex-chromatic-spherical-fusion}).

Red to blue modifications transform any bichrome graph into a $\Proj_\cc$-admissible graph in~$\Sigma$ whose class in the skein module $\Skein_{\Proj_\cc}(\Sigma)$ is well-defined:

\begin{lemma}\label{prop:redtoblue}
Using the red to blue modification, bichrome graphs in $\Sigma$ represent well defined elements of the skein module $\Skein_{\Proj_\cc}(\Sigma)$.
\end{lemma}
\begin{proof}
To prove the lemma, we show that two red to blue modifications of a red curve at different places with different chromatic maps give skein equivalent diagrams. Let $P,Q$ be projective objects and $G,G'$ be projective generators of $\cc$. Pick  a chromatic map $\chr_P$ based on $P$ at $G$ and a chromatic map  $\chr_Q$ based on $Q$ at $G'$. There are two cases to consider. First, if the two
  modifications are made on the same side of the red curve, then
$$
  \epsh{fig4a}{24ex}\put(-19,-32){\ms{{{\chr}}_P}}\put(-21,45){\ms{Q}}
    \quad=\quad\sum_i
    \epsh{fig4b}{24ex}\put(-19,-32){\ms{{{\chr}}_P}}\put(-19,13){\ms{{{\chr}}_Q}}
    \put(-38,36){\ms{x_i}}\put(-38,-11){\ms{x^i}} \quad=\quad\sum_i
    \epsh{fig4c}{24ex}\put(-19,-9){\ms{{{\chr}}_P}}\put(-19,36){\ms{{{\chr}}_Q}}
    \put(-40,12){\ms{{x^*}_i}}\put(-40,-34){\ms{{x^*}^i}} \quad=\quad
    \epsh{fig4d}{24ex}\put(-19,36){\ms{{{\chr}}_Q}}\put(0,-44){\ms{P}}
$$
where ${x^*}_i$ and ${x^*}^i$ are the dual basis obtained by   ${x^*}^i=(x_i)^*\circ(\phi_{G'}\otimes\Id_{G^*})$ and ${x^*}_i=(\phi_{G'}^{-1}\otimes\Id_{G^*})\circ(x^i)^*$.   Here the first and third equalities follow from  \eqref{eq:chrP}  and the second equality from isotopying  the coupon and applying duality of Lemma~\ref{P:Omega-nat}.
Second, if the modifications are made on opposite sides of the red curve, then (with implicit summation):
$$
    \epsh{fig5a}{24ex}\put(-12,0){\ms{{{\chr}}_P}}\put(-38,-2){\ms{Q}}
    \quad=\quad
    \epsh{fig5b}{24ex}\put(-12,-4){\ms{{{\chr}}_P}}\put(-48,18){\ms{{{\chr}}_Q}}
    \put(-65,-5){\ms{x^i}}\put(-65,38){\ms{x_i}}
    \quad=
    \epsh{fig5c}{24ex}\put(-57,14){\ms{{{\chr}}_Q}}\put(-14,22){\ms{{{\chr}}_P}}
    \put(-75,34){\ms{x_i}}\put(-19,-11){\ms{x^i}}
    \quad=
    \epsh{fig5d}{24ex}\put(-55,14){\ms{{{\chr}}_Q}}\put(-11,22){\ms{{{\chr}}_P}}
    \put(-63,34){\ms{\wt x_i}}\put(-24,-11){\ms{\wt x^i}}
    \quad=\quad
    \epsh{fig5e}{18ex}\put(-37,24){\ms{{{\chr}}_Q}}\put(0,27){\ms{P}}
$$
where $\wt x_i$ and ${\wt x}^i$ are the dual basis obtained from   $x_i$ and $x^i$ by the rotation property of Lemma~\ref{P:Omega-nat}.
\end{proof}
\begin{remark}\label{rk:dim-cat}
  If $\cc$ is semisimple,  then  applying Lemma \ref{prop:redtoblue} to a red
  unknot with $P=\un$ implies that $\tr(\chr_\un)$ does not depend of the chromatic map 
  $\chr_\un$ based on $\un$.
\end{remark}
The next lemma shows the usefulness of bichrome graphs.

\begin{lemma}\label{P:slidding}
A blue strand can be slid over a red curve of an admissible bichrome graph in $\Skein_{\Proj_\cc}(\Sigma)$.
\end{lemma}
\begin{proof}
We first consider the case where we want to slide a strand colored by $P\in\Proj_\cc$ over a red curve.  Then we have the following skein relations:
$$
\epsh{fig8a}{16ex}\put(-9,26){\ms{P}}
  =\epsh{fig8b}{18ex}\put(-20,-22){\ms{{{\chr}}_P}}
  =\epsh{fig8c}{24ex}\put(-33,-37){\ms{{{\chr}}_P}}
  \put(-21,18){\ms{{{\chr}}_{P^*}}}
  \put(-43,37){\ms{x_i}}\put(-34,-12.5){\ms{x^i}}
  =\epsh{fig8d}{24ex}\put(-33,-37){\ms{{{\chr}}_P}}
  \put(-21,19){\ms{{{\chr}}_{P^*}}}
  \put(-46,38){\ms{x^{*i}}}\put(-36,-11){\ms{{{x^*}_i}}}
  =\epsh{fig8e}{24ex}\put(-31,-15){\ms{{{\chr}}_P}}
  \put(-20,37){\ms{{{\chr}}_{P^*}}}
  \put(-40,-39.5){\ms{{x^{*i}}}}\put(-34,8.5){\ms{{x^*}_i}}
  =\epsh{fig8f}{18ex}\put(-21.5,26){\ms{{{\chr}}_{P^*}}}
  =\epsh{fig8g}{16ex}\put(-9,-26){\ms{P}}
$$
where ${x^*}_i$ and ${x^*}^i$ are the dual basis defined by ${x^*}^i=(x_i)^*\circ(\phi_{G}\otimes\Id_{P^*\otimes {G}^*})$ and ${x^*}_i=(\phi_{G}^{-1}\otimes\Id_{P^*\otimes {G}^*})\circ(x^i)^*$.  
Next,  consider the general case where we want to slide a strand colored by $Y \in \cc$ over a red curve. Applying the procedure explained in the proof of Lemma~\ref{prop:ddd}, we can push a strand colored by $P \in \Proj_\cc$ next to the $Y$-colored strand. Inserting coupons colored by identities,  we replace the $Y$-colored arc we want to slide by an arc colored by $Y\otimes P\in\Proj_\cc$ which we then slide  over the red curve. By removing then the inserted coupons, we obtain the desired result.
\end{proof}

\subsection{Proof of Theorem~\ref*{T:DiskRmt}} \label{sect-proof-DiskRmtn}
We prove the right version of the first statement of Theorem~\ref{T:DiskRmt} (the left version being analogous).  We associate to any $T\in \Skein_\ideal(D^2)^*$ a family $\mt^T=\{\mt^T_X\co \End_\cat(X)\to \FK\}_{X\in \ideal}$ of linear forms as follows: for any $f\in \End_\cat(X)$ with $X \in \ideal$, set $$\mt^T_X(f)=T(O_f)$$ where $O_f$ is the admissible graph in $D^2$ given by the right closure of the coupon colored with $f$. Let us prove that $\mt^T$ is a right m-trace on $\ideal$. First, since a coupon colored with $f\circ g$ is $\ideal$-skein  equivalent to a coupon colored with $f$ composed with a coupon colored with $g$, we get that $O_{f\circ g}$ is skein equivalent to $O_{g\circ f}$  via an isotopy which exchanges $f$ and $g$:
$$
\epsw{Ographb}{2cm}\put(-51,2){\ms{g\circ f}}
    =\epsw{Ograph2b}{2cm}\put(-45,11){\ms g}\put(-45,-9){\ms f}
    =\epsw{Ograph2b}{2cm}\put(-45,11){\ms f}\put(-45,-9){\ms g}
    =\epsw{Ographb}{2cm}\put(-51,2){\ms{f\circ g}} \;.
$$
Therefore $\mt^T$ satisfies the cyclicity property of an m-trace. Next, for any $f\in \End_\cat(X\otimes Y)$ with $X \in \ideal$ and $Y \in \cat$, the admissible graph $O_{f}$ is skein equivalent to the closure of a coupon colored with $f$ with two incoming and outgoing
arcs colored with $X$ and $Y$ :
$$\epsh{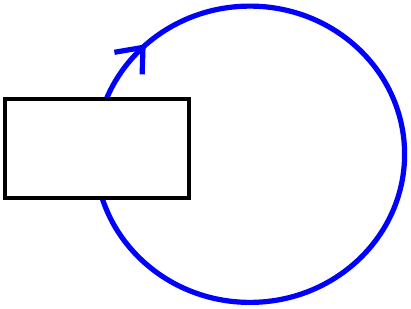}{10ex}
\putc{22}{52}{$\ms{f}$}
\putrc{27}{85}{$\ms{X\otimes Y}$}
 =\epsh{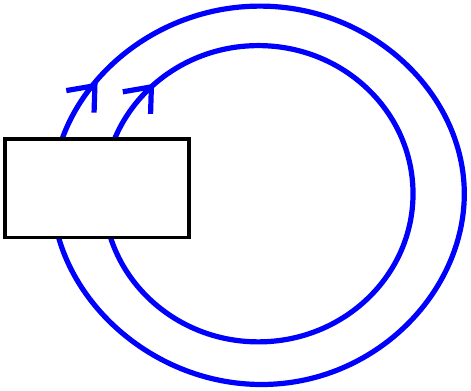}{10ex}
\putc{20}{52}{$\ms{f}$}
\putrc{12}{77}{$\ms{X}$}
\putlc{36}{74}{$\ms{Y}$}
\;.
 $$
This shows that $\mt^T$ satisfies the right partial trace property of an m-trace. Then the  assignment $T \mapsto \mt^T$ is a $\kk$-linear homomorphism $\Skein_\ideal(D^2)^*\to \{\text{right m-traces on } \ideal\}$.

Conversely, we associate to any right m-trace $\mt$ on $\ideal$ an element of $F'_\mt\in\Skein_\ideal(D^2)^*$ as follows.
Let $\Gamma$ be an $\ideal$-admissible graph in $D^2$.  A \emph{cutting path} for $\Gamma$ is any embedding
$\gamma\co [0,1]\to D^2$ starting from a boundary point of $D^2$ and  ending in any point in the interior of $D^2\setminus \Gamma$ such
 that the following three conditions hold:  $\gamma$ does not meet any
 coupon of $\Gamma$, $\gamma$ is transverse to the strands of $\Gamma$, and  $\gamma$ intersects at least one $\ideal$-colored
strand of $\Gamma$.  The complement of a tubular neighborhood of $\gamma$ is a coupon $Q_\gamma$
 whose bottom and top correspond to the left and right side of  $\gamma$, respectively. Then $\Gamma_\gamma=\Gamma\cup Q_\gamma$ can be seen as a $\cc$-colored ribbon graph in $\R \times [0,1]$)
and $F(\Gamma_\gamma)\in\End_\cat(X_\gamma)$ with $X_\gamma\in\ideal$ (because $\gamma$ intersects an $\ideal$-colored strand).  Set
$$F'_\mt(\Gamma)=\mt_{X_\gamma}(F(\Gamma_\gamma)) \in \FK.$$ 
Let us prove that $F'_\mt(\Gamma)$ is independent of the choice of $\gamma$. Pick another cutting path $\gamma'$ for $\Gamma$. Up to slightly isotopying $\gamma'$, we can assume that $\gamma$ and $\gamma'$ intersect transversely in a finite number $n$ of points. We show that $\mt_{X_\gamma}(F(\Gamma_\gamma))=\mt_{X_{\gamma'}}(F(\Gamma_{\gamma'}))$ by induction on $n$:\\

\noindent
\textbf{Case $n=0$.} We first locally modify $\Gamma$ so that all its intersection points
with $\gamma$ and $\gamma'$ are positive:
Away from the tubular neighborhood of $\gamma\cup\gamma'$ the graph
$\wt\Gamma$ is just $\Gamma$.  For each intersection point $p$ of
$\gamma$ or $\gamma'$ with~$\Gamma$, let $e$ be a small segment of the
strand of $\Gamma$ near $p$.  If the orientation of this intersection is
negative (with respect to the orientation of the $D^2$), then we
replace $e$ with a segment containing two coupons colored with
identities joined by an edge crossing $\gamma$ or $\gamma'$ positively
and colored by the dual color of $e$:
$$\epsw{fig8aSkein}{18ex}\put(-4,7){\ms{\gamma}}
\put(-76,0){\ms{X_1}}\put(-44,0){\ms{X_2}}\put(-22,0){\ms{X_3}}
\qquad\longrightarrow\qquad\epsw{fig8bSkein}{18ex}
\put(-79,-4){\ms{X_1^*}}\put(-44,0){\ms{X_2}}\put(-23,-10){\ms{X_3^*}}
\put(-73,20){\ms{\Id}}\put(-28,-24){\ms{\Id}}
\put(-73,-23){\ms{\Id}}\put(-28,21){\ms{\Id}} \;.
$$
Clearly $\Gamma$ and $\Gamma'$ are skein equivalent, $F(\Gamma_\gamma)=F(\wt\Gamma_\gamma)$, and $F(\Gamma_{\gamma'})=F(\wt\Gamma_{\gamma'})$. Thus up to replacing $\Gamma$ with~$\wt\Gamma$, we can assume that all the crossings of $\gamma$ or $\gamma'$ with $\Gamma$ are positive.
In this case, the intersection of~$\Gamma$ with the complement of a tubular neighborhood of $\gamma\cup\gamma'$ can be seen as a $\cc$-colored ribbon graph $\Gamma_{\gamma\cup\gamma'}$  in $\R \times [0,1]$ whose left partial closure is $\Gamma_\gamma$ and right partial closure is   the $\pi$-rotation $\mathrm{rot}_\pi(\Gamma_{\gamma'})$ of $\Gamma_{\gamma'}$.
Note that  $F(\mathrm{rot}_\pi(\Gamma_{\gamma'}))=F(\Gamma_{\gamma'})^*\in \End_\cat(X_{\gamma'}^*)$.
Set $g=F(\Gamma_{\gamma\cup\gamma'})\in \End_\cat(X_{\gamma'}^*\otimes X_{\gamma})$. Then
$$
  \mt_{X_\gamma}(F(\Gamma_\gamma)) \overset{(i)}{=}  \mt_{X_\gamma}\Big(\ptr_l^{X_{\gamma'}^*}(g)\Big)
 \overset{(ii)}{=} \mt_{X_{\gamma'}^{**}}\Big((\ptr_r^{X_\gamma}(g))^*\Big)
    \overset{(iii)}{=}\mt_{X_{\gamma'}^{**}}\Big(\big(F(\mathrm{rot}_\pi(\Gamma_{\gamma'}))\big)^*\Big)
   \overset{(iv)}{=}\mt_{X_{\gamma'}}(F(\Gamma_{\gamma'})).
$$
Here  $(i)$ and $(iii)$ follow from the definition of $\Gamma_{\gamma\cup\gamma'}$, $(ii)$ from Lemma~4.b of \cite{GPV13} which can be restated as $\mt_{U}(\ptr_L^{V^*}(g))=  \mt_{V^{**}}((\ptr_R^{U}(g))^*)$
for all $g\in \End_\cat(V^*\otimes U)$ with $U,V\in \ideal$, and $(iv)$ from the fact that $\mt_{U^{**}}(f^{**})=\mt_U(f)$ for all $f\in\End_\cat(U)$ with $U \in \ideal$.\\

\noindent
\textbf{Inductive case.}  Assume the statement is true for cutting paths intersecting less than $n\geq 1$ times. Let $\gamma$ and~$\gamma'$ be two cutting paths intersecting $n$ times.  We claim that there exists a cutting path $\alpha$ intersecting each of $\gamma$ and $\gamma'$ less than $n$ times, so that by induction we have: $\mt_{X_\gamma}(F(\Gamma_\gamma))=\mt_{X_\alpha}(F(\Gamma_\alpha))=\mt_{X_{\gamma'}}(F(\Gamma_{\gamma'}))$.
Indeed, let $\gamma''$ be the sub-arc of $\gamma$ going from $\partial D$ to the first edge of $\Gamma$ colored by an object of $\ideal$ and then crossing this edge of a small arc.
It is clear that $\gamma''$ is a cutting path for $\Gamma$.   If $\gamma''$ intersects  $\gamma'$ less than~$n$ times, then we can push $\gamma''$ slightly to either side of $\gamma$ to obtain the cutting path $\alpha$. Assume that~$\gamma''$ intersects~$\gamma'$ exactly $n$ times. Let $p$ be the last intersection (in the orientation of $\gamma''$) between $\gamma''$ and $\gamma'$. Consider the arc obtained by following $\gamma'$ until getting to $p$ and then following $\gamma''$ until its end. By pushing this arc slightly to its  right or left (according to the sign of the intersection at $p$ between $\gamma$ and $\gamma'$), we get a cutting arc  $\alpha$  intersecting each $\gamma$ and $\gamma'$ less than $n$ times. This completes the induction.\\
   
\noindent Thus $F'_\mt(\Gamma)$  is independent of $\gamma$. Moreover, $F'_\mt(\Gamma)$ depends only on the class of $\Gamma$ in $\TSkein_\ideal(D^2)$ since one can always find a cutting path avoiding any box involved in an $\ideal$-admissible skein relation.  Consequently, the linear form $F'_\mt\in\Skein_\ideal(D^2)^*$ is well defined. Then the assignment $\mt \mapsto F'_\mt$ is a $\kk$-linear homomorphism $\{\text{right m-traces on } \ideal\} \to \Skein_\ideal(D^2)^*$. It follows from their construction that this homomorphism and the above homomorphism $T \in \Skein_\ideal(D^2)^*\to \mt^T \in \{\text{right m-traces on } \ideal\}$ are inverse of each other, thus proving the first isomorphism of the theorem.

Let us now consider the spherical case.  Any linear form $T$ on
$\Skein_\ideal(S^2)$ induces a right m-trace $\mt^T$ defined on
morphisms by $\mt^T_X(f)=T(O_f)$ as above (except that the graph $O_f$
is now in $S^2$).  Since $O_f$ and $O_{f^*}$ are isotopic in $S^2$,
this right m-trace is equal to its dual, and \cite[Lemma 3]{GPV13} implies this is an m-trace.

Reciprocally, let $\mt$ be a m-trace on $\ideal$. It is in
particular a right m-trace and so defines $F'_\mt \in \Skein_\ideal(D^2)^*$.  For any  $\ideal$-admissible graph $\Gamma$ in $S^2$ and any $p\in S^2\setminus \Gamma$, view $\Gamma_p=\Gamma$ as an $\ideal$-admissible graph in $S^2\setminus\{p\} \cong D^2$ and set
$F'_\mt(\Gamma,p)=F'_\mt(\Gamma_p)$. 
We claim that $F'_\mt(\Gamma,p)$ does
not depend on the choice of the point $p$. Consider a cutting path in $S^2$ for $\Gamma$, that is, a path
$\gamma$ starting from a point $p_1\notin\Gamma$ and ending at a
point $p_2\notin\Gamma$, meeting no coupons of $\Gamma$, transverse to
the strands of $\Gamma$, and intersecting at least one $\ideal$-colored
edge of $\Gamma$.  Let $\bar{\gamma}$ be the inverse
path from $p_2$ to $p_1$.  The path $\gamma$ induces a cutting path in
$D^2$ that can be used to compute $F'_\mt(\Gamma,p_1)$. Similarly,
$F'_\mt(\Gamma,p_2)$ can be computed using the cutting path in $D^2$
induced by $\bar{\gamma}$.  Since the two coupons obtained
by cutting along $\gamma$ and $\bar{\gamma}$ are related
by a $\pi$-rotation and since $\mt$ is an m-trace, we obtain that $F'_\mt(\Gamma,p_1)=F'_\mt(\Gamma,p_2)$. Hence
$F'_\mt(\Gamma)=F'_\mt(\Gamma,p)$ is well defined by choosing any
point $p$.  Finally if an $\ideal$-admissible skein relation in $S^2$
is defined with a coupon, then we can choose for all involved graphs the
same point $p$ outside the coupon, so the skein relation comes from a
skein relation in $D^2$ on which $F'_\mt$ vanishes.

\section{Non-compact TQFTs from chromatic categories}\label{S:TQFT}

Throughout this section, $\cat$ is a chromatic category. We associate to $\cat$ a non-compact (2+1)-TQFT which extends the skein module functor from Theorem~\ref{T:MappingClassActs}. Our construction is based on Juh\'asz's presentation of cobordisms (\cite{Juh2018}).

\subsection{Non-compact (2+1)-TQFTs}
Let $\cob$ be the category whose objects are closed oriented (smooth) surfaces and morphisms are equivalence classes (up to orientation preserving diffeomorphisms preserving the boundary parameterizations)
of oriented 3-dimensional (smooth) cobordisms. Composition in $\cob$ is induced by the gluing of cobordisms (along their common boundary). The category $\cob$ is symmetric monoidal with monoidal product induced by the disjoint union and monoidal unit the empty surface. Denote by $\Vect_\FK$ the symmetric monoidal category of $\FK$-vector spaces and $\FK$-linear homomorphisms.
A \emph{(2+1)-TQFT} is a symmetric monoidal functor $\cob \to \Vect_\FK$. Note that any such TQFT always takes values in the subcategory of finite dimensional vector spaces (since $\cob$ is rigid).

Let $\cob^\nc$ be the largest subcategory of $\cob$
such that each component of every cobordism has a nonempty source. The category $\cob^\nc$ is a symmetric monoidal subcategory of $\cob$.
 A \emph{non-compact (2+1)-TQFT} is a symmetric monoidal functor $\cob^\nc \to \Vect_\FK$.
 A non-compact (2+1)-TQFT is \emph{finite dimensional} if it takes values in the subcategory of finite dimensional vector spaces.

\subsection{Generators of $\cob$ and $\cob^\nc$}\label{S:GenCobStat}
In \cite{Juh2018}, Juh\'asz gives a presentation of $\cob$
whose generators $\{e_{\Sigma,\Sp},e_d\}$ are indexed by framed
$k$-spheres $\Sp$ in a surface $\Sigma$ and diffeomorphisms $d\co \Sigma \to \Sigma'$ between surfaces, see Section \ref{SS:JuhaszPres}.
These generators correspond to $k+1$-handles and mapping cylinders that we now
describe.

Let $\Sigma$ be an oriented surface. For
$k\in \{0,1,2\}$, a \emph{framed k-sphere} in $\Sigma$ is an orientation reversing embedding $\Sp\co S^k \times D^{2-k} \hookrightarrow \Sigma$.
Then we can perform surgery on $\Sigma$ along $\Sp$ by removing the interior of the image of~$\Sp$ and gluing in $D^{k+1} \times S^{1-k}$, getting a well defined topological manifold $\Sigma(\Sp)$ which, using the framing of the sphere, can be endowed with a canonical smooth structure.  The associated oriented cobordism $(\Sigma \times [0,1]) \cup_\Sp (D^{k+1}\times D^{2-k}) $ represents a morphism $W(\Sp)$ in $\cob$ from $\Sigma \to \Sigma(\Sp)$.
 Juh\'asz considers two additional types of framed sphere,  namely $\Sp=0$ and
$\Sp=\emptyset$, where $\Sigma(0)=\Sigma \sqcup S^2$ and
$\Sigma(\emptyset)=\Sigma$  with associated the cobordisms $W(0)=\Sigma\times[-1,1]\sqcup D^3 \co \Sigma \to \Sigma(0)$ and $W(\emptyset)=\Sigma\times [-1,1]\co \Sigma \to \Sigma(\emptyset)$.

Finally, recall that any orientation preserving diffeomorphism
$d\co \Sigma \to \Sigma'$ between closed oriented surfaces gives rise
to the morphism $c_d \co \Sigma \to \Sigma'$ in $\cob$ represented by the cylindrical cobordism  whose underlying manifold
is $\Sigma \times [0,1]$ with boundary
$(-\Sigma \times \{0\}) \sqcup (\Sigma \times \{1\})$ parameterized by
$(x,0) \mapsto x$ and $(x,1) \mapsto d(x)$ for all $x \in \Sigma$.  In
Juh\'asz's presentation, the formal generators $e_{\Sigma,\Sp}$ and
$e_d$ correspond to the above cobordisms $W(\Sp)$ and $c_d$
respectively.

The generators of $\cob^\nc$ are the same with exception of those
associated with the formal spheres $\Sp=0$ since the cobordisms
$W(0)$ do not belong to $\cob^\nc$.

\subsection{Construction of the non-compact TQFT}
The admissible skein module functor associated with the ideal $\Proj_\cc$ of projective objects of $\cc$ (see Theorem~\ref{T:MappingClassActs}) induces (by restriction) a monoidal functor
\begin{equation}\label{eq-skein-funct-man}
\Skein_{\Proj_\cc}\co\man\to \Vect_\FK,
\end{equation}
where $\man\subset \cob^\nc$ is the category of closed oriented surfaces and orientation preserving diffeomorphisms. 
Our goal is to extend it to a functor $\TSkein\co \cob^\nc\to \Vect_\FK$. In particular, for any 
closed oriented surface $\Sigma$ and any orientation preserving diffeomorphism $d\co \Sigma \to \Sigma'$ between closed oriented surfaces, we set
$$
\TSkein(\Sigma)=\Skein_{\Proj_\cc}(\Sigma) \quad  \text{and} \quad \TSkein(\Sigma)(e_d)=\Skein_{\Proj_\cc}(d)\co \TSkein(\Sigma) \to \TSkein(\Sigma').
$$
We need to assign values to the other generators of $\cob^\nc$. More precisely, given a nonempty closed oriented surface $\Sigma$ and a framed sphere $\Sp$ in $\Sigma$, we need to assign a $\FK$-linear homomorphism $\TSkein(e_{\Sigma, {\Sp}}) \co \TSkein(\Sigma) \to \TSkein(\Sigma(\Sp))$
in the case $\Sp=\emptyset$ or $\Sp=\Sp^k$ is a framed $k$-sphere with $k\in\{0,1,2\}$:\\

\noindent
$\bullet$ \underline{\textbf{Case $\mathbf{\Sp=\emptyset}$:}} We set
$$
\TSkein(e_{\Sigma, \emptyset})=\Id_{\TSkein(\Sigma)}.
$$
\smallskip

\noindent
$\bullet$ \underline{\textbf{Case $\mathbf{\Sp=\Sp^0}$:}}
Consider the disjoint embedded disks $D$ and $D'$ in $\Sigma$ given by the a framed 0-sphere ${\Sp^0}$.
Set $\Sigma'=\Sigma\setminus (D\sqcup D')$ and let $C\simeq S^1\times [0,1]$ be the cylinder such that $\Sigma({\Sp^0})=\Sigma'\cup_\partial C$.  Set $\gamma=S^1\times\{\frac12\}$ be a red curve inside~$C$. Let $\Gamma$ be an admissible graph in $\Sigma$. Slightly isotopying $\Gamma$ away from $D$ and $D'$, we obtain an admissible graph $\Gamma'$ in $\Sigma'$. Then $\Gamma'\cup \gamma$ is a bichrome graph in $\Sigma(\Sp^0)$:
\begin{gather*}
    \Gamma=\epsh{fig9e}{10ex}
    \quad \rightsquigarrow\quad \Gamma'= \epsh{fig9b}{10ex}\\
\Gamma'\cup \gamma = \epsh{fig9c}{16ex}\put(-90,20){$\gamma$}
\end{gather*}
By Lemma~\ref{prop:redtoblue}, the  bichrome graph $\Gamma'\cup \gamma$ defines an element in $\TSkein(\Sigma(\Sp^0))$.
\begin{lemma}\label{L:S0}
 The element  $\TSkein(e_{\Sigma, \Sp^0})(\Gamma)=\Gamma'\cup\gamma\in\TSkein(\Sigma({\Sp^0}))$
  only depends on the framed sphere $\Sp^0$ and the class of $\Gamma$ in $\TSkein(\Sigma)$.
\end{lemma}

 \begin{proof}
If $\Gamma_1'$ and $\Gamma_2'$ are two preimages of $\Gamma$ isotopic
  in $\Sigma$ by an isotopy during which an edge passes over the disk $D$ or $D'$,
  then by the sliding property of Lemma \ref{P:slidding}, we
  have
  $(\Sigma',\Gamma'_1)\cup_\partial(C,\gamma)=
  (\Sigma',\Gamma'_2)\cup_\partial(C,\gamma)\in\TSkein(\Sigma({\Sp^0}))$.
  Any isotopy in $\Sigma$ can be modified so that no coupons of
  $\Gamma$ pass through ${\Sp^0}$.  Finally, any skein relation in
  $\Sigma$ is isotopic to a skein relation in a box that does not
  intersect ${\Sp^0}$ which induce a corresponding skein relation between
  $(\Sigma',\Gamma'_1)\cup_\partial(C,\gamma)$ and
  $(\Sigma',\Gamma'_2)\cup_\partial(C,\gamma)$ in $\TSkein(\Sigma)$.
Remark that interchanging $D$ and $D'$ does not change $\TSkein(e_{\Sigma, \Sp^0})(\Gamma)$.
\end{proof}

\noindent
$\bullet$ \underline{\textbf{Case $\mathbf{\Sp=\Sp^1}$:}} Given a framed 1-sphere ${\Sp^1}$ in $\Sigma$, let
$\gamma$ be a simple closed curve embedded in $\Sigma$ so that
${\Sp^1}\simeq \gamma\times[-1,1]$ in $\Sigma$. We fix an orientation and a base point $*$ on $\gamma$. Let $\Gamma$ be an admissible graph in $\Sigma$. Isotopying $\Gamma$, we can assume that $\Gamma$ is transverse
to ${\Sp^1}$ in the sense that ${\Sp^1}\cap\Gamma$ consists in a finite number of portions of
edges of $\Gamma$ in position $\gamma(t_i)\times[-1,1]$ for $t_i\neq*$
and with at least one intersecting edge colored by a projective
object. We define $\TSkein(e_{\Sigma, \Sp^0})(\Gamma)$
to be the  admissible graph in $\Sigma({\Sp^1})$ obtained from
$(\Sigma,\Gamma)\setminus {\Sp^1}$ by filling the two attached discs with
two coupons colored with dual basis (see Section~\ref{SS:trace}):
\begin{equation}\label{E:CuttingOnGreen}
\Gamma=\epsh{fig6c}{10ex}\quad \mapsto\quad
\TSkein(e_{\Sigma, \Sp^1})(\Gamma)=\sum_i\epsh{fig6d}{10ex}\put(-27,-7){{$x_i$}}\put(-69,7){{$x^i$}}\;\;.
\end{equation}

\begin{lemma}\label{L:S1}
The element $\TSkein(e_{\Sigma, \Sp^1})(\Gamma)$  only depends on the framed sphere $\Sp^1$ and the class of $\Gamma$ in
  $\TSkein(\Sigma)$.
\end{lemma}
\begin{proof} 
  If $\Gamma_1$ and $\Gamma_2$ are isotopic in $\Sigma$, with an
  isotopy where no strand
  passes through the base point and no coupon passes
  through $\gamma$, then $\TSkein(e_{\Sigma, \Sp^1})(\Gamma_1)$ and
  $\TSkein(e_{\Sigma, \Sp^1})(\Gamma_2)$ are isotopic in $\Sigma({\Sp^1})$.  
   Assume first that a
  coupon crosses $\gamma$. Assertion~(b) of Lemma \ref{P:Omega-nat} implies that there are two
  coupons $Q_1$ and $Q_2$ such that $F(Q_1)\otimes_\FK F(Q_2)=0$, where $F$ is the functor given in \eqref{eq-defF}.
  Thus, by Lemma \ref{L:2boxes},  the
  difference $\TSkein(e_{\Sigma, \Sp^1})(\Gamma_1)-\TSkein(e_{\Sigma, \Sp^1})(\Gamma_2)$ is a sum of skein relations in $\Sigma({\Sp^1})$.
 Next, Assertions~(a) and~(c) of Lemma~\ref{P:Omega-nat} imply respectively that $\TSkein(e_{\Sigma, \Sp^1})(\Gamma)$ is invariant under the change of the  orientation of $\gamma$ and under the change of the base point on~$\gamma$.   Hence
  $\TSkein(e_{\Sigma, \Sp^1})(\Gamma)$ only depends of the isotopy class of
  $\Gamma$ in $\Sigma$.  Since any skein relation in $\Sigma$ can be
  isotoped to a skein relation involving a coupon disjoint from ${\Sp^1}$, it
  induces an equivalent skein relation inside $\Sigma({\Sp^1})$. 
\end{proof}

\noindent
$\bullet$ \underline{\textbf{Case $\mathbf{\Sp=\Sp^2}$:}} A framed 2-sphere ${\Sp^2}$ in $\Sigma$ determines a spherical component of $\Sigma$ denoted $S^2$.  Recall from Theorem \ref{T:DiskRmt} that the m-trace $\qt$ induces a linear form
\begin{equation}
  \label{eq:F'}
  F'\co\TSkein(S^2)\to \FK.
\end{equation}
Any admissible graph $\Gamma$ in $\Sigma$ decomposes as  $\Gamma=\Gamma_1\sqcup\Gamma_2$ with $\Gamma_1 \subset \Sigma(\Sp^2)$ and $\Gamma_2=\Gamma\cap{\Sp^2}$.
Then the element
$$
\TSkein(e_{\Sigma, \Sp^2})(\Gamma)=F'(\Gamma_2)\Gamma_1 \in\TSkein(\Sigma({\Sp^2}))
$$
only depends on the framed sphere $\Sp^2$ and the class of $\Gamma$ in   $\TSkein(\Sigma)$.
  
\begin{theorem}\label{T:S}
  Recall, $\cat$ is a chromatic category.  The above assignments
  define a finite dimensional non-compact (2+1)-TQFT
  $$\TSkein\co \cob^\nc\to\Vect_\FK.$$ Furthermore $\TSkein$
  (uniquely) extends to a genuine (2+1)-TQFT $\cob\to\Vect_\FK$ if and
  only if $\cat$ is semisimple with nonzero dimension
  (see Section~\ref{sect-chr-cat}).
\end{theorem}

We prove Theorem~\ref{T:S} in Section~\ref{sect-proof-T:S} using a presentation of $\cob^\nc$ given in Section \ref{SS:JuhaszPres}. 

By construction, the non-compact TQFT of Theorem~\ref{T:S} extends the skein module functor \eqref{eq-skein-funct-man}. Also, it follows from the work of Bartlett \cite{Bart2022} that if $\cc$ is a spherical fusion category with nonzero dimension (see Example~\ref{ex-chromatic-spherical-fusion}), then the (2+1)-TQFT associated with $\cc$ by Theorem~\ref{T:S} is isomorphic to the Turaev-Viro TQFT associated with $\cc$.

The next corollary is a direct consequence of Theorems~\ref{T:SphericalChrom} and \ref{T:S}:
\begin{corollary}
Any spherical tensor category over an algebraically closed field defines a finite dimensional non-compact (2+1)-TQFT.
\end{corollary} 

The next theorem relates the TQFT  $\TSkein$ of with the spherical chromatic invariant ${\mathcal K}_\cat$ of closed oriented 3\trait manifolds defined in \cite{CGPT20}.

\begin{theorem}\label{T:Kcat}
Recall, $\cat$ is a chromatic category.  Let $M$ be a closed connected oriented 3-manifold. Consider ${\dot M}=M\setminus \Int(B^3)\co S^2\to \emptyset$ and  ${\ddot M}=M\setminus \Int(S^0\times B^3)\co S^2\to S^2$.   Then
$$
  \TSkein({\dot M})=\Kup_\cat(M)F',
$$
where $F'$ is given by \eqref{eq:F'}.  In particular, if the
m-trace of $\cc$ is unique (up to scalar multiple,  see \cite{GKP22}), then
$\dim_\FK(\TSkein(S^2))=1$ and so
$\TSkein({\ddot M})=\Kup_\cat(M)\Id_{\TSkein(S^2)}.$
\end{theorem}
\begin{proof}
Recall that  $\Kup_\cat(M)$ is defined from a graph formed by an oriented circle $o_G$ colored by a projective
  generator $G$ of $\cc$ and endowed with a coupon colored by an
  endomorphism $h\in \End_\cc(G)$ such that $\qt_G(h)=1$.
To prove the theorem we need to show
  $\TSkein({\dot M})(o_G)=\Kup_\cat(M)$.  Choosing a Heegaard
  splitting of $M$ we see that $\dot M$ is diffeomorphic to the
  cobordism
  $W(\Sp^2)\circ W(\Sp^1_{g})\circ W(\Sp^1_{g-1})\circ
  W(\Sp^1_{1})\circ W(\Sp^0_{g})\circ \cdots \circ W(\Sp^0_1)$ for
  some disjoint $\Sp^0$ framed spheres $\Sp^0_1,\ldots \Sp^0_g$ on
  $S^2$ and some disjoint $\Sp^1$ spheres $\Sp^1_1,\ldots ,\Sp^1_g$ on
  the surface $\Sigma$ obtained after doing surgery in $S^2$ on all of
  $\Sp^0_1,\ldots ,\Sp^0_g$.  The overall composition is then obtained
  by first applying $g$ times Lemma \ref{prop:redtoblue} to the $g$
  red curves created by the first $g$ spheres $\Sp^0_i$ with $i=1,\ldots g$,
then applying $g$ times the cutting map of Equation
  \eqref{E:CuttingOnGreen}, once for each $\Sp^1_i$ with $i=1,\ldots g$, and
  finally applying $F'$.  This is exactly the result of the Kuperberg
  invariant defined in \cite{CGPT20}. Indeed the handlebody $H$ of
  Theorem 2.5 of \cite{CGPT20} is
  $B^3\circ W(\Sp^1_{g})\circ W(\Sp^1_{g-1})\circ W(\Sp^1_{1})$ seen
  as a cobordism from $\Sigma$ to $\emptyset$, the red graph $\Gamma$
  is $\Sp^1_1\sqcup \ldots \sqcup\Sp^1_g\subset \Sigma$ and the blue
  graph is $o_G$.
\end{proof}
An easy consequence of the previous theorem is the following:
\begin{corollary} 
If the m-trace of $\cc$ is unique (up to scalar multiple), then the 3-manifold invariant $\Kup_\cat$
  is multiplicative with respect to connected sums.
\end{corollary}
\begin{proof}
Let $M_1,M_2$ be closed connected oriented 3-manifolds and denote by $M=M_1\sharp M_2$ their connected sum. We have: $\ddot{M}=\ddot M_1\circ\ddot M_2\in\cob^\nc$. Then it follows from Theorem~\ref{T:Kcat} and the functoriality of $\TSkein$ that
  $\Kup_\cat(\ddot M)\Id_{\TSkein(S^2)}=\TSkein(\ddot M)=\TSkein(\ddot M_1)\circ\TSkein(\ddot
  M_2)=\Kup_\cat(M_1)\Kup_\cat(M_2)\Id_{\TSkein(S^2)}$.
\end{proof}

\subsection{Juh\'asz's presentation of $\cob$ and $\cob'$}\label{SS:JuhaszPres}
Following \cite{Juh2018}, we consider the subcategory $\cob'$ of
cobordism such that each component of every cobordism has a nonempty source and nonempty target.
 Here, we consider the empty surface as an
object of $\cob'$.

Let $\mathcal{G}$ be the directed graph described as follows. The
vertices are closed oriented surfaces.  There are two kinds of edges
of $\mathcal{G}$.  First, for each orientation preserving
diffeomorphism $d\co \Sigma \to \Sigma'$ between closed oriented
surfaces, there is an edge $e_d$ going from $\Sigma$ to
$\Sigma'$. Second, for each framed sphere $\Sp$ in a closed oriented
surface $\Sigma$, there is an edge $e_{\Sigma,\Sp}$ from $\Sigma$ to
$\Sigma(\Sp)$.  Let $\mathcal{G}^\nc$ (resp. $\mathcal{G'}$) be the
subgraph of $\mathcal{G}$ obtained by removing the empty surface and
the edges $e_{\Sigma,\Sp}$ where $\Sp=0$ (resp. where $\Sp=0$ or $\Sp$
is a framed $2$-sphere).  Denote by $\mathcal{F}(\mathcal{G})$
(resp. $\mathcal{F}(\mathcal{G}^\nc)$,
resp. $\mathcal{F}(\mathcal{G}')$) the free categories generated by
$\mathcal{G}$ (resp. $\mathcal{G}^\nc$, resp. $\mathcal{G}'$).

In \cite[Definition 1.4]{Juh2018}, Juh\'asz considers a set of relations $\mathcal{R}$ in $\mathcal{F}(\mathcal{G})$ which we recall now.  If $w$ and $w'$ are words consisting of composable
arrows, then we write $w\sim w'$ if
$w=w'$ is a relation in $\mathcal{R}$.

\begin{enumerate}
\item[(R1)] For composable diffeomorphisms $d$ and $d'$ between closed oriented surfaces, we have the relation $e_{d \circ d'} \sim e_d\circ e_{d'}$.  We also have the relations $e_{\Sigma,\emptyset}\sim e_{\Id_\Sigma}$ and $e_d\sim  e_{\Id_\Sigma}$ if $d\co \Sigma \to \Sigma$ is a diffeomorphism isotopic to the identity.
\item[(R2)]  Let  $d \co \Sigma \to \Sigma'$ be an orientation preserving diffeomorphism between closed oriented surfaces and $\Sp$ be  a framed sphere in $\Sigma$.  Consider the framed sphere $\Sp'=d\circ \Sp$ in $\Sigma'$ and denote by $d^\Sp\co  \Sigma(\Sp)\to \Sigma'(\Sp')$ the induced diffeomorphism.   Then the commutativity of the following diagram defines a relation:
$$\xymatrix{
\Sigma \ar[d]_{e_{d}} \ar[r]^{e_{\Sigma,\Sp}} & \Sigma(\Sp) \ar[d]^{e_{d^\Sp}} \\
\Sigma' \ar[r]^{e_{\Sigma',\Sp'}} & \Sigma'(\Sp')}$$
\item[(R3)]  Let  $\Sp, \Sp'$ be disjoint framed sphere in an oriented surface $\Sigma$.  Notice that
$\Sigma(\Sp)(\Sp')= \Sigma(\Sp')(\Sp)$ and denote this surface by $\Sigma(\Sp,\Sp')$.  The  commutativity of the following diagram defines a relation:
$$\xymatrix{
\Sigma \ar[d]_{e_{\Sigma,\Sp'}} \ar[r]^{e_{\Sigma,\Sp}} & \Sigma(\Sp) \ar[d]^{e_{\Sigma(\Sp),\Sp'}} \\
 \Sigma(\Sp')\ar[r]^{e_{\Sigma(\Sp'),\Sp}} & \Sigma(\Sp,\Sp')}$$

 \item[(R4)]
 Let $\Sp$ be a framed $k$-sphere in  an oriented surface  $\Sigma$ and $\Sp'$ a framed $k'$-sphere in $\Sigma(\Sp)$.  If  the attaching sphere $\Sp'(S^{k'}\times \{0\})\subset \Sigma(\Sp)$ intersects the belt sphere
 $\{0\}\times S^{-k+1} \subset \Sigma(\Sp)$ once transversely, then there is a diffeomorphism (well defined up to isotopy)
 $\phi \co  \Sigma \to  \Sigma(\Sp,\Sp')$  (see \cite[Definition 2.17]{Juh2018})  and the following is a relation:
 $$
 e_{\Sigma(\Sp),\Sp'} \circ  e_{\Sigma,\Sp} \sim e_{\phi}.
 $$
\item[(R5)]  For each be a framed $k$-sphere $\Sp$ in an oriented surface $\Sigma$,  there is a relation $ e_{\Sigma,\Sp}  \sim e_{\Sigma,\bar\Sp} $, where the framed $k$-sphere $\bar \Sp \co S^k \times D^{2-k} \hookrightarrow \Sigma $ is defined by $\bar\Sp(x,y)= \Sp(r_{k+1}(x),r_{2-k}(y))$ for any  $x\in S^k \subset \R^{k+1}$ and $y\in D^{2-k}\subset \R^{2-k}$, with   $r_m(x_1,x_2, \dots, x_m)=(-x_1,x_2, \dots, x_m)$.
\end{enumerate}
Let $\mathcal{R}^\nc$ and $\mathcal{R'}$ be the subset of relations
involving only edges in $\mathcal{G}^\nc$ and $\mathcal{G'}$
respectively.

Following \cite[Definition 1.5]{Juh2018}, let
$c\co \mathcal{G}\to \cob$ be the map which is the identity on
vertices, assigns the cylindrical cobordism $c_d$ to the generator
$e_d$ associated to a diffeomorphism $d$, and assigns the cobordism
$W(\Sp)$ to the edge $ e_{\Sigma, \Sp}$.  This extends to a symmetric
strict monoidal functor $c\co \mathcal{F}(\mathcal{G})\to
\cob$. Recall that given a category $\mathcal{F}$ and a set of
relations $\sim$ on its morphisms, the quotient category
$\mathcal{F}/\!\sim$ has the same objects as $\mathcal{F}$ and
equivalence classes of morphisms of $\mathcal{F}$ as
morphisms. Juh\'asz proved (see \cite[Theorem 1.7]{Juh2018}) that the
functor $c\co\mathcal{F}(\mathcal{G})\to \cob$ induces isomorphisms of
symmetric monoidal categories
$$
\mathcal{F}(\mathcal{G})/\mathcal{R}\to \cob\quad\text{and} \quad \mathcal{F}(\mathcal{G'})/\mathcal{R'}\to \cob'.
$$
As a corollary, we obtain:
\begin{corollary}
  The functor $c\co\mathcal{F}(\mathcal{G^\nc})\to \cob^\nc$ induces
  an isomorphism of symmetric monoidal categories
  $$
  \mathcal{F}(\mathcal{G^\nc})/\mathcal{R^\nc}\to \cob^\nc.
  $$
\end{corollary}
\begin{proof}
This corollary follows from Juhasz's argument in \cite{Juh2018} using parameterized Cerf decomposition.  Here we give an argument 
based on the statements of Juhasz's theorems.  

In this proof, the edges of $\mathcal{G}$ associated to
  3-handles are called \emph{singular} (these are edges in ${\mathcal{G}}^\nc$ but not in $\mathcal{G'}$). 
   If $W\co M\to N$ is a cobordism in $\cob^\nc$ then let $\dot W\co M\to N\sqcup (S^2)^{\sqcup n}$ be a cobordism in $\cob'$ obtained by removing a 3-ball from each connected components of $W$ which is disjoint from $N$.
   Then
  $\dot W\in c(\mathcal{G'})\subset c({\mathcal{G}}^\nc)$ and
  $W=\prod_ic(e_i^3)\circ\dot W\in c({\mathcal{G}}^\nc)$ where the
  $e_i^3$ are singular edges. Hence $c$ is surjective on the morphisms of $\cob^\nc$.

  Now let $w_1,w_2\in{\mathcal G}^\nc$ such that 
  $c(w_1)=c(w_2)=W\in\cob^\nc$.   We can assume (up to adding cancelling
  2-3 handles using relation $R4\in{\mathcal R}^\nc$), that for each
  component of $W$, $w_1$ and $w_2$ contain the same number of
  singular edges corresponding to $3$ handles in the component.
 Then for $j=1,2$ up to modifying $w_j$ via relations $R2,R3\in{\mathcal R}^\nc$ we can assume that all these singular edges 
 are at the end of the word, i.e.\ $w_j\sim\prod_ie_i^3w'_j$ where the $e_i^3$ are
  singular edges and $w'_j\in\mathcal{F}(\mathcal{G'})$ (indeed, the
  target of a singular edge is a 2-manifold where the attaching sphere
  has completely disappeared so there is no possible intersection in the
  boundary with other attaching spheres).  Now
  $c(w'_1)\simeq c(w'_2)\simeq \dot W\in\cob'$ are both diffeomorphic
  to a punctured $W$ with the same number of 3-balls removed in each
  component of $W$, that is up to an isotopy moving the punctures, we
  have $c(w'_1)=c(w'_2)\in \cob'$.   Then by \cite[Theorem 1.7]{Juh2018}) we have 
  $w'_1\overset{\mathcal{R'}}\sim w'_2$ and it follows that
  $w_1\overset{{\mathcal R}^\nc}\sim w_2$.  Thus, $c\co\mathcal{F}(\mathcal{G^\nc})/\mathcal{R^\nc}\to \cob^\nc$ is an isomorphism.
\end{proof}

\subsection{Proof of Theorem~\ref*{T:S}}\label{sect-proof-T:S}
  To prove the first statement of the theorem, we need to show that the relations (R1)-(R5) are satisfied by $\TSkein$.
  \begin{enumerate}
  \item[(R1)]
   Since $\TSkein\co \man\to\Vect_\FK$ is functorial we have  $\TSkein(e_{d\circ d'})=\TSkein(e_{d})\circ \TSkein(e_{d'})$.  Also,  since elements of
    $\TSkein(\Sigma)$ are defined by graphs up to isotopy we
    clearly have $\TSkein(e_{d})=\Id$ if $d$ is isotopic to
    $\Id_\Sigma$.
  \item[(R2)] Since the construction of the maps $\TSkein(e_{\Sigma, \Sp})$ are local,
    they are covariant under diffeomorphisms of the pair
    $({\Sigma,\Sp})$.
  \item[(R3)] Again, since the construction of the maps
    $\TSkein(e_{\Sigma, \Sp})$ are local, they commute for disjoint framed
    spheres.
  \item[(R4)] The 1-2 handle cancellation reduces to the chromatic
    identity \eqref{eq:chrP} as shown in the following picture:
    $$
    \TSkein(e_{\Sigma, \Sp^0})((\Sigma,\Gamma))=\epsh{fig28a}{12ex}\put(-38,-11){\ms{P}}
    = \epsh{fig28b}{12ex}\put(-44,-11){\ms{P}}\put(-89,16){\ms{{\chr}_P}}
    $$
    $$
    (\Sigma,\Gamma)=
    \epsh{fig28d}{12ex}\put(-38,-11){\ms{P}}\stackrel{\eqref{eq:chrP}}=
    \epsh{fig28c}{12ex}\put(-87,18){\ms{x^i}}\put(-87,-02.5){\ms{x_i}}
    \put(-88,09){\ms{{\chr}_P}}\put(-43,-12){\ms{P}}\put(-10,20){${\swarrow}$
      $\TSkein(e_{\Sigma, \Sp^1})$}\qquad\,
    $$
    Here, $\Gamma$ is a skein element in the surface $\Sigma$ with an
    edge colored by $P\in\Proj_\cc$.  On the top left we depict the result
    of a $\TSkein(e_{\Sigma, \Sp^0})$ move which is cancelled
    then by a $\TSkein(e_{\Sigma, \Sp^1})$ (diagonal arrow) where
    the $\Sp^1$ is the green curve on the top right hand side. The
    bottom equality reduces to Equation \eqref{eq:chrP} for $Q=\unit$
    after rotating the coupons colored with the dual basis and
    applying the duality property of Lemma \ref{P:Omega-nat}.

    The 2-3 handle cancellation reduces to a skein relation which
    replaces a skein in a disk whose image by $F$ is
    $f\in\Hom(\unit,P)$ by a unique coupon colored by
    $\sum_i \mt_P(fx^i)x_i=f$.
  \item[(R5)] As stated in the proof of Lemma \ref{L:S0}  interchanging the disks $D$ and $D'$ does not change the map $\TSkein(e_{\Sigma, \Sp^0})$.  This implies that  (R5) is satisfied for any framed 0-sphere.  Similarly, in the proof of Lemma \ref{L:S1} it is shown that the map $\TSkein(e_{\Sigma, \Sp^1})$
 does not depend on the orientation of $\gamma$, implying that (R5) is satisfied for any framed 1-sphere.
  \end{enumerate}

We now prove the second statement of the theorem.  Assume that $\cc$ is semisimple with nonzero dimension (as a chromatic category, see Section~\ref{sect-chr-cat}). To extend $\TSkein$ to a (2+1)-TQFT, we first need to assign the value under $\TSkein$ for the generator $e_{\Sigma, 0} \co \Sigma \to \Sigma(0)=\Sigma \cup S^2$ where $\Sigma$ is an oriented closed surface. Let $\Gamma$ be an admissible graph in $\Sigma$. Consider the graph $\gamma$ in $S^2$ defined by
\[\gamma=\frac1{\dim(\cc)}\epsh{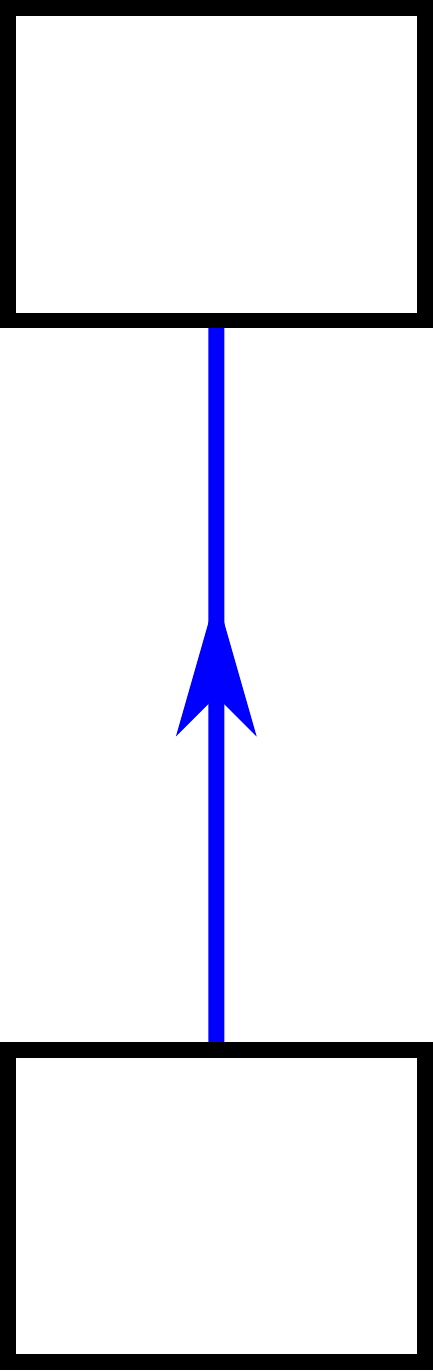}{10ex}
  \putc{48}{88}{$\ms{\id_\un}$}
  \putc{49}{11}{$\ms{\id_\un}$}
  \putlc{68}{50}{$\ms{\un}$}\qquad,\] where $\dim(\cc)$ is the dimension of
$\cat$.  
Then
$$
\TSkein(e_{\Sigma, 0})(\Gamma)=\Gamma\cup\gamma\in\TSkein(\Sigma(0))
$$
only depends on the class of $\Gamma$ in $\TSkein(\Sigma)$.
Next we need to verify that the relation (R4) is satisfied for 0-1-handle cancellation: the result of a 0-handle followed by a cancelling 1-handle sends a skein $\Gamma\in\TSkein(\Sigma)$ to the same graph union the graph $\gamma$ encircled by a red unknot.  Now an admissible skein relation replaces the encircled
$\gamma$ with $\frac1{\dim(\cc)}\tr_\cat(\chr_\un)=1$.

Conversely, assume that $\cat$ is not semisimple or is semisimple with dimension zero. 
We will prove that the 3d-pants cobordism $M\co S^2\sqcup S^2\to S^2$
given by a 3-ball minus two smaller 3-balls is sent to $0$ by
$\TSkein$.  As a consequence, since the cobordism $M$ has a right inverse in
$\cob$ given by $c_{\Id_{S^2}}\sqcup B^3\co S^2\to S^2\sqcup S^2$ (recall the notation introduced at the beginning of Subsection \ref{S:GenCobStat}) and since
$\Id_{\TSkein(S^2)}\neq0$, this implies that $\TSkein$ can
not be extended to a functor with domain the category $\cob$.  To
compute $\TSkein(M)$, we remark that $M$ is given by gluing a unique
$1$-handle to the cylinder over $S^2\sqcup S^2$, that is, 
$
M=W(\Sp^0)=((S^2\sqcup S^2) \times [0,1]) \cup_{\Sp^0} (D^{2}\times D^{1}).
$
The the $\FK$-linear homomorphism $\TSkein(M) \co \TSkein(S^2\sqcup S^2) \to \TSkein(S^2)$ defines a map given by $\Gamma_1 \sqcup \Gamma_2 \mapsto  \Gamma$ where $\Gamma$ is the admissible graph in $S^2$ represented by
 a red curve at the
equator and the graphs $\Gamma_1$ and $\Gamma_2$ in the upper and lower hemispheres, respectively.
We now consider the two cases.  First, if $\cat$ is not semisimple, then after making the red circle of $\Gamma$ blue, we obtain the disjoint union of two admissible graphs in $S^2$ which is skein
equivalent to $0$. Indeed, any admissible closed graph is sent to $0$
by the functor $F$ (given in \eqref{eq-defF}) associated to a non-semisimple category. Second, if  $\cat$ is semisimple with dimension zero, then the unit object $\unit$ is projective and it can be used to make the red circle of $\Gamma$ blue.  In this case, $\Gamma$ becomes skein equivalent to $F(\Gamma_1)\tr(\chr_\unit)\Gamma_2=0$ because $\tr(\chr_\unit)=0$ (see Section~\ref{sect-chr-cat}).

\section{Existence of chromatic maps}\label{sect-chromatic}

Throughout this section, $\cc$ is a finite tensor category over an algebraically closed field~$\kk$. We introduce left and right chromatic maps for $\cc$ and prove that such maps always exist. As an example, the case of categories of representations of finite dimensional Hopf algebras is treated in detail.

\subsection{Left and right chromatic maps}\label{sect-def-lef-right-chromatic}
Pick a projective cover $\varepsilon\co P_0 \to \un$  of the unit object and a  monomorphism $\eta\co \alpha \to P_0$, where $\alpha$ is the distinguished invertible object of $\cc$.

\begin{lemma}\label{lem-Lambda}
There are unique natural transformations
$$\Lambda^r=\{\Lambda_X^r \co \alpha \otimes X \to X\}_{X \in \cc}
\quad \text{and} \quad
\Lambda^l=\{\Lambda_X^l \co X \otimes \alpha \to X\}_{X \in \cc}
$$
such that for any indecomposable projective object $P$ non isomorphic to $P_0$,
$$
\Lambda^r_P=0, \quad  \Lambda^l_P=0,  \quad \text{and} \quad \Lambda^r_{P_0}=\eta \otimes \varepsilon, \quad \Lambda^l_{P_0}=\varepsilon \otimes \eta.
$$
\end{lemma}
We prove Lemma~\ref{lem-Lambda} in Section~\ref{proof-lem-Lambda}.

Let $P$ be a projective object and $G$ be a projective generator of~$\cc$.  A \emph{right chromatic map} based at $P$ for~$G$ is a morphism
$$
\chr_P^r\in\Hom_\cc(P \otimes \rrdual{G}, P \otimes G \otimes \alpha )
$$
such that for all $X\in\cc$,
$$
(\Id_P \otimes \rev_G \otimes  \Id_X)(\Id_{P \otimes G} \otimes \Lambda^{r}_{\rdual{G} \otimes X})(\chr_P^r \otimes \Id_{\rdual{G} \otimes X}) (\Id_P \otimes \rev_{\rdual{G}} \otimes \Id_X)=\Id_{P \otimes X}.
$$
Using graphical calculus for monoidal categories (with the convention of diagrams to be read from bottom to top), the latter condition depicts as:
$$
\epsh{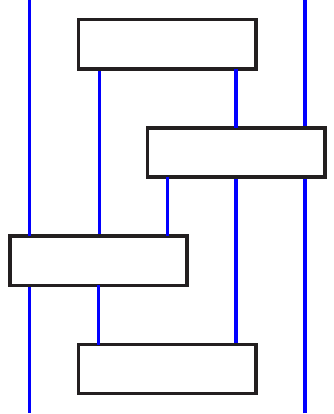}{24ex}
\putrc{6}{97}{$\ms{P}$}
\putc{50}{91}{$\ms{\rev_{G}}$}
\putlc{94}{96}{$\ms{X}$}
\putrc{68}{77}{$\ms{R}$}
\putrc{28}{65}{$\ms{G}$}
\putc{71}{65}{$\ms{\Lambda^r_{\rdual{G} \otimes X}}$}
\putrc{47}{50}{$\ms{\alpha}$}
\putc{29}{37}{$\ms{\chr_P^r}$}
\putlc{32}{23}{$\ms{\rrdual{G}}$}
\putlc{73}{23}{$\ms{\rdual{G}}$}
\putc{49}{12}{$\ms{\rcoev_{\rdual{G}}}$}
\putrc{5}{4}{$\ms{P}$}
\putlc{95}{5}{$\ms{X}$}
\; =\;\;\, 
\epsh{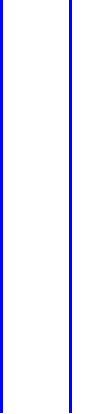}{24ex}
\putlc{12}{3}{$\ms{P}$}
\putlc{80}{4}{$\ms{X}$}\;.
$$
Similarly, a \emph{left chromatic map} based at $P$ for $G$ is a morphism
$$
\chr_P^l\in\Hom_\cc(\lldual{G} \otimes P, \alpha \otimes G \otimes P)
$$
such that for all $X\in\cc$,
$$
(\Id_X \otimes \lev_G \otimes \Id_P)(\Lambda^l_{X \otimes \ldual{G}} \otimes \Id_{G \otimes P})(\Id_{X \otimes \ldual{G}} \otimes \chr_P^l) (\Id_X \otimes \lcoev_{\ldual{G}} \otimes \Id_P)=\Id_{X \otimes P}.
$$
This condition depicts as:
$$
\epsh{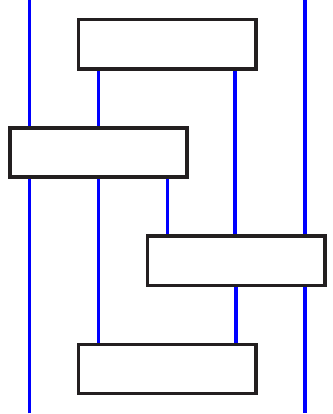}{24ex}
\putrc{6}{97}{$\ms{X}$}
\putlc{94}{97}{$\ms{P}$}
\putc{49}{91}{$\ms{\lev_{G}}$}
\putlc{32}{77}{$\ms{\ldual{G}}$}
\putc{29}{65}{$\ms{\Lambda^l_{X \otimes \ldual{G}}}$}
\putlc{73}{63}{$\ms{G}$}
\putlc{52}{50}{$\ms{\alpha}$}
\putc{71}{38}{$\ms{\chr_P^l}$}
\putlc{32}{23}{$\ms{\ldual{G}}$}
\putrc{69}{23}{$\ms{\lldual{G}}$}
\putc{49}{12}{$\ms{\lcoev_{\ldual{G}}}$}
\putrc{6}{4}{$\ms{X}$}
\putlc{95}{4}{$\ms{P}$}
\; =\;\;\, 
\epsh{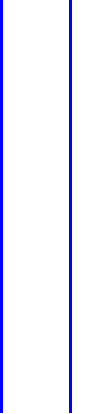}{24ex}
\putlc{12}{4}{$\ms{X}$}
\putlc{80}{4}{$\ms{P}$}\;.
$$

The main result of this section is the existence of right and left chromatic maps for any finite tensor category over an algebraically closed field~$\kk$:
\begin{theorem}\label{thm-chromatic}
For any projective object $P$ and any projective generator $G$ of~$\cc$, there are a right chromatic map and a left chromatic map based at $P$ for $G$.
\end{theorem}
We prove Theorem~\ref{thm-chromatic} in Section~\ref{sect-proof-chromatic} using the notions of central Hopf monad and (co)integrals based at~$\alpha$  reviewed Section~\ref{sect-Hopf-monads}. In Section~\ref{exa-H-mod} we explicitly compute right and left chromatic maps for the category of representations of a Hopf algebra.

\subsection{The case of spherical tensor categories}\label{sec-case-spherical}
Assume that $\cc$ is spherical, meaning that the unit object $\un$ is the distinguished invertible object of $\cc$ (see Section~\ref{sect-spherical-cat-def}). As in the previous section, pick a  projective cover $\varepsilon \co P_0 \to \un$  and a mono\-mor\-phism $\eta \co \un \to P_0$. By~\cite[Corollary 5.6]{GKP22}, there is a unique non-degenerate m-trace 
$$\mt=\{\mt_P \co \End_\cc(P) \to \kk\}_{P \in \Proj_\cc}$$ 
such that $\mt_{P_0}(\eta \varepsilon)=1_\kk$. 
Let $\Lambda^\mt_{P}\in\End_\cc(P)$ be the morphism \eqref{E:DefOmegaLambda} associated to a projective object $P$ and the m-trace $\mt$.
Consider the natural transformations $\Lambda^r$ and $\Lambda^l$ associated with $\varepsilon$ and $\eta$ as in~Lemma~\ref{lem-Lambda}.
\begin{lemma}\label{lem-Lambdas-comparison}
For any projective object $P$ of $\cc$, we have: $\Lambda^\mt_P=\Lambda^r_P=\Lambda^l_P$.
\end{lemma}
\begin{proof}
Since $\Hom_\cc(P_0,\un)=\kk \, \varepsilon$,  $\Hom_\cc(\un,P_0)=\kk\,\eta$,  and $\mt_{P_0}(\eta \varepsilon)=1_\kk$, the definition of $\Lambda^\mt$ gives that $\Lambda^\mt_{P_0}=\eta \varepsilon$. Using that $\eta \varepsilon=\eta \otimes \varepsilon= \varepsilon \otimes \eta$, we get that $\Lambda^\mt_{P_0}=\Lambda^r_{P_0}=\Lambda^l_{P_0}$.

Let $P$ be an indecomposable projective object non isomorphic to $P_0$. Since $\Hom_\cc(P,\un)=0=\Hom_\cc(\un,P)$, the definition of $\Lambda^\mt$ gives that $\Lambda^\mt_P=0$, and so we get that $\Lambda^\mt_P=\Lambda^r_P=\Lambda^l_P$.

Since any projective object is a (finite) direct sum of indecomposable projective objects, the above equalities together with the naturality $\Lambda^\mt$, $\Lambda^r$, $\Lambda^l$ implies that $\Lambda^\mt_P=\Lambda^r_P=\Lambda^l_P$ for all $P \in \Proj_\cc$.
\end{proof}

The next result is a direct consequence of Theorem~\ref{thm-chromatic} and Lemma~\ref{lem-Lambdas-comparison}. 
\begin{corollary}\label{cor-exists-chromatic}
Recall $\cc$ is a spherical tensor category.  For any projective object $P$ and any projective generator $G$ of~$\cc$, there is a chromatic map based at $P$ for $G$.
\end{corollary}
Since $\cc$ is a finite tensor category, any non zero morphism to $\unit$ is an epimorphism. Consequently, Theorem~\ref{T:SphericalChrom} is a direct consequence of Corollary~\ref{cor-exists-chromatic}.

\begin{proof}
Let $\phi$ be the pivotal structure of $\cc$ (see Section~\ref{sect-pivotal-cat}). The fact that $\rcoev_G=(\Id_{G^*} \otimes \phi^{-1}_{G})\lcoev_{G^*}$ and Lemma~\ref{lem-Lambdas-comparison}  imply that a morphism $\chr_P\co G \otimes P \to G \otimes P$ is a chromatic map  if and only if the morphism $\chr_P^l= \chr_P( \phi_G^{-1} \otimes \Id_P) \co \bdual{G} \otimes P \to G \otimes P$
is a left chromatic map. 
The existence of a chromatic map based on any projective object $P$ for any projective generator $G$ follows then from Theorem~\ref{thm-chromatic}.
\end{proof}

\subsection{The case of finite dimensional Hopf algebras}\label{exa-H-mod}
Let $H$ be a finite dimensional Hopf algebra over $\kk$.  The category $H$\trait$\mathrm{mod}$ of finite dimensional (left) $H$-modules and $H$-linear homomorphisms is a finite tensor category. 
Recall that the left dual of an object $M$ of $H$-mod  is the $H$-module $\ldual{M}=M^\ast=\Hom_\kk(M,\kk)$  where each $h \in H$ acts as the transpose of $m \in M \mapsto S(h)\cdot m \in M$, with~$S$  the antipode of $H$. The right dual of $M$ is $\rdual{M}=M^\ast$  where each $h \in H$ acts as the transpose of $m \in M \mapsto S^{-1}(h)\cdot m \in M$. The associated left and right evaluation morphisms are computed for any $m \in M$ and $\varphi \in M^\ast$ by
$$
\lev(\varphi \otimes m)=\varphi(m)=\rev(m \otimes \varphi).
$$
A projective generator of $H$-mod is $H$ equipped with its left
regular action. It follows from \cite[Proposition 6.5.5.]{EGNO} that
the distinguished object $\alpha$ of $H$-mod is~$\kk$ with action
$H \otimes \hspace{.1em} \kk \cong H \to \kk$ given by the inverse
$\alpha_{H}\in H^*$ of the distinguished grouplike element of $H^*$.
(The form $\alpha_H$ is characterized by
$\Lambda S(h)=\alpha_H(h) \Lambda$ for all $h \in H$ and all left
cointegral $\Lambda\in H$.)  Pick a projective cover $\varepsilon\co P_0 \to \kk$  of the unit object and a  monomorphism $\eta\co \alpha \to P_0$.
Since the counit
$\varepsilon_H\co H\to\kk$ of $H$ is an epimorphism, there exists an epimorphism
$p\co H\to P_0$ such that $\varepsilon_H=\varepsilon p$. Let $i\co P_0\to H$
be a section of $p$ in $H$-mod and set $\Lambda=S(i\eta(1_\kk))\in H$.

\begin{lemma}
Then $\Lambda$ is a nonzero left cointegral. 
\end{lemma}
\begin{proof}
It follows from  \cite[Proposition 10.6.2.]{R2012} that the set  $L_{\alpha_H}=\{a\in H \, | \,hx=\alpha_H(h)a \text{ for all } h \in H \}$ is a one dimensional left ideal of $H$ which is equal to the set of right cointegrals of~$H$. The element $a=i\eta(1_\kk) \in H$ is nonzero (because $p(a)=\eta(1_\kk)\neq0$ since $\eta$ is a monomorphism). Moreover, the $H$-linearity of $i\eta$ implies that $a \in L_{\alpha_H}$. Thus $x$ is a nonzero right cointegral. Consequently, $\Lambda=S(a)$ is a nonzero left cointegral.
\end{proof}

By \cite[Theorem 10.2.2]{R2012}, there is a unique right integral
$\lambda\in H^*$ such that $\lambda(\Lambda)=1$.  Consider the canonical $\kk$-linear isomorphism $x \in H \mapsto e_x \in H^{\ast\ast}$ (defined by $e_x(\varphi)=\varphi(x)$ for all $\varphi \in H^\ast$).
\begin{theorem}\label{them-left-right-chromatic-Hopf}
A left chromatic map based at $H$ for $H$ is
$$
 \chr_H^l \co \left \{\begin{array}{ccl} \lldual{H}\otimes H & \to & \alpha\otimes H\otimes H \\ e_x \otimes y & \mapsto &  \lambda\bigl(S(y_{(1)})x\bigr)\alpha_{H}(y_{(2)})\otimes y_{(3)}\otimes y_{(4)}
  \end{array}\right. 
$$
and a right chromatic map based at $H$ for $H$ is 
$$
 \chr_H^r \co \left \{\begin{array}{ccl} H\otimes \rrdual{H} & \to & H\otimes H\otimes \alpha \\ y \otimes e_x & \mapsto & 
  y_{(1)}\otimes y_{(2)} \otimes \alpha_{H}(y_{(2)})  \lambda\bigl(S(x)y_{(1)}\bigr).
  \end{array}\right. 
$$
More generally, for any finite dimensional projective $H$-module~$P$, 
$$
 \chr_P^l= \sum_i (\Id_{\alpha \otimes H} \otimes g_i) \chr_H^l (\Id_{\lldual{H}} \otimes f_i) \quad \text{and} \quad \chr_P^r=\sum_i (g_i \otimes \Id_{H \otimes \alpha}) \chr_H^l (\Id_{f_i \otimes \rrdual{H}}) 
$$
are a left chromatic map and right chromatic map based on $P$ for $H$, where $\{f_i\co P\to H,g_i \co H\to P\}_i$ is any finite family of $H$-linear homomorphisms such that $\Id_P=\sum_ig_if_i$.
\end{theorem}

We prove Theorem~\ref{them-left-right-chromatic-Hopf} in Section~\ref{proof-chromatic-H-mod}.

If $H$ is unimodular and unibalanced, then $\alpha_{H}\in H^*$ is the counit of $H$ and the pivotal structure evaluated at $H$ is computed by $x \in H \mapsto \phi_H(x)=e_{gx} \in H^{\ast\ast}$, where $g$ is the  pivot of $H$. Consequently, using the computation of the left chromatic map $\chr_H^l$ based at $H$ for $H$ given by Theorem~\ref{them-left-right-chromatic-Hopf} and the fact that $\chr_H^l( \phi_H \otimes \Id_H)$ is a chromatic map for $H$ (see the proof of Corollary~\ref{cor-exists-chromatic}), we obtain the expression of the chromatic map $\chr_H$ for $H$ given in Example~\ref{ex-Hopf-chromatic}.

\subsection{Proof of Lemma~\ref*{lem-Lambda}}\label{proof-lem-Lambda}
For any indecomposable projective object $P$ non isomorphic to $P_0$ and all morphisms $f \in \End_\cc(P_0)$, $g \in \Hom_\cc(P,P_0)$, $h \in \Hom_\cc(P_0,P)$, it follows from \cite[Lemma 4.3]{GKP22} that
$$
\eta \otimes \varepsilon f=f\eta \otimes \varepsilon, \quad \varepsilon g=0, \quad h\eta=0.
$$
This and the fact that any projective object is a (finite) direct sum of indecomposable projective objects imply that the prescriptions of Lemma~\ref{lem-Lambda} uniquely define natural transformations
$\{\Lambda_P^r \co \alpha \otimes P \to P\}_{P \in \Proj_\cc}$
and $\{\Lambda_P^r \co P \otimes \alpha \to P\}_{P \in \Proj_\cc}$, where $\Proj_\cc$ is the full subcategory of $\cc$ of projective objects. These natural transformations further uniquely  extend to $\cc$ by applying the next Lemma~\ref{lem-extension} with the functor $F=\alpha \otimes -$ (which is exact since it is an equivalence because $\alpha$ is invertible)
and the identity functor $G=1_\cc$.

\begin{lemma}\label{lem-extension}
Let $F,G \co \aaa \to \bb$ be additive functors between abelian categories. Assume that $\aaa$ has enough projectives and that $F$ is right exact. Denote by  $\Proj_\cc$ the full subcategory of $\aaa$ of projective objects.
Then any natural transformation
$\{\alpha_P \co F(P) \to G(P)\}_{P \in \Proj_\cc}$ uniquely extends to $\aaa$, that is, to a natural transformation $\{\alpha_X \co F(X) \to G(X)\}_{X \in \aaa}$.
\end{lemma}
\begin{proof}
Consider a natural transformation $\alpha=\{\alpha_P \co F(P) \to G(P)\}_{P \in \Proj_\cc}$. Assume first that $\bar{\alpha}$ and $\tilde{\alpha}$ are both extensions of $\alpha$ to $\aaa$. Let $X \in \aaa$. Pick an
epimorphism $p\co P \to X$ with $P$ projective.
Using the  naturality of $\bar{\alpha}$ and $\tilde{\alpha}$ together with the fact that both $\bar{\alpha}$ and $\tilde{\alpha}$ extend $\alpha$, we have:
$$
\bar{\alpha}_X F(p)=G(p) \bar{\alpha}_P=G(p) \alpha_P=G(p) \tilde{\alpha}_P=\tilde{\alpha}_X F(p).
$$
Thus $\bar{\alpha}_X=\tilde{\alpha}_X$ since $F(p)$ is an epimorphism (because $p$ is and $F$ is right exact). This proves the uniqueness of an extension of $\alpha$ to $\aaa$.

We now prove the existence of an extension of $\alpha$ to $\aaa$. Let $X \in \aaa$. Pick an epimorphism $p\co P \to X$ with~$P$ projective. Then there is a unique morphism $\bar{\alpha}_X \co F(X) \to G(X)$ in $\aaa$ such that
\begin{equation}\label{eq-p1}
\bar{\alpha}_X F(p)=G(p) \alpha_P.
\end{equation}
Indeed, since $\aaa$ is abelian, the epimorphism $p$ is the cokernel of its kernel $k \co K \to P$. Pick an epimorphism $r\co Q \to K$ with $Q$ projective. Then $p$ is the cokernel of $q=kr \co Q \to P$, and so $F(p)$ is the cokernel of $F(q)$ (because $F$ is right exact). Consequently, since
$$
G(p) \alpha_P F(q)=G(p)G(q) \alpha_Q= G(pq) \alpha_Q=G(0) \alpha_Q=0,
$$
there is a unique morphism $\bar{\alpha}_X \co F(X) \to G(X)$ in $\aaa$ satisfying \eqref{eq-p1}. Note that the morphism $\bar{\alpha}_X$ does not depend on the choice of $p$. Indeed, let $r \co R \to X$ be another epimorphism with $R$ projective and denote by $\tilde{\alpha}_X \co F(X) \to G(X)$ the unique morphism such that $G(r) \alpha_R=\tilde{\alpha}_X F(r)$. 
Since $P$ is projective and $r$ is an epimorphism, there is a morphism $s \co P \to R$ such that $p=rs$. Then
$$
\bar{\alpha}_X F(p)=G(p) \alpha_P=G(r)G(s) \alpha_P=G(r)\alpha_R F(s)=\tilde{\alpha}_X F(r)F(s)=\tilde{\alpha}_XF(p),
$$
and so $\tilde{\alpha}_X=\bar{\alpha}_X$  (since $F(p)$ is an epimorphism). Note also that $\bar{\alpha}_P=\alpha_P$ for all $P \in \Proj_\cc$.  Indeed, since $\Id_P \co P \to P$ is an epimorphism with $P$ projective and using the defining relation \eqref{eq-p1}, we have:
$$
\bar{\alpha}_P= \bar{\alpha}_P F(\Id_P)=G(\Id_P) \alpha_P=\alpha_P.
$$
It remains to prove that the family $\bar{\alpha}=\{\bar{\alpha}_X \co F(X) \to G(X)\}_{X \in \aaa}$ is natural  in~$X$.
Let $f \co X \to Y$ be a morphism in $\aaa$. Pick epimorphisms $p\co P \to X$ and $q\co Q \to Y$ with $P,Q$ projective. Since $P$ is projective and $q$ is an epimorphism, there is a morphism $g \co P \to Q$ such that $fp=qg$.  Consider the following diagram:
\[\begin{tikzcd}[column sep={1.3em,between origins}, row sep={1.2em,between origins}]
	&&&&& {F(X)} &&&&&&& {F(Y)} \\ \\ \\ &&&&& {\small\text{(i)}} \\
	&&&&&&&&&&&& {\small\text{(ii)}} \\
	&&&&&& {F(Q)} \\ \\ {F(P)} &&&&&& {\small\text{(v)}} &&&&&& {G(Q)} &&&&& {G(Y)} \\ 	\\
	&&&&&& {G(P)} \\ &&&&&&&&&&&& {\small\text{(iii)}} \\
	&&&&& {\small\text{(iv)}} \\ \\ \\ 	&&&&& {F(X)} &&&&&&& {G(X)}
	\arrow["{F(p)}", from=8-1, to=1-6]
	\arrow["{F(f)}", from=1-6, to=1-13]
	\arrow["{\bar{\alpha}_Y}", from=1-13, to=8-18]
	\arrow["{F(p)}"', from=8-1, to=15-6]
	\arrow["{\bar{\alpha}_X}"', from=15-6, to=15-13]
	\arrow["{G(f)}"', from=15-13, to=8-18]
	\arrow["{F(q)}"', from=6-7, to=1-13]
	\arrow["{\alpha_P}"'{pos=0.7}, from=8-1, to=10-7]
	\arrow["{G(p)}", from=10-7, to=15-13]
	\arrow["{F(g)}"{pos=0.7}, from=8-1, to=6-7]
	\arrow["{G(q)}", from=8-13, to=8-18]
	\arrow["{\alpha_Q}"{pos=0.4}, from=6-7, to=8-13]
	\arrow["{G(g)}"'{pos=0.4}, from=10-7, to=8-13]
\end{tikzcd}\]
The inner squares (i) and (iii) commute by the functoriality of $F$ and $G$ applied to the equality $fp=qg$. The inner squares (ii) and (iv) commute by the defining relation \eqref{eq-p1}. The inner square (v) commutes by the naturality of $\alpha$. Consequently, the outer diagram commutes: $\bar{\alpha}_YF(f)F(p)=G(f)\bar{\alpha}_XF(p)$. Since $F(p)$ is an epimorphism (because $p$ is and $F$ is right exact), we obtain $\bar{\alpha}_YF(f)=G(f)\bar{\alpha}_X$.
\end{proof}

\subsection{Hopf monads, based (co)integrals and central Hopf monad}\label{sect-Hopf-monads}
In this subsection we review the notions of a Hopf monad and their based (co)integrals and recall the construction of the central Hopf monad. These are instrumental in the proof of Theorem~\ref{thm-chromatic} in Section~\ref{sect-proof-chromatic}.

A \emph{monad} on a category $\cc$  is a monoid in the category of endofunctors of $\cc$, that
is, a triple $(T,m,u)$ consisting of a functor $T\co \cc \to \cc$ and two natural transformations
$$
m=\{m_X\co T^2(X) \to T(X)\}_{X \in \cc}\quad \text{and} \quad  u=\{u_X\co X \to T(X)\}_{X \in \cc}
$$
called the \emph{product} and the \emph{unit} of $T$, such that for any $X\in\cc$,
$$
m_XT(m_X)=m_Xm_{T(X)} \quad {\text {and}} \quad m_Xu_{T(X)}=\Id_{T(X)}=m_X T(u_X).
$$

A \emph{bimonad} on monoidal category   $\cc$ is a monoid in the
category of  comonoidal endofunctors of $\cc$. In other words, a bimonad on $\cc$ is a
monad $(T,m,u)$ on $\cc$ such that the functor $T$ and the natural transformations $m$ and $u$ are comonoidal. The comonoidality of $T$ means that $T$ comes equipped with a natural transformation $ T_2=\{T_2(X,Y) \co  T(X \otimes Y)\to T(X)
\otimes T(Y)\}_{X,Y \in \cc} $ and a morphism $T_0\co T(\un) \to \un$ such that for all $X,Y,Z \in \cc$,
\begin{gather*}
\bigl(\Id_{T(X)} \otimes T_2(Y,Z)\bigr) T_2(X,Y \otimes Z)= \bigl(T_2(X,Y) \otimes \Id_{T(Z)}\bigr) T_2(X \otimes Y, Z) ,\\
(\Id_{T(X)} \otimes T_0) T_2(X,\un)=\Id_{T(X)}=(T_0 \otimes \Id_{T(X)}) T_2(\un,X).
\end{gather*}
The comonoidality of $m$ and $u$ means that for all $X,Y \in~\cc$,
\begin{gather*}
T_2(X,Y)m_{X \otimes Y}=(m_X \otimes m_Y) T_2(T(X),T(Y))T(T_2(X,Y)),\\
 T_2(X,Y)u_{X \otimes Y}=u_X \otimes u_Y.
\end{gather*}

Let $T=(T,m,u)$ be a bimonad on a monoidal category $\cc$ and $A$ be an object of~$\cc$. A \emph{left $A$-integral} for~$T$ is a morphism $\Lambda_l \co T(A) \to \un$ in $\cc$ such that
$$
(\Id_{T(\un)} \otimes \Lambda_l)T_2(\un,A)= u_\un  \Lambda_l.
$$
Similarly, a \emph{right $A$-integral} for $T$ is a morphism $\Lambda_r \co T(A) \to \un$ in $\cc$ such that
$$
(\Lambda_r \otimes \Id_{T(\un)})T_2(A,\un)= u_\un  \Lambda_r.
$$
An \emph{$A$-cointegral} for $T$ is a morphism $\lambda \co \un \to T(A)$ in $\cc$ which is $T$-linear:
$$
m_A T(\lambda)=\lambda T_0.
$$

A \emph{Hopf monad} on  monoidal category $\cc$ is a bimonad on $\cc$
whose left and right fusion operators  are isomorphisms (see~\cite{BLV}).
When $\cc$ is a rigid category,    a bimonad $T$ on
$\cc$ is a Hopf monad if and only if it has a left antipode and a right antipode (see \cite{BV2}).  (Here, we will not need the actual definition of a Hopf monad and so just refer to~\cite{BLV,BV2}.)

Let $\cc$ be a rigid monoidal category. Assume that for any $X \in \cc$, the coend
\begin{equation}\label{eq-coend-Z}
Z(X)=\int^{Y \in \cc} \ldual{Y} \otimes X \otimes Y
\end{equation}
exists. Denote by $i_{X,Y}\co \ldual{Y} \otimes X \otimes Y \to  Z(X)$ the associated universal dinatural transformation and set
\begin{equation*}
\partial_{X,Y}=(\Id_{Y} \otimes i_{X,Y})(\lcoev_{Y} \otimes \Id_{X \otimes Y}) \co X \otimes Y \to Y \otimes Z(X).
\end{equation*}
We will depict the morphism $\partial_{X,Y}$ as
$$
\partial_{X,Y}=\;
\epsh{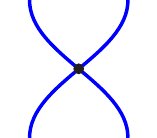}{9ex}
\putlc{87}{94}{$\ms{Z(X)}$}
\putrc{14}{94}{$\ms{Y}$}
\putlc{87}{6}{$\ms{Y}$}
\putrc{14}{6}{$\ms{X}$}
$$
and call $\partial=\{\partial_{X,Y}\}_{X,Y \in \cc}$ the \emph{centralizer} of $\cc$.
The universality of $\{i_{X,Y}\}_{Y \in \cc}$ translates to a universal factorization property for  $\partial$ as follows: for  any natural transformation $\{\xi_Y\co X \otimes Y \to Y \otimes M\}_{Y \in \cc}$ with $X,M\in\cc$, there exists a unique morphism $r\co Z(X) \to M$ in $\cc$ such that $\xi_Y=(\Id_Y \otimes r)\partial_{X,Y}$ for all $Y \in \cc$:
$$
\epsh{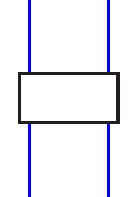}{10ex}
\putrc{16}{95}{$\ms{Y}$}
\putlc{85}{95}{$\ms{M}$}
\putc{50}{49}{$\ms{\xi_Y}$}
\putrc{16}{8}{$\ms{X}$}
\putlc{85}{9}{$\ms{Y}$}
\;\,=\;\,
\epsh{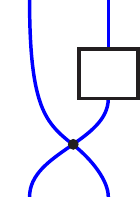}{10ex}
\putrc{15}{95}{$\ms{Y}$}
\putlc{84}{95}{$\ms{M}$}
\putc{77}{60}{$\ms{r}$}
\putrc{16}{6}{$\ms{X}$}
\putlc{82}{6}{$\ms{Y}$}
$$
Also, the parameter theorem for coends (see \cite{ML1}) implies that the family of coends $\{Z(X)\}_{X \in \cc}$ uniquely extend to a functor $Z \co \cc \to \cc$ so that~$\partial=\{\partial_{X,Y}\}_{X,Y \in \cc}$ is natural in $X$ and~$Y$.

By \cite[Corollary 5.14 and Theorem 6.5]{BV3}, the functor $Z$ has the structure of a quasitriangular Hopf monad, called the \emph{central Hopf monad} of $\cc$, which describes the center $\zz(\cc)$ of $\cc$ (meaning that the Eilenberg-Moore category of $Z$ is isomorphic to $\zz(\cc)$ as braided monoidal categories). The product $m$, unit $u$, and comonoidal structure $(Z_2,Z_0)$  are characterized (using the universal factorization property for~$\partial$) by the following equalities
with $X,X_1,X_2,Y,Y_1,Y_2\in\cc$:
\begin{center}
$\epsh{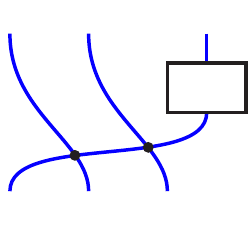}{14ex}
\putct{4}{88}{$\ms{Y_1}$}
\putct{36}{88}{$\ms{Y_2}$}
\putct{83}{88}{$\ms{Z(X)}$}
\putc{84}{60}{$\ms{m_X}$}
\putcb{4}{14}{$\ms{X}$}
\putcb{36}{14}{$\ms{Y_1}$}
\putcb{68}{14}{$\ms{Y_2}$}
\, = \, 
\epsh{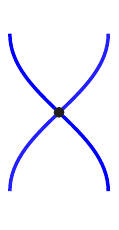}{14ex}
\putct{8}{88}{$\ms{Y_1\!\otimes\! Y_2}$}
\putct{95}{88}{$\ms{Z(X)}$}
\putcb{8}{14}{$\ms{X}$}
\putcb{91}{14}{$\ms{Y_1\! \otimes\! Y_2}$},
\qquad \qquad
u_X= 
\epsh{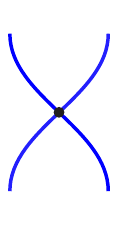}{14ex}
\putct{8}{88}{$\ms{\un}$}
\putct{91}{88}{$\ms{Z(X)}$}
\putcb{8}{14}{$\ms{X}$}
\putcb{91}{14}{$\ms{\un}$},
$
\\[1em]
$\epsh{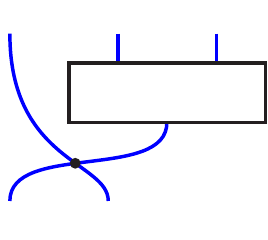}{14ex}
\putct{44}{90}{$\ms{Z(X_1)}$}
\putct{81}{90}{$\ms{Z(X_2)}$}
\putct{4}{90}{$\ms{Y}$}
\putc{63}{60}{$\ms{Z_2(X_1,X_2)}$}
\putcb{4}{14}{$\ms{X_1 \otimes X_2}$}
\putcb{41}{14}{$\ms{Y}$}
\,\,\, $=$ \,
\epsh{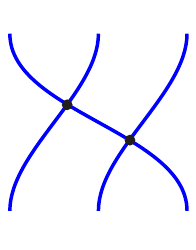}{14ex}
\putct{46}{90}{$\ms{Z(X_1)}$}
\putct{100}{90}{$\ms{Z(X_2)}$}
\putct{5}{90}{$\ms{Y}$}
\putcb{5}{14}{$\ms{X_1}$}
\putcb{50}{14}{$\ms{X_2}$}
\putcb{96}{14}{$\ms{Y}$},
\qquad \qquad 
\epsh{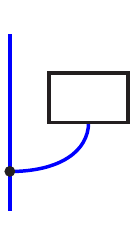}{12ex}
\putct{7}{90}{$\ms{Y}$}
\putcb{7}{13}{$\ms{Y}$}
\putc{68}{61}{$\ms{Z_0}$}
=
\epsh{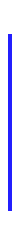}{14ex}
\putct{50}{90}{$\ms{Y}$}
\putcb{50}{13}{$\ms{Y}$}.
$
\end{center}
Note that the left and right antipodes and $R$-matrix of $Z$ can similarly be described (see~\cite{BV3}), but we do not recall these descriptions since we do not use them in the sequel.

\subsection{Proof of Theorem~\ref*{thm-chromatic}}\label{sect-proof-chromatic}

Note that a left chromatic map  based at a projective object $P$ for a projective generator $G$ is nothing but a right chromatic map based at $P$ for $G$ in the monoidal-opposite finite tensor category $\cc^{\otimes \opp}=(\cc,\otimes^\opp,\un)$. Thus we only need to prove the existence of right chromatic maps.

Since $\cc$ has a projective generator, the coend \eqref{eq-coend-Z} exists for all $X \in \cc$ (by \cite[Lemma 5.1.8]{KerLy2001}). By Section~\ref{sect-Hopf-monads}, we can then consider the central Hopf monad $Z=(Z,m,u,Z_2,Z_0)$ of~$\cc$ and its associated centralizer $\partial=\{\partial_{X,Y}\co X \otimes Y \to Y \otimes Z(X)\}_{X,Y \in \cc}$.

Recall the natural transformation $\Lambda^r=\{\Lambda_Y^r \co \alpha \otimes Y \to Y\}_{Y \in \cc}$ from Lemma~\ref{lem-Lambda}. The universal factorization property for  $\partial$ gives that there is a unique morphism $\Lambda_r\co Z(\alpha) \to \un$ in $\cc$ such that $\Lambda_Y^r=(\Id_Y \otimes \Lambda_r)\partial_{\alpha,Y}$ for all $Y\in \cc$:
$$
\epsh{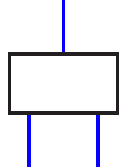}{10ex}
\putrc{45}{92}{$\ms{Y}$}
\putc{50}{49}{$\ms{\Lambda^r_Y}$}
\putrc{18}{6}{$\ms{\alpha}$}
\putlc{82}{6}{$\ms{Y}$}
\;\,=\;\,
\epsh{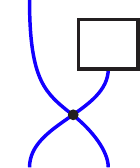}{10ex}
\putrc{16}{92}{$\ms{Y}$}
\putc{79}{71}{$\ms{\Lambda_r}$}
\putrc{15}{6}{$\ms{\alpha}$}
\putlc{82}{6}{$\ms{Y}$}\,\; .
$$

\begin{lemma}\label{lem-Ar-integ}
The morphism $\Lambda_r\co Z(\alpha) \to \un$ is a nonzero right $\alpha$-integral for $Z$.
\end{lemma}

\begin{proof}
Clearly $\Lambda_r\neq 0$ since $\Lambda^r$ is nonzero (because $\Lambda^r_{P_0}=\eta \otimes \varepsilon\neq 0$).
We need to prove that $(\Lambda_r \otimes \Id_{Z(\un)})Z_2(\alpha,\un)= u_\un  \Lambda_r$. It follows from the universal factorization property for~$\partial$ that this amounts to showing the equality of the natural transformations $l=\{l_Y\}_{Y \in \cc}$ and $r=\{r_Y\}_{Y \in \cc}$ defined by
$$
l_Y=\bigl(\Id_Y \otimes( \Lambda_r \otimes \Id_{Z(\un)})Z_2(\alpha,\un)\bigr)\partial_{\alpha,Y} \quad \text{and} \quad r_Y=(\Id_Y \otimes u_\un  \Lambda_r)\partial_{\alpha,Y}.
$$
Note that the definitions of $\Lambda_r$ and $Z_2(\alpha,\un)$ imply that $ r_Y=(\Id_Y \otimes u_\un)\Lambda^r_Y$ and
$$
l_Y=\bigl((\Id_Y \otimes \Lambda_r)\partial_{\alpha,Y} \otimes \Id_{Z(\un)}\bigr)(\Id_\alpha \otimes \partial_{\un,Y})
= (\Lambda^r_Y \otimes \Id_{Z(\un)})(\Id_\alpha \otimes \partial_{\un,Y}).
$$
Then $l_P=0=r_P$ for any indecomposable projective object $P$ non isomorphic to~$P_0$ (since
$\Lambda^r_P=0$). Also
$$
l_{P_0}\overset{(i)}{=} \eta \otimes \bigl((\varepsilon \otimes \Id_{Z(\un)})\partial_{\un,P_0}\bigr) \overset{(ii)}{=} \eta \otimes \partial_{\un,\un}\varepsilon\overset{(iii)}{=} (\Id_{P_0} \otimes u_\un)(\eta \otimes \varepsilon) \overset{(iv)}{=}r_{P_0}.
$$
Here  $(i)$ and $(iv)$ follow from the equality $\Lambda^r_{P_0}=\eta \otimes \varepsilon$, $(ii)$ from the naturality of~$\partial$, and $(iii)$ from the definition of $u_\un$. Consequently, using that any projective object is a (finite) direct sum of indecomposable projective objects, we obtain that $l_P=r_P$ for all $P \in \Proj_\cc$. Finally we conclude that $l=r$ by applying Lemma~\ref{lem-extension} with the functors $F=\alpha \otimes -$  and $G=-\otimes Z(\un)$.
\end{proof}

Since the central Hopf monad $Z$ is the central Hopf comonad for the finite tensor category $\cc^\opp$ opposite to $\cc$, it follows from Lemma~\ref{lem-Ar-integ} and \cite[Theorem 4.8]{Sh2} that  there is a unique $\alpha$-cointegral $\lambda \co \un \to Z(\alpha)$ such that $\Lambda_r\lambda=\Id_\un$. 

\begin{lemma}\label{lem-lambda-Lambda-id}
For any $X\in\cc$, $(\Id_X \otimes \Lambda_rm_\alpha)\partial_{Z(\alpha),X}(\lambda \otimes \Id_X)=\Id_X$.
\end{lemma}
\begin{proof}
We have:
\begin{gather*}
(\Id_X \otimes \Lambda_rm_\alpha)\partial_{Z(\alpha),X}(\lambda \otimes \Id_X) \overset{(i)}{=}
(\Id_X \otimes \Lambda_rm_\alpha Z(\lambda))\partial_{\un,X} \\ \overset{(ii)}{=} (\Id_X \otimes \Lambda_r\lambda Z_0)\partial_{\un,X} \overset{(iii)}{=} (\Id_X \otimes Z_0)\partial_{\un,X}  \overset{(iv)}{=}\Id_X.
\end{gather*}
Here $(i)$ follows from the naturality of~$\partial$, $(ii)$ from the fact that $m_\alpha Z(\lambda)= \lambda Z_0$ (because $\lambda$ is an $\alpha$-cointegral),  $(iii)$ from the equality $\lambda\Lambda_r=\Id_\un$, and $(iv)$ from the definition of $Z_0$.
\end{proof}

Let $P$ be a projective object and $G$ be a  projective generator of~$\cc$. Set
$$
a_P=(\Id_P \otimes \rev_{G} \otimes \Id_{Z(\alpha)})(\Id_{P\otimes G} \otimes \partial_{\alpha,\rdual{G}}) \co P \otimes G \otimes \alpha \otimes \rdual{G} \to P \otimes Z(\alpha).
$$
Graphically,
$$
a_P=\epsh{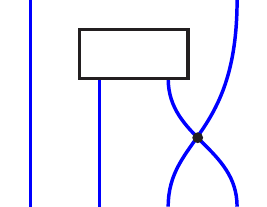}{12ex}
\putc{50}{75}{$\ms{\rev_{G}}$}
\putlc{92}{93}{$\ms{Z(\alpha)}$}
\putrc{8}{8}{$\ms{P}$}
\putrc{34}{8}{$\ms{G}$}
\putrc{60}{8}{$\ms{\alpha}$}
\putlc{92}{8}{$\ms{\rdual{G}}$}\;.
$$
\begin{lemma}\label{lem-ap-epi}
$a_P$ is an epimorphism.
\end{lemma}
\begin{proof}
Since $\rdual{G}$ is a projective generator of $\cc$, the universal dinatural transformation $i_{\alpha,\rdual{G}}\co \ldual{}(\rdual{G}) \otimes \alpha \otimes \rdual{G} \to Z(\alpha)$ is an epimorphism (by \cite[Corollary 5.1.8]{KerLy2001}). Then $b_P=\Id_P \otimes i_{\alpha,\rdual{G}}$ is an epimorphism (since $\otimes$ is exact because $\cc$ is rigid).
Considering the isomorphism $\varphi_G=(\rev_G \otimes \Id_{\ldual{}(\rdual{G})})(\Id_G \otimes \lcoev_{\rdual{G}})\co G \to \ldual{}(\rdual{G})$, we conclude that $a_P=b_P(\Id_P \otimes \varphi_G \otimes \Id_{\alpha \otimes \rdual{G}})$ is an epimorphism.
\end{proof}

Since $a_P$ is an epimorphism (by Lemma~\ref{lem-ap-epi}) and $P$ is a projective object, the morphism $\Id_P \otimes \lambda \co P \to P \otimes Z(\alpha)$ factors through $a_P$, that is, $\Id_P \otimes \lambda=a_Pd_P$ for some morphism $d_P \co P \to P \otimes G \otimes \alpha \otimes \rdual{G}$. Set
$$
\chr_P^r=(\Id_{} \otimes \rev_{\rdual{G}})(d_P \otimes \Id_{\rrdual{G}}) \co P \otimes \rrdual{G} \to P \otimes G \otimes \alpha.
$$
Graphically,
$$
\chr_P^r=\epsh{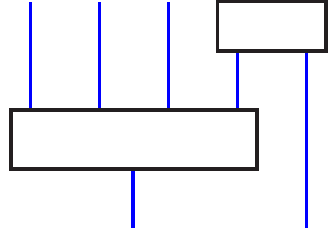}{14ex}
\putrc{6}{91}{$\ms{P}$}
\putrc{26}{92}{$\ms{G}$}
\putrc{47}{92}{$\ms{\alpha}$}
\putrc{71}{66}{$\ms{\rdual{G}}$}
\putc{81}{91}{$\ms{\rev_{\rdual{G}}}$}
\putc{40}{39}{$\ms{d_P}$}
\putrc{36}{11}{$\ms{P}$}
\putlc{95}{13}{$\ms{\rrdual{G}}$}\;.
$$
Then $\chr_P^r$ is a right chromatic map based at $P$ for $G$. Indeed, for any $X \in \cc$,
\begin{gather*}
\epsh{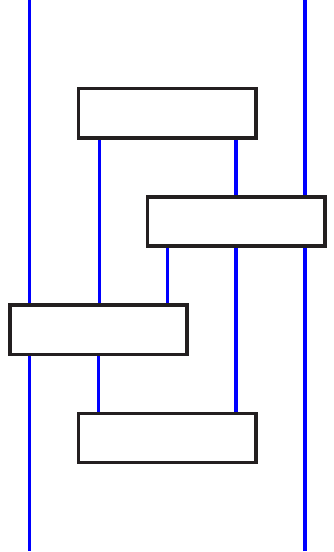}{30ex}
\putrc{5}{96}{$\ms{P}$}
\putlc{96}{95}{$\ms{X}$}
\putc{50}{81}{$\ms{\rev_{G}}$}
\putrc{70}{70}{$\ms{\rdual{G}}$}
\putrc{25}{61}{$\ms{G}$}
\putc{71}{61}{$\ms{\Lambda^r_{\rdual{G} \otimes X}}$}
\putrc{46}{51}{$\ms{\alpha}$}
\putc{29}{41}{$\ms{\chr_P^r}$}
\putlc{32}{31}{$\ms{\rrdual{G}}$}
\putrc{70}{32}{$\ms{\rdual{G}}$}
\putc{50}{22}{$\ms{\rcoev_{\rdual{G}}}$}
\putrc{4}{6}{$\ms{P}$}
\putlc{96}{6}{$\ms{X}$}
  \overset{(i)}{=}\;
\epsh{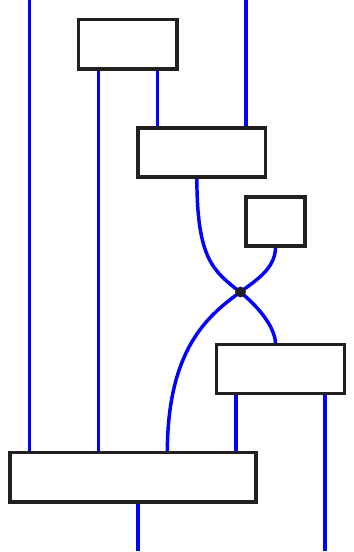}{30ex}
\putrc{4}{96}{$\ms{P}$}
\putc{35}{93}{$\ms{\rev_{G}}$}
\putlc{74}{97}{$\ms{X}$}
\putlc{49}{82}{$\ms{\rdual{G}}$}
\putc{57}{73}{$\ms{\Id_{\rdual{G}\! \otimes X}}$}
\putrc{24}{57}{$\ms{G}$}
\putc{78}{60}{$\ms{\Lambda_r}$}
\putrc{49}{38}{$\ms{\alpha}$}
\putc{80}{34}{$\ms{\Id_{\rdual{G}\! \otimes X}}$}
\putlc{70}{23}{$\ms{\rdual{G}}$}
\putc{38}{14}{$\ms{d_P}$}
\putrc{35}{3}{$\ms{P}$}
\putlc{96}{3}{$\ms{X}$}
  \overset{(ii)}{=}\;
\epsh{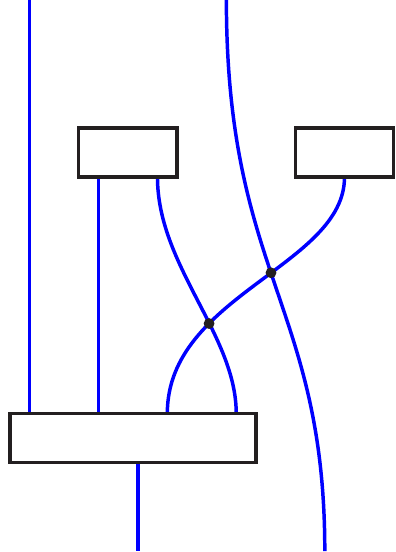}{30ex}
\putrc{4}{96}{$\ms{P}$}
\putlc{61}{96}{$\ms{X}$}
\putc{32}{73}{$\ms{\rev_{G}}$}
\putc{87}{73}{$\ms{\Lambda_r m_\alpha}$}
\putrc{22}{56}{$\ms{G}$}
\putlc{46}{57}{$\ms{\rdual{G}}$}
\putrc{43}{34}{$\ms{\alpha}$}
\putc{35}{21}{$\ms{d_P}$}
\putrc{31}{4}{$\ms{P}$}
\putlc{86}{4}{$\ms{X}$}
  \\[1em]\overset{(iii)}{=}\;
\epsh{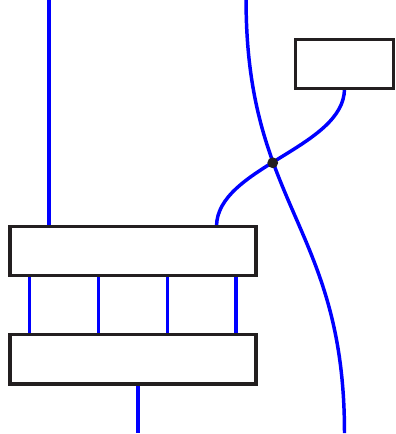}{24ex}
\putrc{10}{97}{$\ms{P}$}
\putrc{59}{97}{$\ms{Y}$}
\putc{87}{85}{$\ms{\Lambda_r m_\alpha}$}
\putc{33}{42}{$\ms{a_P}$}
\putrc{5}{30}{$\ms{P}$}
\putrc{22}{31}{$\ms{G}$}
\putrc{40}{30}{$\ms{\alpha}$}
\putrc{59}{31}{$\ms{\rdual{G}}$}
\putc{35}{17}{$\ms{d_P}$}
\putrc{32}{4}{$\ms{P}$}
\putlc{90}{4}{$\ms{X}$}
  \overset{(iv)}{=}\;
\epsh{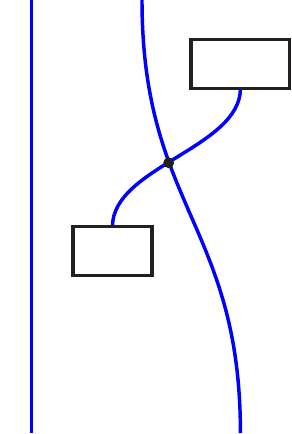}{24ex}
\putrc{44}{97}{$\ms{Y}$}
\putc{82}{85}{$\ms{\Lambda_r m_\alpha}$}
\putc{37}{42}{$\ms{\lambda}$}
\putrc{6}{4}{$\ms{P}$}
\putlc{87}{4}{$\ms{X}$}
  \overset{(v)}{=}\;
\epsh{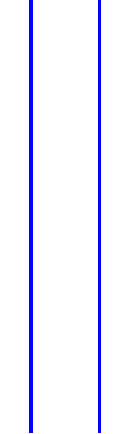}{24ex}
\putrc{14}{4}{$\ms{P}$}
\putlc{86}{4}{$\ms{X}$}
\end{gather*}
Here  $(i)$ follows from the definitions of $\chr_P^r$ and $\Lambda_r$, $(ii)$ from the definition of the product $m$ of $Z$, $(iii)$ from the definition of $a_P$, $(iv)$ from the fact that $a_Pd_P=\Id_P \otimes \lambda$, and $(v)$ from Lemma~\ref{lem-lambda-Lambda-id}.

\subsection{Proof of Theorem~\ref*{them-left-right-chromatic-Hopf}}\label{proof-chromatic-H-mod}
We first prove that $\chr_H^l$ is $H$-linear. For any $x,y,h \in H$,
\begin{align*}
\chr_H^l \bigl(h \cdot(e_x \otimes y)\bigr) 
&\overset{(i)}{=} \chr_H^l (e_{S^2(h_{(1)})x} \otimes h_{(2)}y) \\
&\overset{(ii)}{=}  \lambda\bigl(S(h_{(2)}y_{(1)}) S^2(h_{(1)})x\bigr)\, \alpha_{H}( h_{(3)}y_{(2)})\otimes  h_{(4)}y_{(3)}\otimes  h_{(5)}y_{(4)}   \\
&\overset{(iii)}{=} \lambda\bigl(S\bigl(S(h_{(1)})h_{(2)}y_{(1)}\bigr) x\bigr)\, \alpha_{H}( h_{(3)}y_{(2)})\otimes  h_{(4)}y_{(3)}\otimes  h_{(5)}y_{(4)}   \\
&\overset{(iv)}{=} \lambda\bigl(S(y_{(1)}) x\bigr)\, \alpha_{H}( h_{(1)}y_{(2)})\otimes  h_{ (2)}y_{(3)}\otimes  h_{(3)}y_{(4)}   \\
&\overset{(v)}{=} \alpha_{H}( h_{(1)})\, \lambda\bigl(S(y_{(1)}) x\bigr) \, \alpha_{H}( y_{(2)})\otimes  h_{(2)}y_{(3)}\otimes  h_{(3)}y_{(4)} \\   
& \overset{(vi)}{=}  h \cdot \chr_H^l (e_x \otimes y).
\end{align*}
Here $(i)$ follows from the definition of the monoidal product in $H$-mod, $(ii)$ from the definition of $\chr_H^l$ and the multiplicativity of the coproduct of $H$,  $(iii)$ from the anti-multiplicativity of $S$, $(iv)$ from the axiom of the antipode, $(v)$ from
the multiplicativity of $\alpha_H$, and $(vi)$ from the definitions of $\chr_H^l$ and of the monoidal product in $H-\mathrm{mod}$.

We next compute the natural transformation $\Lambda^l$. 
For any finite dimensional $H$-module~$M$, consider the $\kk$-linear  homomorphism
$$
 \tilde{\Lambda}_M^l \co \left \{\begin{array}{ccl} M\otimes\alpha & \to & M \\ m \otimes 1_\kk & \mapsto &  S^{-1}(\Lambda) \cdot m.
  \end{array}\right. 
$$
Then $\tilde{\Lambda}_M^l$ is $H$-linear. Indeed, for any $h \in H$ and $m \in M$,
\begin{gather*}
\tilde{\Lambda}^l_M(h \cdot(m\otimes 1_\kk)) \overset{(i)}{=}
\alpha_H(h_{(2)}) \, S^{-1}(\Lambda)h_{(1)} \cdot m \overset{(ii)}{=} \varepsilon_H(h_{(1)})\, \alpha_H(h_{(2)}) \, S^{-1}(\Lambda) \cdot m \\
\overset{(iii)}{=} \alpha_H(h)\, S^{-1}(\Lambda) \cdot m  
\overset{(iv)}{=}S^{-1}(\Lambda S(h)) \cdot m
\overset{(v)}{=} hS^{-1}(\Lambda) \cdot m 
 \overset{(vi)}{=}  h \cdot \tilde{\Lambda}^l_M(m\otimes 1_\kk),
\end{gather*}
where $\varepsilon_H$ is the counit of $H$.  Here $(i)$ follows from
the definitions of $ \tilde{\Lambda}_M^l$ and of the action of
$M \otimes \alpha$, $(ii)$ from the fact that $S^{-1}(\Lambda)$
is a right cointegral of $H$, $(iii)$ from the counitality of the
coproduct, $(iv)$ from the property characterizing $\alpha_H$ (see
Example~\ref{exa-H-mod}), $(v)$ from the anti-multiplicativity of $S$,
and $(vi)$ from the definition of~$\tilde{\Lambda}_M^l$. Clearly, the
family $\{\tilde{\Lambda}_M^l\}_M$ is natural in $M$. Now for
any $h\in H$,
$$
\tilde{\Lambda}_H^l(h\otimes 1_\kk)\overset{(i)}{=}S^{-1}(\Lambda)h\overset{(ii)}{=}\varepsilon_H(h)i\eta(1_\kk)\overset{(iii)}{=}(\varepsilon p\otimes
i\eta)(h\otimes 1_\kk).
$$
Here $(i)$ follows from the definition of $\tilde{\Lambda}_H^l$, $(ii)$ from the fact that $S^{-1}(\Lambda)$ is a right cointegral and from the definition of $\Lambda$, and $(iii)$ from the definition of $p$.
Thus, using that $pi=\Id_{P_0}$ and the naturality of $\tilde{\Lambda}^l$, we obtain:
$$
\tilde{\Lambda}_{P_0}^l=\tilde{\Lambda}_{P_0}^l(pi \otimes \Id_\alpha)=p\tilde{\Lambda}_{H}^l(i\otimes\Id_\alpha)=\varepsilon
pi\otimes pi\eta=\varepsilon\otimes \eta.
$$
Also, for any indecomposable projective object $P$ non isomorphic
to $P_0$, we have $\tilde{\Lambda}_{P}^l=0$. Indeed, in the projective generator $H$,
the image $\kk S^{-1}(\Lambda)=i\eta(\kk)\subset i(P_0)$ of $\tilde{\Lambda}_{H}^l$ is isomorphic to the simple $H$-module $\alpha$, and $i(P_0) \cong P_0$ is the only (up to isomorphism) indecomposable projective $H$-module which has a submodule isomorphic to $\alpha$ (by uniqueness of the socle, see Section~\ref{sect-finite-tensor}).
Consequently, the uniqueness in Lemma~\ref{lem-Lambda} implies that $\Lambda^l=\tilde{\Lambda}^l$.

We now prove that $\chr_H^l$ is a left chromatic map. Let $M$ be a finite dimensional $H$-module. Pick any $m \in M$ and $x \in H$. In $M\otimes\ldual H\otimes\alpha\otimes H\otimes H$, we have:
$$
(\Id_{M\otimes\ldual{H}} \otimes \chr_H^l) (\Id_{M}\otimes \lcoev_{\ldual{H}} \otimes
\Id_H)(m\otimes x)=  m\otimes\lambda(S(x_{(1)})\_)\otimes\alpha_{H}(x_{(2)})\otimes
x_{(3)}\otimes x_{(4)}.
$$
Evaluating this vector under $(\Id_M \otimes \lev_H \otimes \Id_H)(\Lambda^l_{M \otimes \ldual{H}} \otimes \Id_{H \otimes H})$ gives
\begin{align*}
\alpha_{H}&(x_{(2)})\, \lambda\bigl(S(x_{(1)})\, \Lambda_{(1)}x_{(3)}\bigr)\, \bigl( S^{-1}(\Lambda_{(2)})\cdot m \bigr) \otimes x_{(4)} \\
& \overset{(i)}{=} \alpha_{H}(x_{(2)})\, \alpha_{H}\bigl(S(x_{(3)})\bigr)\,\lambda\bigl(S^2(x_{(4)})S(x_{(1)})\, \Lambda_{(1)}\bigr)\, \bigl( S^{-1}(\Lambda_{(2)})\cdot m \bigr) \otimes x_{(5)} \\
& \overset{(ii)}{=} \lambda\bigl(S^2(x_{(2)})S(x_{(1)})\, \Lambda_{(1)}\bigr)\, \bigl( S^{-1}(\Lambda_{(2)})\cdot m \bigr) \otimes x_{(3)}\\
&  \overset{(iii)}{=} \varepsilon_H(x_{(1)}) \, \lambda(\Lambda_{(1)})\, \bigl( S^{-1}(\Lambda_{(2)})\cdot m \bigr) \otimes x_{(2)} \\
& \overset{(iv)}{=}\, \bigl( S^{-1}\bigl(\lambda(\Lambda_{(1)})\Lambda_{(2)}\bigr)\cdot m \bigr) \otimes x 
  \overset{(v)}{=}\, m \otimes x.
\end{align*}
Here $(i)$ follows from the fact that $\lambda(ab)=\alpha_{H}(S(b_{(1)}))\lambda(S^{2}(b_{(2)})a)$ for all $a,b \in H$ (see~\cite[Theorem 10.5.4]{R2012}), $(ii)$ from multiplicativity of $\alpha_H$ and the axiom of the antipode, $(iii)$ from the axiom of the antipode, $(iv)$ from the counitailty of the coproduct, and $(v)$ from the fact that $\lambda(\Lambda_{(1)})\Lambda_{(2)}=\lambda(\Lambda) 1_H=1_H$.
Consequently, 
$$
(\Id_M \otimes \lev_H \otimes \Id_H)(\Lambda^l_{M \otimes \ldual{H}} \otimes \Id_{H \otimes H})(\Id_{M \otimes \ldual{H}} \otimes \chr_H^l) (\Id_M \otimes \lcoev_{\ldual{H}} \otimes \Id_H)=\Id_{M \otimes H},
$$
that is, $\chr_H^l$ is a left chromatic map based at $H$ for $H$.

Finally, the expression for $\chr_H^r$ is derived from that of $\chr_H^l$ by noticing that for any projective generator $G$ and projective object $P$ in $H$-mod, a right chromatic map based at $P$ for $G$ in $H$-mod is a left chromatic map at $P$ for $G$ in $(H$-mod$)^{\otimes \opp}$, that is, in $(H^{\mathrm{cop}})$-mod, where $H^{\mathrm{cop}}$ is $H$ with opposite coproduct (for which $S^{\mathrm{cop}}=S^{-1}$, $\Lambda^{\mathrm{cop}}=\Lambda$, $\lambda^{\mathrm{cop}}=\lambda S$, and $\alpha_{H^{\mathrm{cop}}}=\alpha_H$).

\linespread{1}


\begin{thebibliography}{BHMV}
\bibitem[At]{At} Atiyah, M.,  {\em Topological quantum field theories},  Publications Mathématiques de l'IHÉS. 68 (1988), 175--186.

\bibitem[AGPS]{aghaei2018} Aghaei, N., Gainutdinov, A., Pawelkiewicz, M., Schomerus, V.,  {\em Combinatorial quantization of {C}hern-{S}imons
    theory {I}: {T}he torus}, ar{X}iv:1811:09123.
    
\bibitem[BW]{BW99} Barrett, J.W., Westbury, B.W., {\em Spherical
    categories}, Adv. Math., 143 (1999):357--375.

\bibitem[Ba]{Bart2022} Bartlett, B., {\em Three-dimensional TQFTs via
    string-nets and two-dimensional surgery,} preprint
  arXiv:2206.13262.
  
\bibitem[BBG]{BBG} Beliakova, A., Blanchet, C., Gainutdinov,
  A., \emph{Modified trace is a symmetrised integral}, Selecta
  Math. (N.S.) 27 (2021), Paper No. 31, 51pp.

\bibitem[BCGP]{BCGP14} Blanchet, C., Costantino, F., Geer, N., 
 Patureau-Mirand, B., \textit{Non-Semisimple TQFTs, Reidemeister
    Torsion and Kashaev's Invariants}, Adv. Math. 301 (2016),
  1--78.

\bibitem[BGPR]{BGPR} Blanchet, C., Geer, N., Patureau-Mirand, B., Reshetikhin, N., {\em Holonomy braidings, biquandles and quantum
    invariants of links with $SL_2(\C)$ flat connections},
  Selecta Mathematica (N.S.) 26 (2020),  Paper
  No. 19, 58 pp.

\bibitem[BHMV]{BHMV95} Blanchet, C., Habegger, N., Masbaum, M., Vogel, P.,
  {\em Topological quantum field theories derived from the Kauffman
    bracket}, Topology, 34 (1995), 883--927.

\bibitem[BW]{BW2011} Bonahon F., Wong, H., \textit{Quantum traces for
 representations of surface groups in $SL_2(\C)$},
Geom. Topol. 15 (2011).

\bibitem[BJSS]{BJSS2021} Brochier, A., Jordan, D., Safronov, P., Snyder, N.,
  {\em Invertible braided tensor categories},
  Algebr. Geom. Topol. 21 (2021), 2107--2140.

\bibitem[BLV]{BLV} Brugui{\`e}res, A., Lack, S., Virelizier, A.,
  \emph{Hopf monads on monoidal categories}, Adv. in Math. 227 (2011),
  745--800.

\bibitem[BV1]{BV2}
Brugui{\`e}res, A.,  Virelizier, A.,  \emph{Hopf monads}, Adv. in Math.  215  (2007), 679--733.

\bibitem[BV2]{BV3}
\bysame,  \emph{Quantum double of {H}opf monads and categorical centers},
 Trans. Amer. Math. Soc. 364 (2012),   1225--1279.

\bibitem[Bu]{Bu} Bullock, D.,  \textit{Rings of $SL_2(\C)$-characters and
 the Kauffman bracket skein module}. Comment. Math. Helv. 72 (1997).


\bibitem[CGHP]{CGHP23} Costantino, F., Geer, N., Haïoun, B., Patureau-Mirand, B.,
  {\em Skein (3+1)-TQFTs from non semi-simple ribbon categories},
  preprint, arXiv:2306.03225.

\bibitem[CGP]{CGP14} Costantino, F., Geer, N., Patureau-Mirand, B.,{\em
     Quantum invariants of 3-manifolds via link surgery presentations
     and non-semi-simple categories}, Journal of Topology 7 (2014), 1005--1053.


\bibitem[CGPT]{CGPT20} Costantino, F., Geer, N., Patureau-Mirand, B., Turaev, V., {\em Kuperberg and Turaev-Viro Invariants in Unimodular
  Categories}, Pacific J. Math. 306 (2020), 421--450.


\bibitem[CGuP]{CGuP2021} Costantino, F., S. Gukov, S., Putrov, P., {\em
    Non-semisimple TQFTs and BPS q-series}, arXiv:2107.14238, 2021.


\bibitem[CDGG]{CDGG2021} Creutzig, T., Dimofte, T., Garner, N., Geer, N., 
  {\em A QFT for non-semisimple TQFT}, arXiv:2112.01559, 2021.

\bibitem[D]{D17} De Renzi, M., {\em Non-Semisimple Extended Topological
    Quantum Field Theories}, Mem. Amer. Math. Soc., 2022, 277.

\bibitem[DGGPR]{DRGGPMR2022} De Renzi, M., Gainutdinov, A., Geer, N., Patureau-Mirand, B., Runkel, I., {\em 3-Dimensional TQFTs from
    non-semisimple modular categories}, Selecta Math. (N.S.) 28 (2022), Paper No. 42, 60.

\bibitem[DGP]{DRGPM20} De Renzi, M., Geer, N., Patureau-Mirand, B., {\em
    Nonsemisimple quantum invariants and TQFTs from small and unrolled
    quantum groups}, Algebr. Geom. Topol. 20 (2020), 3377--3422.

\bibitem[DSS]{DSPS20} Douglas, C., Schommer-Pries, C.,  Snyder, N., 
  {\em Dualizable tensor categories}, Amer. Math. Soc. 268 (2020).

\bibitem[EG]{EG} Etingof, P., Gelaki, S., \emph{On finite-dimensional semisimple and cosemisimple Hopf algebras in positive characteristic}, Internat. Math. Res. Notices 1998 (1998), 851--864.

\bibitem[EGNO]{EGNO} Etingof, P., Gelaki, S., Nikshych, D., Ostrik, V., {\em Tensor categories}, volume 205 of {\em
    Mathematical Surveys and Monographs}, American Mathematical
  Society, Providence, RI, 2015.

\bibitem[GKP]{GKP1} 
Geer, N., J. Kujawa, J., Patureau-Mirand, B., \textit{
 Generalized trace and modified dimension functions on ribbon
 categories}, Selecta Math.  17  (2011), 453--504.

\bibitem[GKP2]{GKP2} \bysame, \textit{
 Ambidextrous objects and trace functions for nonsemisimple
 categories},
Proc. Amer. Math. Soc. 141 (2013), 2963--2978.
 
\bibitem[GKP3]{GKP22} \bysame, {\em
    M-traces in (non-unimodular) pivotal categories}
  Algebr. Represent. Theory 25 (2022), 759--776.

\bibitem[GP]{GP2018} Geer N., Patureau-Mirand, B., \textit{ The trace on
 projective representations of quantum groups. } Letters in
 Mathematical Physics 108 (2018), 117--140.

\bibitem[GPV]{GPV13} Geer, N., Patureau-Mirand, B., Virelizier, A., {\em Traces on ideals in pivotal categories}.
Quantum Topol. 9 (2013), 91--124.

\bibitem[GY]{GY2022} Geer N., Young, M., {\em Three dimensional
    topological quantum field theory from $U_q(gl(1|1))$ and $U(1|1)$
    Chern-Simons theory}, preprint, arXiv:2210.04286.

\bibitem[JS]{JS}  Joyal, A.,  Street, R., \emph{The geometry of tensor calculus I}, Adv. in Math. 88 (1991), 55--112.

\bibitem[J]{Juh2018}  Juh\'asz, A., {\em Defining and classifying TQFTs via
    surgery} Quantum Topol. 9 (2018), 229--321.

\bibitem[KV]{KV2019} Kashaev, R.,  Virelizier, A., {\em Generalized
    Kuperberg invariants of 3-manifolds,}
  Algebr. Geom. Topol. 19 (2019), 2575--2624.
  
\bibitem[Ke]{Ker2003}  Kerler, T., {\em Homology TQFT’s and the
    Alexander-Reidemeister invariant of 3-manifolds via Hopf algebras
    and skein theory}, Canad. J. Math. 55 (2003), 766--821.

\bibitem[KL]{KerLy2001} Kerler, T., Lyubashenko, L., {\em Non-semisimple
    topological quantum field theories for 3-manifolds with corners,}
  volume 1765 of Lecture Notes in Mathematics. Springer-Verlag,
  Berlin, 2001.

\bibitem[KTV]{KTV23} Kontsevich, M., Takeda, A., Vlassopoulos, Y., {\em Pre-Calabi-Yau algebras and topological quantum field theories} arxiv.org-2112.14667.

\bibitem[K1]{Ku90} 
Kuperberg, G., {\em Involutory Hopf algebras and
    3-manifold invariants}, Internat. J. Math. 2 (1991), 41--66.

\bibitem[K2]{Ku97} \bysame, {\em Non-involutory Hopf algebras and
    3-manifold invariants}, Duke Math J. 84 (1996), 83--129.

\bibitem[Lu]{Lu2010} Lurie, J., {\em On the Classification of Topological Field Theories}, 2010 preprint arxiv.org/abs/0905.0465.

\bibitem[ML]{ML1}
MacLane, S., \emph{Categories for the working mathematician},
Second edition,   Springer-Verlag, New York, 1998.

\bibitem[Mi]{Mik2015} Mikhaylov, V., {\em Analytic Torsion, 3d Mirror
    Symmetry, And Supergroup Chern-Simons Theories}, 2015 preprint.

\bibitem[Mil]{Mil1965} Milnor, J., {\em Lectures on $h$-cobordism}
  Princeton University Press,1965.

\bibitem[Mu]{Mull} Muller, G.,  \textit{Skein algebras and cluster algebras
 of marked surfaces,} preprint arXiv:1204.0020.

\bibitem[P1]{Pr1991}  Przytycki, J., \textit{Skein modules of 3-manifolds},
  Bull. Polish Acad. Sci. Math. 39 (1991), 91--100.

\bibitem[P2]{Pr1999} \bysame, \textit{Fundamentals of Kauffman
 bracket skein modules}, Kobe J. Math. 16 (1999), 45--66.

\bibitem[Ra]{R2012} Radford, D. E., \emph{Hopf algebras}, World
  Scientific Publishing Co. Pte. Ltd., Hackensack, NJ, (2012), 49,
  xxii+559 pp.

\bibitem[RT]{RT2} Reshetikhin, N., Turaev, V., {\em Invariants of
    3-manifolds via link polynomials and quantum groups},
  Invent. Math. 103 (1991), 547--597.
  
\bibitem[RS]{RS92}  Rozansky, L., Saleur, H., {\em Quantum field theory
    for the multi-variable Alexander-Conway polynomial}, Nuclear
  Phys. B 376 (1992), 461--509.

\bibitem[SL]{SW2020} Schweigert, C., Woike, L., {\em The trace field
    theory of a finite tensor category. } preprint, arXiv:2103.15772.

\bibitem[SS]{SS2021} Shibata, T., Shimizu, K., {\em Modified traces and
    the Nakayama functor}, Algebr. Represent. Theory 26 (2023),  513--551.
    
\bibitem[Sh]{Sh2}
Shimizu, K., \emph{Integrals for Finite Tensor Categories}
Algebr. Represent. Theor. 22 (2019), 459--493.

\bibitem[Si]{Si2000} Sikora, A., \textit{ Skein modules and TQFT. } Knots
 in Hellas '98 (Delphi), 436--439, Ser. Knots Everything, 24, World
 Sci. Publ., River Edge, NJ, 2000.
 
     
\bibitem[Th]{dTh} Thurston, D.,  \textit{Positive basis for surface skein
 algebras}, Proc. Natl. Acad. Sci. USA 111 (2014),
 9725--9732.

\bibitem[T1]{Tu1991b} Turaev, V., \textit{Skein quantization of Poisson
 algebras of loops on surfaces},  Ann. Sci. Sc. Norm. Sup. 24 (1991).

 \bibitem[T2]{Tu} \bysame, {\em Quantum invariants of knots and
     3-manifolds},  de Gruyter Studies in Mathematics, 18. Walter de
   Gruyter \& Co., Berlin, (1994).

\bibitem[TV]{TV} Turaev V., Viro, O., {\em State sum invariants of
    $3$-manifolds and quantum $6j$-symbols}, Topology  31 (1992), 865--902.

\bibitem[TVi]{TVi5}
Turaev, V., Virelizier, A., \emph{Monoidal Categories and Topological Field Theory}, Progress in  Mathematics, 322. Birkh\"auser, Basel, 2017. xii+523 pp.

\bibitem[Vi]{Vir2006} Virelizier, A., {\em Kirby elements and quantum
    invariants}, Proc. London Math. Soc.  93 (2006), 474--514.

\bibitem[W1]{Walker2021}  Walker, K., {\em A universal state sum}, preprint
  arXiv:2104.02101.

\bibitem[W2]{Walker} \bysame, \emph{Going from $n+\epsilon$ to $n+1$
 in non-semisimple oriented TQFTs,} Talk in UQSL seminar, Feb 2022,
 https://canyon23.net/math/talks/np1.pdf
\end{thebibliography}
\end{document}